\documentclass{article}
\RequirePackage{graphicx}
\RequirePackage{xcolor}
\RequirePackage{amscd}
\RequirePackage{framed}
\RequirePackage{bbm}
\RequirePackage{amsmath,amsthm,amssymb}
\usepackage{fancyhdr,a4wide}
\usepackage{comment}
\RequirePackage{zref-xr}
\RequirePackage[numbers]{natbib}
\RequirePackage[dvipdfmx,colorlinks=true,citecolor=blue,urlcolor=blue]{hyperref}

\newcommand{\mc}{\mathcal}
\newcommand{\mbb}{\mathbb}
\newcommand{\mr}{\mathrm}
\newcommand{\argmin}{\mathop{\rm argmin}\limits}

\newcommand{\veca}{\mathbf{a}}

\newcommand{\vecg}{\mathbf{g}}

\newcommand{\vecu}{\mathbf{u}}
\newcommand{\vecv}{\mathbf{v}}
\newcommand{\vecw}{\mathbf{w}}
\newcommand{\vecx}{\mathbf{x}}
\newcommand{\vecX}{\mathbf{X}}

\newcommand{\vecz}{\mathbf{z}}

\newcommand{\vecbeta}{\boldsymbol \beta}
\newcommand{\vectheta}{\boldsymbol \theta}

\newcommand{\vecvarrho}{\boldsymbol \varrho}

\usepackage{bm}

\usepackage{algorithm}
\usepackage{algorithmic}
\renewcommand{\algorithmicrequire}{\textbf{Input:}}
\renewcommand{\algorithmicensure}{\textbf{Output:}}

\theoremstyle{plain}
\numberwithin{equation}{section}
\newtheorem{theorem}{Theorem}[section]
\newtheorem{proposition}{Proposition}[section]
\newtheorem{remark}{Remark}[section]
\newtheorem{lemma}{Lemma}[section]
\newtheorem{corollary}{Corollary}[section]
\newtheorem{definition}{Definition}[section]
\newtheorem{assumption}{Assumption}[section]

\begin{document}
\title{Outlier Robust and Sparse Estimation of Linear Regression Coefficients}
\author{Takeyuki Sasai
\thanks{Department of Statistical Science, The Graduate University for Advanced Studies, SOKENDAI, Tokyo, Japan. Email: sasai@ism.ac.jp}
\and Hironori Fujisawa
\thanks{The Institute of Statistical Mathematics, Tokyo, Japan. 
Department of Statistical Science, The Graduate University for Advanced Studies, SOKENDAI, Tokyo, Japan. 
Center for Advanced Integrated Intelligence Research, RIKEN, Tokyo, Japan. Email:fujisawa@ism.ac.jp}
}
\maketitle
\begin{abstract}
	We consider outlier-robust and sparse estimation of linear regression coefficients, when the covariates and the noises are contaminated by adversarial outliers and noises are sampled from a heavy-tailed distribution. Our results present sharper error bounds under weaker assumptions than prior studies that share similar interests with this study.
	Our analysis relies on some sharp concentration inequalities resulting from  generic chaining.
\end{abstract}

\section{Introduction}
\label{intro}
This study considers outlier-robust and sparse estimation of linear regression coefficients. Consider the following sparse linear regression model:
\begin{align}
\label{model:normal}
	y_i = \vecx_i^\top\vecbeta^*+\xi_i,\quad i=1,\cdots,n,
\end{align}
where $\vecbeta^* \in \mbb{R}^d$ represents the true coefficient vector with $s$ nonzero elements, $\left\{\vecx_i\right\}_{i=1}^n$ denotes a sequence of independent and identically distributed (i.i.d.) random covariate vectors, and $\left\{\xi_i\right\}_{i=1}^n$ denotes a sequence of i.i.d. random noises. Throughout the present paper, we assume $s\geq 1$ and $d/s\geq 3$ for simplicity.
There are many studies about estimation problems of $\vecbeta^*$  \citep{Tib1996Regression, FanLi2001Variable,ZouHas2005Regularization,YuaLin2006Model,CanTao2007Dantzig,BicRItTsy2009Simultaneous,RasWaiYu2010Restricted,Zha2010Nearly,BelCheWan2011Square,DalChe2012Fused,SivBenRav2015Beyond,SuCan2016Slope,FanLiWan2017Estimation,Der2018Improved,BelLecTsy2018Slope,LecMen2018Regularization,FanWanZhu2021Shrinkage}.
Let $\|\vecv\|_2$ denote the $\ell_2$ norm for a vector $\vecv$.
Especially, using the method in \cite{BelLecTsy2018Slope}, with probability at least $1-\delta$, we can construct an estimator $\hat{\vecbeta}$  such that
\begin{align}
\label{ine:bound}
	\|\hat{\vecbeta} -\vecbeta^*\|_2 \lesssim \sqrt{\frac{s\log (d/s)}{n}}+\sqrt{\frac{\log (1/\delta)}{n}},
\end{align}
where $\lesssim$ is the inequality up to an absolute constant factor, when, for simplicity, $\{\vecx_i\}_{i=1}^n$ and $\{\xi_i\}_{i=1}^n$ are the sequences of i.i.d. random covariate vectors sampled from the multivariate Gaussian distribution with $\mbb{E}\vecx_i = 0$ and $\mbb{E}\vecx_i \vecx_i^\top = I$, and random noises sampled from the Gaussian distribution with $\mbb{E}\xi_i = 0$ and $\mbb{E}\xi_i^2 = 1$, respectively. 

This paper considers the situation where $\left\{\vecx_i,y_i\right\}_{i=1}^n$ suffers from  malicious outliers.
We allow an adversary to inject outliers into \eqref{model:normal}, yielding
\begin{align}
\label{model:adv}
	y_i = \vecX_i^\top\vecbeta^*+\xi_i+\sqrt{n}\theta_i,\quad i=1,\cdots,n,
\end{align}
where $\vecX_i = \vecx_i+\vecvarrho_i$ for $i=1,\cdots,n$, and $\{\vecvarrho_i\}_{i=1}^n$ and $\{\theta_i\}_{i=1}^n$ are outliers.
Let $\mc{O}$ be the index set of outliers. We assume the following.
\begin{assumption}
\label{a:outlier}
	Assume that 
	\begin{itemize}
		\item [(i)] the adversary can freely choose the index set $\mc{O}$;
 		\item [(ii)] $\{\vecvarrho_i\}_{i\in \mc{O}}$ and $\{\theta_i\}_{i\in \mc{O}}$ are allowed to be correlated freely with each other and with $\{\vecx_i\}_{i=1}^n$ and $\{\xi_i\}_{i=1}^n$;
 		\item [(iii)] $\vecvarrho_i= (0,\cdots,0)^\top$ and $\theta_i=0$ for $i\in\mc{I}= \{1,2,\cdots,n\} \setminus \mc{O}$.
	\end{itemize}
\end{assumption}
\noindent We note that, under Assumption \ref{a:outlier}, $\{\vecx_i\}_{i\in\mc{I}}$ and $\{\xi_i\}_{i\in\mc{I}}$ are no longer sequences of i.i.d. random variables because $\mc{O}$ is freely chosen by an adversary. 
This type of contamination by outliers is sometimes called the strong contamination, in contrast to the Huber contamination \citep{DiaKan2019Recent}. The Huber contamination is more manageable to tame than strong contamination because outliers of Huber contamination are not correlated to the inliers and do not destroy the independence of the inliers.
We consider a problem to estimate $\vecbeta^*$ in \eqref{model:adv}, and construct a computationally tractable  estimator having a property similar to \eqref{ine:bound}.

We briefly review recent developments in robust and computationally tractable estimators.
\cite{CheGaoRen2018robust} derived  optimal error bounds for the estimation of means and covariance (scatter) matrices  under the presence of outliers and proposed estimators, which achieves the optimal error bounds. However, the estimators are computationally intractable. Subsequently, \cite{LaiRaoVem2016Agnostic} and \cite{DiaKanKanLiMoiSte2019Robust} considered tractable estimators for similar problem settings.
After \cite{LaiRaoVem2016Agnostic} and \cite{DiaKanKanLiMoiSte2019Robust}, many outlier-robust tractable estimators have been developed in \cite{DiaKanKanLiMoiSte2017Being,KotSteSte2018Robust,DiaKamKanLiMoiSte2018Robustly,CheDiaGe2019High,CheDiaGeWoo2019Faster,DonHopLi2019Quantum,PraBalRav2020Robust,DepLec2022Robust,LugMen2021Robust,DalMin2022All,BalDuLiSin2017Computationally, DiaKanKarPriSte2019Outlier,DiaKonSte2019Efficient,BakPra2021Robust,LiuSheLiCar2020High,Chi2020Erm,DiaGouTza2019Distribution,MonGoeDiaSre2020Efficiently, DiaKanMan2020Complexity,DiaKonTzaZaf2020Learning,DiaKanKonTzaZaf2021Efficiently,Tho2020Outlier,NguTra2012Robust,CheCarMan2013Robust, DalTho2019Outlier,CheAraTriJorFlaBar2020Optimal,LecLer2020Robust,GeoLecLer2020Robust,PenJogLoh2020robust,MinMohWan2022Robust,MerGai2023robust}. These studies treated  various estimating problems; e.g., estimation of mean, covariance, linear regression coefficients,  half-spaces, parameters of Gaussian mixture models. Their primary interests are deriving sharp error bounds, deriving information-theoretical optimal error bounds, and reducing computational complexity.
However, there are few studies on combining outlier-robust properties with sparsity \citep{CheCarMan2013Robust,BalDuLiSin2017Computationally,DiaKanKarPriSte2019Outlier,DalTho2019Outlier,LiuSheLiCar2020High,Chi2020Erm,Gao2020Robust,LecLer2020Robust,GeoLecLer2020Robust,Sas2022Robust,MinMohWan2022Robust,MerGai2023robust}.
Especially, \cite{CheCarMan2013Robust}, \cite{BalDuLiSin2017Computationally} and \cite{LiuSheLiCar2020High} dealt with the estimation problem of $\vecbeta^*$ from \eqref{model:adv} under the assumption of  Gaussian and subGaussian tails of $\{\vecx_i,\xi_i\}_{i=1}^n$, with computationally tractable estimators.
Our study can be considered as an extension of these prior studies from two perspectives: 
sharpening the error bound and relaxing the assumption with improved analyses.
The overview of the analyses is described in Section \ref{sec:method}.

The present paper is organized as follows. In Section \ref{sec:2}, we present our main results of the estimation error in rough statements and describe some relationships to previous studies.
In Sections \ref{se:om} and \ref{sec:un}, we describe our estimation methods and main results, without proofs.
In Section \ref{sec:keyL}, we describe the key propositions, lemma and  corollary, without proofs. 
In Section \ref{sec:pl}, we provide  proofs of a part of the propositions in Sections \ref{se:om}--\ref{sec:keyL}.
In the Appendix, we provide the proofs that are omitted  in Sections \ref{se:om}--\ref{sec:pl}. 
In the remainder of this paper, we assume that Assumption \ref{a:outlier} holds for outliers, and for simplicity, $0<\delta\leq 1/4$.

\section{Our results and relationship to previous studies}
\label{sec:2}

	\subsection{Our results}
	\label{preresults}
	In Section \ref{preresults}, we state our results.
	Before that, we introduce some definitions.
	First, we introduce the $\psi_\alpha$-norm and $\mathfrak{L} $-subGaussian random vector, which is an extension of subGaussian random variables in a high dimension.
	\begin{definition}[$\psi_\alpha$-norm]
	\label{d:orlicz}
		For a random variable $f$, let
		\begin{align}
			\|f\|_{\psi_\alpha}:=	\inf\left\{ \eta>0\,:\, \mbb{E}\exp(|f/\eta|^\alpha)\leq 2\right\} < \infty.
		\end{align}	
	\end{definition}

	\begin{definition}[$\mathfrak{L}$-subGaussian random vector]
	\label{d:l}
		A random vector $\vecx \in \mbb{R}^d$ with  mean $\mbb{E}\vecx =0$ is said to be an $\mathfrak{L} $-subGaussian random vector if for any fixed $\vecv \in \mbb{R}^d$,
		\begin{align}
		\label{ine:lsg}
			\|\langle \vecx,\vecv\rangle\|_{\psi_2}\leq \mathfrak{L} \left(\mbb{E}|\langle \vecx,\vecv\rangle|^2\right)^\frac{1}{2},
		\end{align}
		where the norm $\|\cdot\|_{\psi_2}$ is defined in Definition \ref{d:orlicz} and $\mathfrak{L}$ is a numerical constant such that $\mathfrak{L}\geq 1$.
	\end{definition}
  We note that, for example, the multivariate standard Gaussian random vector is 	an $\mathfrak{L}$-subGaussian random vector.
	Second, we introduce the restricted eigenvalue condition for $\Sigma$ \citep{BulGee2011Statistics}.
For a vector $\vecv \in \mbb{R}^d$, define $\vecv |_i$ as the $i$-th element of $\vecv$, and define the $\ell_1$ norm of $\vecv$ as $\|\vecv\|_1$.
For a vector $\vecv\in \mbb{R}^d$ and set $\mc{J}$, define $\vecv_{\mc{J}}$ as a vector such that $\vecv_{\mc{J}}|_i = \vecv|_i$ for $i \in \mc{J}$ and $\vecv|_i = 0$ for $i \notin \mc{J}$.
For a set $\mc{J}$, define $\mc{J}^c$ as a complement set of $\mc{J}$. 
Additionally, for a vector $\vecv$, define the number of non-zero elements of $\vecv$  as $\|\vecv\|_0$.	For a set $\mc{S}$, let $|\mc{S}|$ be the number of the elements of $\mc{S}$. Let $o=|\mc{O}|$.
\begin{definition}[Restricted eigenvalue condition for $\Sigma$]
	\label{RE}
	The covariance matrix $\Sigma$ is said to satisfy the restricted eigenvalue condition $\mr{RE}(s,c_{\mr{RE}},\mathfrak{r})$ with some positive constants $c_{\mr{RE}}, \mathfrak{r}$, if $\|\Sigma^\frac{1}{2}\vecv\|_2\geq \mathfrak{r}\|\vecv\|_2$ for any vector $\vecv \in \mbb{R}^p$ and any set $\mc{J}$ such that $|\mc{J}|\leq s$ and $\|\vecv_{\mc{J}^c}\|_1\leq c_{\mr{RE}}\|\vecv_{\mc{J}}\|_1$.
\end{definition}
For simplicity, we redefine $\mathfrak{r}=\inf_{\vecv \in \mbb{R}^d,\, \|\vecv\|_0\leq s,\,\|\vecv_{\mc{J}^c}\|_1\leq c_{\mr{RE}}\|\vecv_{\mc{J}}\|_1}\frac{\|\Sigma^\frac{1}{2}\vecv\|_2}{\|\vecv\|_2}$. 
Lastly, we introduce the following two quantities related to minimun/maximum eigenvalue:
\begin{align}
	\kappa_{\mr{l}} = \inf_{\|\vecv\|_0\leq s} \frac{\|\Sigma^\frac{1}{2}\vecv\|_2}{\|\vecv\|_2}, \quad \kappa_{\mr{u}} = \sup_{\|\vecv\|_0\leq 2s^2} \frac{\|\Sigma^\frac{1}{2}\vecv\|_2}{\|\vecv\|_2},
\end{align}
We note that, from the definition, we see that the minimum eigenvalue of $\Sigma$ is smaller than $\kappa_{\mr{l}}$.
	Define  $\rho = \max_{i\in \{1,\cdots,d\}}\sqrt{\Sigma_{ii}}$  and the maximum eigenvalue of $\Sigma$ as $\Sigma_{\max}$.
Then, we use the following assumption for the sequence of the covariates $\{\vecx_i\}_{i=1}^n$.
	We make the following assumption on $\{\vecx_i,\xi_i\}_{i=1}^n$:
	\begin{assumption}
	\label{a:intro}
		Assume that
		\begin{itemize}
			\item [(i)]$\left\{\vecx_i \right\}_{i=1}^n$ is a sequence of i.i.d. random vectors sampled from an $\mathfrak{L} $-subGaussian random vector with $\mbb{E}\vecx_i = 0$,  $\mbb{E}\vecx_i^\top \vecx_i = \Sigma$, $\rho\geq 1$ and  $\kappa_{\mr{l}}>0$.
			Assume that $\Sigma$ satisfies $\mr{RE}(s,c_{\mr{RE}},\mathfrak{r})$ with $c_{\mr{RE}}>1$ and $\mathfrak{r}\leq 1$;
			\item [(ii)]$\left\{\xi_i\right\}_{i=1}^n$ is a sequence of i.i.d. random variables with  $\mbb{E}\xi_i^2 \leq \sigma^2$;
			\item [(iii)] $\left\{\xi_i\right\}_{i=1}^n$ and $\left\{\vecx_i \right\}_{i=1}^n$ are independent.
		\end{itemize}
	\end{assumption}
	Define
	\begin{align}
		R_{\mr{lasso}} &= \sigma\left(\frac{\rho}{\mathfrak{r}} \sqrt{\frac{s\log(d/s)}{n}}+\sqrt{\frac{\log(1/\delta)}{n}} \right),\nonumber\\
		R_{\mr{outlier}} &=\sigma\frac{\kappa_{\mr{u}}}{\kappa_{\mr{l}}}\left(\sqrt{\frac{o}{n}}\sqrt{s \sqrt{\frac{\log(d/s)}{n}}+\sqrt{\frac{\log(1/\delta)}{n}}}+\frac{o}{n}\sqrt{\log \frac{n}{o}}\right),\nonumber\\
		R_{\mr{outlier}}'&= \sigma\sqrt{\frac{o}{n}} \times \left(\frac{\kappa_{\mr{u}}}{\kappa_{\mr{l}}}\sqrt{s \sqrt{\frac{\log(d/s)}{n}}+\sqrt{\frac{\log(1/\delta)}{n}}}+\frac{\Sigma_{\max}}{\kappa_{\mr{l}}}\right).
	\end{align}
	Our first result is as follows 	(for a precise statement, see Theorem \ref{t:main} in Section \ref{sec:results}).
	\begin{theorem}
	\label{theoreminformal1}
		Suppose that Assumption \ref{a:intro} holds, and $\Sigma$ is allowed to be used in estimation.
		Define $C_{c_{\mr{RE}}}$ is a  constant that depends on $c_{\mr{RE}}$
		Then, with probability at least $1-4\delta$, we can construct an estimator $\hat{\vecbeta}$ such that 
		\begin{align}
			\label{rec1-sigma}
				\|\Sigma^\frac{1}{2}(\hat{\vecbeta} -\vecbeta^*)\|_2 &\leq C_{c_{\mr{RE}}}\mathfrak{L}^3\left( 	R_{\mr{lasso}}+ R_{\mr{outlier}} \right),
		\end{align}
		with a computationally tractable method, when $R_{\mr{lasso}} $ and $R_{\mr{outlier}}$  are sufficiently small.
	\end{theorem}
	Our second result is as follows (for a precise statement, see Theorem \ref{t:main2} in Section \ref{sec:un}).
	\begin{theorem}
		\label{theoreminformal2}
			Suppose that Assumption \ref{a:intro} holds, and $\Sigma$ is not allowed to be used in estimation.
			Define $C_{c_{\mr{RE}}}'$ is a constant that depends on $c_{\mr{RE}}$
			Then, with probability at least $1-3\delta$, we can construct an estimator $\hat{\vecbeta}$ such that 
			\begin{align}
				\label{rec2-sigma}
					\|\Sigma^\frac{1}{2}(\hat{\vecbeta} -\vecbeta^*)\|_2 &\leq C_{c_{\mr{RE}}}'\mathfrak{L}^3\left( 	R_{\mr{lasso}}+ R'_{\mr{outlier}} \right),
			\end{align}
			with a computationally tractable method, when $R_{\mr{lasso}} $ and $R'_{\mr{outlier}}$  are sufficiently small.
		\end{theorem}
		From \eqref{rec1-sigma} and \eqref{rec2-sigma}, we see that the error bounds of our estimators match the ones of the normal lasso, up to $\mathfrak{L}$ and numerical constant factors, when there is no outliers because $R_\mr{lasso}$ is equivalent to the upper bound of \eqref{ine:bound} up to constant factors and $R_{\mr{outlier}}=0$ when there is no outliers.
		We see that $R_{\mr{outlier}}'$ is larger than $ R_{\mr{outlier}}$.
		From the fact that $R'_{\mr{outlier}}$ is larger than $R_{\mr{outlier}}$, we see that there is a deterioration in error bounds, as a trade-off for not utilizing $\Sigma$ in estimation.
	\begin{remark}
		Conditions (ii)  in Assumption \ref{a:intro} are provided to make the results simple, and condition (ii) can be weakens to a  tail probability condition. For detail, see Assumption \ref{a:noise}.
	\end{remark}
In Theorems \ref{t:main} and \ref{t:main2}, we explicitly describe the relationship between the tuning parameters used for estimation and the obtained error bounds in  Theorems \ref{theoreminformal1} and \ref{theoreminformal2}, respectively
Additionally, in Theorems \ref{t:main} and \ref{t:main2}, we derive error bounds not only about $\|\Sigma^\frac{1}{2}(\cdot)\|_2$ but also $\|\cdot\|_2$ and $\|\cdot\|_1$.

	\subsection{Relationship to previous studies}
	\label{sec:rel}
	As we mentioned in Section \ref{intro},  \cite{CheCarMan2013Robust,BalDuLiSin2017Computationally,LiuSheLiCar2020High} dealt with the estimation problem of $\vecbeta^*$ from \eqref{model:adv} under Assumption \ref{a:outlier} with computationally tractable estimators.

	We note that \cite{LecLer2020Robust,GeoLecLer2020Robust,Gao2020Robust} dealt with the estimation problem of $\vecbeta^*$ from \eqref{model:adv} under a stronger assumption about outlier than Assumption \ref{a:outlier}, with computationally  intractable estimators deriving sharp error bounds. Additionally, we note that \cite{Sas2022Robust} dealt with a situation  where $\{\vecx_i\}_{i=1}^n$  is a sequence of i.i.d. random vectors sampled from a heavy-tailed distribution, and \cite{MerGai2023robust} dealt with more challenging situation weakening the assumptions for covariates than that of \cite{Sas2022Robust}.
	Their error bounds are looser than the results of the present paper because the weak assumptions restrict the techniques available. Therefore, these papers \citep{LecLer2020Robust,GeoLecLer2020Robust,Gao2020Robust,Sas2022Robust,MerGai2023robust} treat computationally intractable estimators or suppose more weaker assumptions than our method, and hence we do not mention such papers further because the interests of such papers are different from that of our paper.

	Therefore, we mainly discuss the results of \cite{CheCarMan2013Robust,BalDuLiSin2017Computationally,LiuSheLiCar2020High} in the remainder of Section \ref{sec:rel}.

		\subsubsection{Case where $\Sigma$ is  allowed to be used in estimation}
		The result of  \cite{BalDuLiSin2017Computationally} and part of the result of \cite{LiuSheLiCar2020High} use $\Sigma$ in their estimation.
		\cite{BalDuLiSin2017Computationally} and \cite{LiuSheLiCar2020High} considered situations where the covariate vectors are sampled from the standard
		multivariate Gaussian distribution and the noises are sampled from a Gaussian distribution with mean $0$ and variance $\sigma^2$.
		In contrast, our method works well for the case where the covariate vectors are sampled from an $\mathfrak{L}$-subGaussian random variable with covariance which satisfies the restricted eigenvalue condition and the noises are sampled from a heavy-tailed distribution.  
		\cite{BalDuLiSin2017Computationally} and \cite{LiuSheLiCar2020High} only considered the case where $o/n$ is a sufficiently small constant. Let $o/n = \mathfrak{e}$. The $\ell_2$-norm error bound  of \cite{BalDuLiSin2017Computationally}  is $\lesssim_{\sigma} (\sqrt{1+\|\vecbeta^*\|_2^2} \frac{o}{n} \log^2 \frac{n}{o})$ when $\frac{s^2}{\mathfrak{e}^2} \log d +\frac{s^2}{\mathfrak{e}^2} \log(1/\delta)\lesssim_{\sigma} n$, where $\lesssim_{\sigma}$ is the inequality up to an absolute constant factor and the standard deviation of the random noise $\sigma$.
		The $\ell_2$-norm error bound of  \cite{LiuSheLiCar2020High} is $\lesssim \sigma   ( \frac{o}{n} \log \frac{n}{o})$ when $\left(\frac{s^2}{\mathfrak{e}^2} \log (dT) +\frac{s^2}{\mathfrak{e}^2} \log (1/\delta)  \right)\times T \lesssim n$, with $ \log\left(\frac{\|\vecbeta^*\|_2}{\mathfrak{e}\sigma}\right) \lesssim  T$. Under the same situation, the $\ell_2$-norm error bound of our result becomes $\lesssim \sigma \frac{o}{n}\sqrt{ \log \frac{n}{o}}$ when $ \frac{s^2}{\mathfrak{e}^2  } \log(d/s)+ \frac{1}{\mathfrak{e}^2}\log(1/\delta)\lesssim n$.
		We see that  our result does not depend on $\vecbeta^*$ and the error bound is sharper than the ones of \cite{BalDuLiSin2017Computationally} and \cite{LiuSheLiCar2020High}. Additionally, our sample complexity is smaller  than the ones of \cite{BalDuLiSin2017Computationally} and \cite{LiuSheLiCar2020High} because, in our sample complexity, the term such that $\frac{s^2}{\mathfrak{e}} \times \log(1/\delta)$ does not appear.
	
		We consider the optimality of the error bound in \eqref{rec1-sigma}. From Theorem D.3 of \cite{CheAraTriJorFlaBar2020Optimal}, we see that the error bound can not avoid a term such that $\mathit{constant} \times \sigma \frac{o}{n}\sqrt{\log\frac{n}{o}}$ even when $d=1$. Detailed investigation of the influence of $\Sigma$ in high dimension on information theoretical estimation limit is a task for future research.

		When there is no outlier ($o=0$),  our error bound coincides with the one of the normal lasso, up to numerical and $\mathfrak{L}$ factors.
		The results in \cite{BalDuLiSin2017Computationally} and \cite{LiuSheLiCar2020High} do not have this property.

		\subsubsection{Case where $\Sigma$ is not allowed to be used in estimation}
		\label{sec:u}
		\cite{CheCarMan2013Robust}, and a part of the result of \cite{LiuSheLiCar2020High} gives error bounds with tractable methods and do not require $\Sigma$ in estimation.
		However, the method in \cite{LiuSheLiCar2020High} assumes a sparse structure of $\Sigma$, and the sample complexity depends on not only $s^2$ but also the sparsity of $\Sigma$. \cite{CheCarMan2013Robust} propose some methods, however, the term in their error bounds containing $o$ depends on $\log d$ and $s$ even when $s^2\log(d/s) \lesssim n$.

		When there is no outlier ($o=0$), similarly to the case where $\Sigma$ is  allowed to be used in estimation, our error bound coincides with the one of the normal lasso, up to numerical and $\mathfrak{L}$ factor.
		The results in \cite{LiuSheLiCar2020High} and \cite{CheCarMan2013Robust} do not have this property.

		\subsubsection{Remaining problem}
		Our estimator and estimators in \cite{BalDuLiSin2017Computationally} and \cite{LiuSheLiCar2020High} require $n$ to be proportional to $s^2$, which is not needed to derive \eqref{ine:bound} from \ref{model:normal}.
		Similar phenomena can be observed in \cite{WanBerSam2016Statistical,FanWanZhu2021Shrinkage,LiuSheLiCar2020High,BalDuLiSin2017Computationally, DiaKanKarPriSte2019Outlier}.
    Some relationships between computational tractability and similar quadratic dependencies are unraveled  \citep{WanBerSam2016Statistical,DiaKonSte2019Efficient,DiaKanSte2017Statistical}.
		We leave the analysis in our situation for future work.
		We note that \cite{NguTra2012Robust, DalTho2019Outlier,Chi2020Erm,Tho2020Outlier,MinMohWan2022Robust} considered a simpler case in 	which only noise is contaminated by outliers ($\vecvarrho_i=0$ for any $i \in \{1,\cdots,d\}$), with computationally tractable estimators, and  the error bounds and sample complexities depend on $s$, not $s^2$.
		In actual applications, it is necessary to consider an appropriate method after carefully clarifying the nature of  data.

\section{Method and result}
\label{se:om}
Assume that $\vecx$ is a random vector drawn from the same distribution of $\{\vecx_i\}_{i=1}^n$.
Hereafter, we often use the following simply notations to express error orders: 
\begin{align}
	r_{d,s} = \sqrt{\frac{\log (d/s)}{n}},\quad r_\delta = \sqrt{\frac{\log (1/\delta)}{n}}, \quad r_o=\frac{o}{n}\sqrt{\log \frac{n}{o}},\quad r_o' = \frac{o}{n}\log \frac{n}{o}.
\end{align}

	\subsection{Some properties of $\mathfrak{L}$-subGaussian random vector}
	\label{sec:Lsgprop}
	We show some additional properties of $\mathfrak{L}$-subGaussian random vector $\vecx$.	We note that, from \eqref{ine:lsg}, we have 
	\begin{align}
	\label{ine:lsgtmp}
		\|\langle \vecx,\vecv\rangle\|_{\psi_2}\leq \mathfrak{L} \left(\mbb{E}|\langle \vecx,\vecv\rangle|^2\right)^\frac{1}{2}\leq \mathfrak{L} \|\Sigma^\frac{1}{2}\vecv\|_2,
	\end{align}
	 and from (2.14) - (2.16) of \cite{Ver2018High}, for any $\vecv \in \mbb{R}^d$ and $t\geq 0$, we have 
	\begin{align}
	\label{cl}
		\|\vecv^\top \vecx\|_{L_p} &\left[:= \left\{\mbb{E}|\vecv^\top \vecx|^p \right\}^\frac{1}{p}\right]\leq c_{\mathfrak{L}} \sqrt{p}\|\vecv^\top \vecx\|_{\psi_2}\leq c_{\mathfrak{L}} \sqrt{p}\mathfrak{L} \|\Sigma^\frac{1}{2}\vecv\|_2,\\
	\label{ine:lsg2}
		\mbb{E}\exp(\vecv^\top \vecx)& \leq \exp(c_{\mathfrak{L}}^2  \mathfrak{L}^2 \|\Sigma^\frac{1}{2}\vecv\|_2^2 ),\\
	\label{ine:lsg2-2}
		\mbb{E}\exp\left(\frac{(\vecv^\top \vecx)^2}{c_{\mathfrak{L}}^2\mathfrak{L}^2\|\Sigma^\frac{1}{2}\vecv\|_2^2}\right)& \leq 2,\\
	\label{ine:lsg3}
		\mbb{P}\left(|\vecv^\top \vecx|>t\right) &\leq 2 \exp \left(-\frac{t^2}{c_{\mathfrak{L}}^2\mathfrak{L}^2 \|\Sigma^\frac{1}{2}\vecv\|_2^2}\right),
	\end{align}
	where $c_{\mathfrak{L}}$ is a numerical constant. Define $L = \mathfrak{L}\times \max\{1,c_{\mathfrak{L}}\} $.

	\subsection{Method}
 \label{sec:method}
	To estimate $\vecbeta^*$ in \eqref{model:adv}, we propose the outlier-robust and sparse estimation (Algorithm \ref{ourmethod}).
	Algorithm \ref{ourmethod} is similar to the ones used in  \cite{PenJogLoh2020robust} and \cite{Sas2022Robust}.
	However, \cite{PenJogLoh2020robust} considered non-sparse case, and \cite{Sas2022Robust} considered heavy-tailed covariates.
	Therefore, to consider the sparsity of $\vecbeta^*$ or to derive a sharper error bound than that in \cite{Sas2022Robust} taking advantage $\mathfrak{L}$-subGaussian assumption, our analysis is more involved than \cite{PenJogLoh2020robust} and \cite{Sas2022Robust}.
	Concretely, unlike the previous studies,  extensions of Hanson--Wright inequalities, which appear later in Proposition \ref{p:cwpre} and Corollary \ref{c:cwpre} proved via generic chaining, play  important roles. 
	We will analyze an $\ell_1$-penalized Huber loss in step 3 in Algorithm \ref{ourmethod}.
	Our analysis of the $\ell_1$-penalized Huber loss is similar to the ones in \cite{AlqCotLec2019Estimation} and \cite{LecMen2018Regularization}, however, the analyses of \cite{AlqCotLec2019Estimation} and \cite{LecMen2018Regularization} are mainly interested in the case $\mbb{E}\vecx \vecx^\top = I$, and not very effective for a more general covariance. We modify their analysis and our analysis is effective for a more general covariance. In particular, Proposition \ref{p:main} is important and a modified analysis method is described in Appendices \ref{sec:mainproposition-pre} and \ref{sec:mainproof}.

	\begin{algorithm}[h]
		\caption{Outlier-robust and sparse estimation}
		\begin{algorithmic}[1]
		\label{ourmethod}
			\renewcommand{\algorithmicrequire}{\textbf{Input:}}
			\renewcommand{\algorithmicensure}{\textbf{Output:}}
			\REQUIRE $\left\{y_i,\vecX_i\right\}_{i=1}^n$, $\Sigma(=\mbb{E}\vecx \vecx^\top)$ and  tuning parameters $\tau_{\rm cut},\varepsilon,r_1,r_2,\lambda_o,\lambda_s$
			\ENSURE $\hat{\vecbeta}$
			\STATE $\left\{ \hat{w}_i\right\}_{i=1}^n \leftarrow \text{WEIGHT}(\{\vecX_i\}_{i=1}^n ,\tau_{\rm cut},\varepsilon,r_1,r_2,\Sigma)$
			\STATE $\{ \hat{w}_i'\}_{i=1}^n \leftarrow \text{TRUNCATION}(\left\{ \hat{w}_i\right\}_{i=1}^n )$
			\STATE $\hat{\vecbeta} \leftarrow \text{WEIGHTED-PENALIZED-HUBER-REGRESSION}\left(\{y_i,\vecX_i\}_{i=1}^n, \,\{\hat{w}'_i\}_{i=1}^n,\,\lambda_o,\,\lambda_s\right)$
		\end{algorithmic} 
	\end{algorithm}
	Here we give simple explanations of the output steps. 
	The details are provided in Sections \ref{sec:cw}, \ref{sec:t}, and \ref{sec:whr}.
	The first step produces the weights $\{\hat{w}_i\}_{i=1}^n$ reducing adverse effects of covariate outliers. 
	The second step is the truncation of the provided weights to zero or $1/n$, say $\{\hat{w}_i'\}_{i=1}^n$. The third step is the $\ell_1$-penalized Huber regression based on the weighted errors using the truncated weights $\{\hat{w}_i'\}_{i=1}^n$. The $\ell_1$-penalization addresses the high dimensional setting, and the Huber regression weakens the adverse effects of the response outliers. 

		\subsubsection{WEIGHT}
		\label{sec:cw}
		For a matrix $M= (m_{ij})_{1\leq i\leq d_1,1\leq j\leq d_2} \in \mbb{R}^{d_1}  \times \mbb{R}^{d_2}$, we define
		\begin{align}
			\|M\|_1 = \sum_{i=1}^{d_1}\sum_{j=1}^{d_2}|m_{ij}|.
		\end{align}
		For a symmetric matrix $M$, we write $M\succeq 0$ if $M$ is positive semi-definite.
		Define $\mr{Tr}(M)$ for a squared matrix $M$ as the trace of $M$. Define the following convex set:
		\begin{align}
			\mathfrak{M}_{r_1,r_2,d}^{\ell_1,\mr{Tr}} &= \left\{M\in {\cal S}(d)\,\mid \,  \|M\|_1 \leq r_1^2,\,\mr{Tr}(M) \leq r^2_2,\,M\succeq 0\right\},
		\end{align}
 		where ${\cal S}(d)$ is a set of symmetric matrices on $\mbb{R}^{d} \times \mbb{R}^d$. 
		For a vector $\vecv$, we define the $\ell_\infty$ norm of $\vecv$ as $\|\vecv\|_\infty$ and 
		define the probability simplex $\Delta^{n-1}(\varepsilon)$ with $0<\varepsilon<1$ as follows:
		\begin{align}
			\Delta^{n-1}(\varepsilon) = \left\{\vecw \in [0,1]^n \,\mid \,  \sum_{i=1}^nw_i =1, \quad \|\vecw\|_\infty\leq \frac{1}{n(1-\varepsilon)}\right\}.
		\end{align}
		The first step of Algorithm \ref{ourmethod}, WEIGHT, is stated as follows.
		\begin{algorithm}[H]
		\caption{WEIGHT}
		\label{alg:cw0}
			\begin{algorithmic}
				\REQUIRE{data $\{\vecX_i \}_{i=1}^n$, tuning parameters $\tau_{\rm cut}, \varepsilon,r_1,\,r_2$.}
				\ENSURE{weight estimate $\hat{\vecw} = \{\hat{w}_1,\cdots,\hat{w}_n\}$.}\\
				Let $\hat{\vecw}$ be the solution to 
				\begin{align}
				\label{cw}
					\min_{\vecw \in \Delta^{n-1}(\varepsilon)} \max_{M\in \mathfrak{M}_{r_1,r_2,d}^{\ell_1,\mr{Tr}} }\sum_{i=1}^n w_i \langle\vecX_i \vecX_i^\top-\Sigma,M\rangle
				\end{align}
				{\bf if} {the optimal value of \eqref{cw} $\leq \tau_{\rm cut} $}\\
				\ \ \ \ \ {\bf return} {$\hat{\vecw}$}\\
				{\bf else} \\
				\ \ \ \ \ {\bf return} {$fail$}\\
			\end{algorithmic}
		\end{algorithm}

		Algorithm \ref{alg:cw0} is a special case of Algorithm 3 of \cite{BalDuLiSin2017Computationally}. Therefore, as in Algorithm 3 of \cite{BalDuLiSin2017Computationally}, Algorithm \ref{alg:cw0} can  also be computed efficiently.
		An intuitive meaning of \eqref{cw} is given in Section \ref{sec:int}. 
		For the detail of the value of $\tau_{\mr{cut}}$ and its validity, see Theorem \ref{t:main} and Proposition \ref{p:cwpre}, respectively.

		\subsubsection{TRUNCATION}
		\label{sec:t}
		The second step in Algorithm \ref{ourmethod} is the discretized truncation of $\left\{ \hat{w}_i\right\}_{i=1}^n$, say $\{\hat{w}_i'\}_{i=1}^n$, as in Algorithm \ref{alg:t}. The discretized truncation, which makes it easy to analyze the estimator. 
		\begin{algorithm}[H]
		\caption{TRUNCATION}
		\label{alg:t}
			\begin{algorithmic}
				\REQUIRE{weight vector $\hat{\vecw} = \{\hat{w}_i\}_{i=1}^n$.}
				\ENSURE{truncated weight vector $\hat{\vecw}'= \{\hat{w}'_i\}_{i=1}^n$.}\\
				{\bf For} {$i=1:n$}\\
				\ \ \ \ \ {\bf if} $\hat{w}_i \geq \frac{1}{2n}$\\
				\ \ \ \ \ \ \ \ \ \ $\hat{w}'_i = \frac{1}{n}$\\
				\ \ \ \ \ {\bf else}\\
				\ \ \ \ \ \ \ \ \ \ $\hat{w}'_i = 0$\\
	 			{\bf return}	$\hat{\vecw}'$. 
		\end{algorithmic}
		\end{algorithm}

		\subsubsection{WEIGHTED-PENALIZED-HUBER-REGRESSION}
		\label{sec:whr}
		The Huber loss function $H(t)$ is defined as follows:
		\begin{align}
			H(t) = \begin{cases}
			|t| -1/2 & (|t| > 1) \\
			t^2/2 & (|t| \leq 1)
			\end{cases},
		\end{align}
		and let
		\begin{align}
			h(t) =	\frac{d}{dt} H(t) = \begin{cases}
			\mr{sgn}(t)\quad &(|t| >1)\\
			t\quad &(|t| \leq 1)
			\end{cases}.
		\end{align}
		We consider the $\ell_1$-penalized Huber regression with the weighted samples $\{\hat{w}_i'y_i,\hat{w}_i'X_i\}_{i=1}^n$ in Algorithm \ref{alg:WH}. This is the third step in Algorithm \ref{ourmethod}. 
		\begin{algorithm}[H]
		\caption{WEIGHTED-PENALIZED-HUBER-REGRESSION}
		\label{alg:WH}
			\begin{algorithmic}
				\REQUIRE{data $\left\{y_i,\vecX_i\right\}_{i=1}^n$, truncated weight vector $\hat{\vecw}' = \{\hat{w}'_i\}_{i=1}^n$ and tuning parameters $\lambda_o, \lambda_s$.}
				\ENSURE{estimator $\hat{\vecbeta}$.}\\
				\begin{align}
					\label{hl}
					\hat{\vecbeta}=\argmin_{\vecbeta \in \mbb{R}^d} \sum_{i=1}^n \lambda_o^2 H\left(n\hat{w}_i'\frac{y_i-\vecX_i^\top\vecbeta}{\lambda_o\sqrt{n}}\right)+\lambda_s\|\vecbeta\|_1,
				\end{align}
	 			{\bf return}	$\hat{\vecbeta}$. 
			\end{algorithmic}
		\end{algorithm}
		Several studies, such as \cite{NguTra2012Robust,SheOwe2011Outlier,DalTho2019Outlier,SunZhoFan2020Adaptive,CheZho2020Robust,Chi2020Erm, PenJogLoh2020robust,Sas2022Robust}, have suggested that the Huber loss is effective for linear regression under heavy-tailed noise or the existence of outliers. 

		Lastly, we introduce the assumption on $\{\xi_i\}_{i=1}^n$:
		\begin{assumption}
		\label{a:noise}
			\begin{itemize}
				\item[(i)] $\{\xi_i\}_{i=1}^n$ is a sequence of i.i.d. random variables such that 
					\begin{align}
					\label{ine:lambda}
						\mbb{P}\left(\frac{\xi_i}{\lambda_o\sqrt{n}} \geq \frac{1}{2}\right) \leq \frac{1}{144 L^4 };
					\end{align}
					\item[(ii)] $\mbb{E}h\left(\frac{\xi_i}{\lambda_o\sqrt{n}}\right) \times \vecx_i=0$.
				\end{itemize}
		\end{assumption}
		\begin{remark}
		\label{r:lambda}
			For example, when $\mbb{E}\xi_i^2 \leq \sigma^2$, from Markov's inequality, we have
			\begin{align}
					\mbb{P}\left(\frac{\xi_i}{\lambda_o\sqrt{n}} \geq \frac{1}{2}\right) \leq \frac{4}{\lambda_o^2 n} \mbb{E}\xi_i^2 \leq \frac{4\sigma^2}{\lambda_o^2 n}.
			\end{align}
			Therefore, to  satisfy \eqref{ine:lambda}, it is sufficient to  set 
			\begin{align}
				\label{ine:lambdacond}
					24  L^2 \sigma\leq \lambda_o \sqrt{n}.
			\end{align}
			In this case, Condition (i) in Assumption \ref{a:noise} is weaker than  $\mbb{E}\xi_i^2\leq \sigma^2$. 
		\end{remark}
		\begin{remark}
			Condition (ii) in Assumption \ref{a:noise} is weaker than the independence between $\{\xi_i\}_{i=1}^n$ and $\{\vecx_i\}_{i=1}^n$.
		\end{remark}

		\subsubsection{An intuitive meaning of Algorithm \ref{alg:cw0}}
		\label{sec:int}
		For $l = 0, 1,2$, we define $d$-dimensional $\ell_l$-ball  with radius $a$ as 
		$a\mbb{B}^d_l = \{\vecv \in \mbb{R}^d \,\mid\, \|\vecv\|_l\leq a \}$, and for $l = \Sigma$, we define $a\mbb{B}^d_\Sigma = \{\vecv \in \mbb{R}^d \,\mid\, \|\Sigma^\frac{1}{2}\vecv\|_2\leq a \}$.
		We explain an intuitive meaning of WEIGHT.
		\cite{Tho2020Outlier} and \cite{MinMohWan2022Robust} considered the case where covariates are not contaminated by outliers, namely $\varrho_i = 0$ for all $i \in  \{1,\cdots ,n\}$. Proposition 4 of \cite{Tho2020Outlier} or Theorem 2.4 of \cite{MinMohWan2022Robust} played an important role in theoretical analysis.
			Proposition 4 of \cite{Tho2020Outlier} or Theorem 2.4 of \cite{MinMohWan2022Robust} implies that, for any $\vecv \in r_1 \mbb{B}^d_1 \cap r_\Sigma \mbb{B}^d_\Sigma$ with appropriate values of $r_1,\,r_\Sigma$, and for any $(u_1,\cdots,u_n)\in \mbb{R}^n$ such that $u_i= 0$ for $i \in \mc{I}$ and $\|\vecu\|_\infty \leq 2$,  we have
		\begin{align}
		\label{ine:tho}
			\mbb{P}\left\{\sum_{i=1}^n\vecv^\top \vecx_i u_i\lesssim \rho r_{d,s}r_1+(\sqrt{s}r_{d,s}+r_\delta+r_o)r_\Sigma\right\}\geq 1-\delta.
		\end{align}
		We expect a property similar to \eqref{ine:tho} even when the covariates are contaminated by outliers, more precisely, when $\{\vecx_i\}_{i=1}^n$ is replaced by $\{\vecX_i\}_{i=1}^n$. In this case, the result \eqref{ine:tho} does not hold as it is, but we can have a similar property for a weighted type 
		$\sum_{i=1}^n\vecv^\top w_i\vecX_i u_i$  with a devised weight defined in Algorithm \ref{alg:cw0}. To obtain a property similar to \eqref{ine:tho}, $\sum_{i=1}^n\vecv^\top w_i\vecX_i u_i$  must be sufficiently small. Suppose $ \vecw \in \Delta^{n-1}(\varepsilon)$. 
		We see that, for any $\vecv \in r_1 \mbb{B}^d_1 \cap r_2 \mbb{B}^d_2$, from H{\"o}lder's inequality, 
		\begin{align}
		\label{ine:th3}
			\left(\sum_{i=1}^n\vecv^\top w_i \vecX_i u_i\right)^2&\leq \sum_{i\in \mc{O}} w_i u_i^2 \sum_{i \in \mc{O}} w_i (\vecX_i^\top \vecv)^2\nonumber\\
			&\leq \frac{4}{1-\varepsilon}\frac{o}{n}\sum_{i \in \mc{O}} w_i \langle \vecX_i \vecX_i^\top, \vecv \vecv^\top \rangle\nonumber \\
			&\leq  \frac{4}{1-\varepsilon}\left(\frac{o}{n}\sum_{i \in \mc{O}} w_i \langle \vecX_i \vecX_i^\top-\Sigma, \vecv \vecv^\top \rangle+\frac{o}{n}\sum_{i \in \mc{O}} w_i \langle \Sigma, \vecv \vecv^\top \rangle\right)\nonumber\\
			&\stackrel{(a)}{\leq}  \frac{4}{1-\varepsilon}\left(\frac{o}{n}\sum_{i \in \mc{O}} w_i \langle \vecX_i \vecX_i^\top-\Sigma, \vecv \vecv^\top \rangle+\frac{o^2}{n^2(1-\varepsilon)} \langle \Sigma, \vecv \vecv^\top \rangle\right),
		\end{align}
		where (a) follows from $w_i\leq 1/(n(1-\varepsilon))$.
		Evaluation of   $\sum_{i \in \mc{O}} w_i \langle \vecX_i \vecX_i^\top-\Sigma, \vecv \vecv^\top \rangle$ is difficult because it contains outlier.
		However, we see that it is sufficient to evaluate
		\begin{align}
		\label{ine:th4}
			\sum_{i =1}^n w_i \langle \vecX_i \vecX_i^\top-\Sigma, \vecv \vecv^\top \rangle
		\end{align}
		in the proof of our results.
		Therefore, to make $\sum_{i=1}^n \vecv^\top  w_i \vecX_i u_i$ sufficiently small, we want to minimize \eqref{ine:th4} in $\vecw \in \Delta^{n-1}(\varepsilon)$ for any $\vecv \in 	r_1 \mbb{B}^d_1 \cap r_2 \mbb{B}^d_2$, in other words, we want to consider
		\begin{align}
		\label{cw-norelax}
			\min_{\vecw \in \Delta^{n-1}(\varepsilon)} \max_{\vecv \in 	r_1 \mbb{B}^d_1 \cap r_2 \mbb{B}^d_2}\sum_{i=1}^n w_i \langle\vecX_i \vecX_i^\top-\Sigma,\vecv \vecv^\top \rangle.
		\end{align}
		A convex relaxation of \eqref{cw-norelax} is \eqref{cw}, which is an essential part in Algorithm 2. We can see in Proposition \ref{p:cwpre} that the optimization \eqref{cw} is enough to have a property similar to \eqref{ine:tho} for a weighted type $\sum_{i=1}^n\vecv^\top w_i\vecX_i u_i$. 

		\subsubsection{Approaches of the  the previous studies}
		In this section, we explain the approaches of the previous studies \cite{BalDuLiSin2017Computationally} and \cite{LiuSheLiCar2020High}.
		To estimate $\vecbeta^*$ in \eqref{model:adv},
		\cite{BalDuLiSin2017Computationally} used sparse and outlier-robust mean estimation on 
		$y_i \vecX_i$, directly. 
		\cite{BalDuLiSin2017Computationally} assumes that $\{\vecx_i\}_{i=1}^n$ and $\{\xi_i\}_{i=1}^n$ are the sequences  of i.i.d. random vectors  sampled from  the multivariate standard Gaussian distribution and random variables sampled from the standard Gaussian, respectively, and $\{\vecx_i\}_{i=1}^n$ and $\{\xi_i\}_{i=1}^n$ are independent. Then, we have
		\begin{align}
			\mbb{E}y_i\vecx_i = \vecbeta^*,\quad \mbb{V} (y_i \vecx_i) (y_i \vecx_i)^\top =(\|\vecbeta^*\|_2^2+1)I + \vecbeta^* \vecbeta^{* \top},
		\end{align}
		where we use Isserlis' theorem, which is a formula for the multivariate Gaussian distribution.
		From Remark 2.2 of \cite{CheGaoRen2018robust}, we see that the error bound of any outlier-robust  estimator for mean vector can not avoid the effect of the operator norm of the covariance $\mbb{V} (y_i \vecx_i) (y_i \vecx_i)^\top$. Consequently, the approach of \cite{BalDuLiSin2017Computationally} can not remove $\|\vecbeta^*\|_2$ from their error bound.
		Additionally, we note that \cite{LiuSheLiCar2020High} proposed a gradient method  that repeatedly updates the estimator of $\vecbeta^*$.
		Define the $t$-step of the update of the estimator of $\vecbeta^*$ as $\vecbeta^t$. The gradient for the next update is 
		based on the result of an outlier-robust estimation of $\vecX_i(\vecX_i^\top \vecbeta^t-y_i)$. Therefore, for reason similar to the one of  \cite{BalDuLiSin2017Computationally}, the sample complexity of \cite{LiuSheLiCar2020High}
		is affected by $\|\vecbeta^*\|_2$.
		On the other hand, our estimator is based on the $\ell_1$-penalized Huber regression, and  we can avoid this problem.

	\subsection{Result}	
	\label{sec:results}
	We state our main theorem.
	Define
	\begin{align}
		R_{d,n,o} = \rho c_{r_1}\sqrt{s} r_{d,s}+r_\delta +c_{r_2}\kappa_{\mr{u}}\left(\sqrt{\frac{o}{n}(sr_{d,s}+r_\delta)}+r_o\right).
	\end{align}
	\begin{theorem}
	\label{t:main}
		Suppose that (i) and (iii) of Assumption \ref{a:intro} and Assumption \ref{a:noise} hold. 
		Suppose that the parameters $\lambda_o,\lambda_s,\varepsilon,\tau_{\rm cut},r_1,r_2, r_\Sigma$ satisfy 
				\begin{align}
		\label{ine:lambda_o1}
			1&\geq 7c_ o \left(4 + c_s  \right) c_{\max}^2  \sqrt{1+\log L}L^2R_{d,n,o},\\
		\label{ine:lambda_s1}
			\lambda_s&= c_s c_{\max}^2L\lambda_o\sqrt{n}\frac{1}{c_{r_1}\sqrt{s}}R_{d,n,o},\quad
		 	\varepsilon = c_\varepsilon \frac{o}{n},\quad
			\tau_{\rm cut}=c_{\rm cut}(L\kappa_{\mr{u}})^2\left(sr_{d,s}+r_\delta +r_o'\right)r_2^2,\\
		\label{ine:r121}
		 	r_1&= c_{r_1}\sqrt{s}r_\Sigma ,\quad r_2 =  c_{r_2}r_\Sigma, \quad
			r_\Sigma =  7\left(4 + c_s  \right)c_{\max}^2 L\lambda_o\sqrt{n}R_{d,n,o},
		\end{align}
	 	where $c_o,c_s,c_\varepsilon, c_{\rm cut}, c_{r_1},c_{r_2}$, and $c_{\max}$ are sufficiently large numerical constants such that $c_o\geq 4$, $c_s \geq 3(c_{\mr{RE}}+1)/(c_{\mr{RE}}-1)$, $2> c_\varepsilon\geq 1$, $c_{\rm cut}\geq c_2$, $c_{r_1} = c_{r_1}^{\mr{num}}(1+c_{\mr{RE}})/\mathfrak{r}$, $c_{r_2} = c_{r_2}^{\mr{num}}(3+c_{\mr{RE}})/\kappa_{\mr{l}},\,\min\{c_{r_1}^{\mr{num}},\,c_{r_2}^{\mr{num}}\}\geq 2$ and $c_{r_1}^{\mr{num}}/c_{r_2}^{\mr{num}}\leq 1$. In  Proposition \ref{p:cwpre} and  Definition \ref{d:max}, $c_2$ and  $c_{\max}$  are defined, respectively.
		Suppose that $\max\{\sqrt{2}r_\delta, sr_{d,s}\}\leq 1$ and  $0<o/n\leq 1/(5e)$ hold.
	 	Then, the optimal solution $\hat{\vecbeta}$ satisfies the following:
		\begin{align}
		\label{ine:result1}
		\|\Sigma^\frac{1}{2}(\hat{\vecbeta} -\vecbeta^*)\|_2 & \leq r_\Sigma, \quad \|\hat{\vecbeta} -\vecbeta^*\|_2  \leq r_2\text{ and } \|\hat{\vecbeta} -\vecbeta^*\|_1 \leq r_1,
		\end{align}
		with probability at least $1-4\delta$.
	\end{theorem}
	We note that the conditions \eqref{ine:lambda_o1} and \eqref{ine:r121} in Theorem \ref{t:main}
	imply
	\begin{align}
		\label{ine:nsufficient}
		\lambda_o \sqrt{n} &\geq c_oL r_\Sigma\sqrt{1+\log L}.
	\end{align}
	\begin{remark} 
	\label{rem:result1} 
		We consider the results of \eqref{ine:result1} in details.
		Assume $\mbb{E}\xi_i^2\leq \sigma^2$ and the equality of \eqref{ine:lambdacond} hold.
		Define $C_{c_{\mr{RE}},1},\,C_{c_{\mr{RE}},2}$ and $C_{c_{\mr{RE}},3}$ are constants depending on $c_{\mr{RE}}$.
		Then, we have
		\begin{align}
			\label{ine:result1-rem}
			\|\Sigma^\frac{1}{2}(\hat{\vecbeta} -\vecbeta^*)\|_2 & \leq C_{c_{\mr{RE}},1}\mathfrak{L}^3(R_{\mr{lasso}}+R_{\mr{outlier}}), \\
			 \|\hat{\vecbeta} -\vecbeta^*\|_2  &\leq C_{c_{\mr{RE}},2}\mathfrak{L}^3\frac{1}{\kappa_{\mr{l}}}(R_{\mr{lasso}}+R_{\mr{outlier}}),\\
			\|\hat{\vecbeta} -\vecbeta^*\|_1 &\leq C_{c_{\mr{RE}},3}\mathfrak{L}^3\frac{1}{\mathfrak{r}}\sqrt{s}(R_{\mr{lasso}}+R_{\mr{outlier}}),
			\end{align}
			and we see that \eqref{ine:result1-rem} recovers \eqref{rec1-sigma}.
	\end{remark}

\section{Estimator without the covariance}
\label{sec:un}
In Theorem \ref{t:main}, we use the covariance of the covariates when we construct an estimator.
On the other hand, especially in practical terms, the use of the covariance would be unfavorable.
When we do not use the covariance in estimation, we need to modify algorithm WEIGHT as follows:
\begin{algorithm}[H]
	\caption{WEIGHT WITHOUT COVARIANCE}
	\label{alg:cw0-un}
		\begin{algorithmic}
			\REQUIRE{data $\{\vecX_i \}_{i=1}^n$, tuning parameters $\tau_{\rm cut}, \varepsilon,r_1,\,r_2$.}
			\ENSURE{weight estimate $\hat{\vecw} = \{\hat{w}_1,\cdots,\hat{w}_n\}$.}\\
			Let $\hat{\vecw}$ be the solution to 
			\begin{align}
			\label{cw-un}
				\min_{\vecw \in \Delta^{n-1}(\varepsilon)} \max_{M\in \mathfrak{M}_{r_1,r_2,d}^{\ell_1,\mr{Tr}} }\sum_{i=1}^n w_i \langle\vecX_i \vecX_i^\top,M\rangle
			\end{align}
			{\bf if} {the optimal value of \eqref{cw} $\leq \tau_{\rm cut} $}\\
			\ \ \ \ \ {\bf return} {$\hat{\vecw}$}\\
			{\bf else} \\
			\ \ \ \ \ {\bf return} {$fail$}\\
		\end{algorithmic}
	\end{algorithm}
In WEIGHT WITHOUT COVARIANCE, it is necessary to set the value of $\tau_{\mr{cut}}$ larger than that in WEIGHT. For detail, see Corollary \ref{c:cwpre}.
Then, Theorem \ref{t:main} is changed as follows.
Define
\begin{align}
R_{d,n,o}' =\rho c_{r_1}\sqrt{s} r_{d,s}+r_\delta +c_{r_2}\sqrt{\frac{o}{n}}\sqrt{\kappa_{\mr{u}}^2\left(sr_{d,s}+r_\delta \right)+\Sigma_{\max}^2}r_2.
\end{align}
	\begin{theorem}
		\label{t:main2}
			Suppose that (i) and (iii) of Assumption \ref{a:intro} and Assumption \ref{a:noise} hold. 
			Suppose that the parameters $\lambda_o,\lambda_s,\varepsilon,\tau_{\rm cut},r_1,r_2, r_\Sigma$ satisfy 
					\begin{align}
			\label{ine:lambda_o1:un}
				1&\geq 7c_ o \left(4 + c_s  \right) c'^2_{\max}  \sqrt{1+\log L}L^2R_{d,n,o}' ,\\
			\label{ine:lambda_s1:un}
				\lambda_s&= c_s c_{\max}^2L\lambda_o\sqrt{n}\frac{1}{c_{r_1}\sqrt{s}}R_{d,n,o}' ,\,\,
				 \varepsilon = c_\varepsilon \frac{o}{n},\,\,
				\tau_{\rm cut}=c_{\rm cut}\left((L\kappa_{\mr{u}})^2\left(sr_{d,s}+r_\delta \right)+\Sigma_{\max}^2\right)r_2^2,\\
			\label{ine:r121:un}
				 r_1&= c_{r_1}\sqrt{s}r_\Sigma ,\quad r_2 =  c_{r_2}r_\Sigma, \quad
				r_\Sigma =  7\left(4 + c_s  \right)c'^2_{\max}L\lambda_o\sqrt{n}R_{d,n,o}' ,
			\end{align}
			 where $c_o,c_s,c_\varepsilon, c_{\rm cut}, c_{r_1},c_{r_2}$, and $c'_{\max}$ are sufficiently large numerical constants such that $c_o\geq 4$, $c_s \geq 3(c_{\mr{RE}}+1)/(c_{\mr{RE}}-1)$, $2> c_\varepsilon\geq 1$, $c_{\rm cut}\geq c_7$, $c_{r_1} = c_{r_1}^{\mr{num}}(1+c_{\mr{RE}})/\mathfrak{r}$, $c_{r_2} = c_{r_2}^{\mr{num}}(3+c_{\mr{RE}})/\kappa_{\mr{l}},\,\min\{c_{r_1}^{\mr{num}},\,c_{r_2}^{\mr{num}}\}\geq 2$ and $c_{r_1}^{\mr{num}}/c_{r_2}^{\mr{num}}\leq 1$. In  Corollary \ref{c:cwpre} and  Definition \ref{d:max2}, $c_7$ and  $c_{\max}'$  are defined, respectively.
			Suppose that $\max\{\sqrt{2}r_\delta, sr_{d,s}\}\leq 1$,  $0<o/n\leq 1/2$ hold.
			 Then, the optimal solution $\hat{\vecbeta}$ satisfies the following:
			\begin{align}
			\label{ine:result:un}
			\|\Sigma^\frac{1}{2}(\hat{\vecbeta} -\vecbeta^*)\|_2 & \leq r_\Sigma, \,\, \|\hat{\vecbeta} -\vecbeta^*\|_2  \leq r_2\,\,\text{and} \,\,\|\hat{\vecbeta} -\vecbeta^*\|_1 \leq r_1,
			\end{align}
			with probability at least $1-3\delta$.
		\end{theorem}
		\begin{remark} 
				We consider the results of \eqref{ine:result:un} in details.
				Assume $\mbb{E}\xi_i^2\leq \sigma^2$ and the equality of \eqref{ine:lambdacond} hold.		Define $C_{c_{\mr{RE}},4},\,C_{c_{\mr{RE}},5}$ and $C_{c_{\mr{RE}},6}$ are constants depending on $c_{\mr{RE}}$.
				Then, we have
				\begin{align}
					\label{ine:result2-rem}
					\|\Sigma^\frac{1}{2}(\hat{\vecbeta} -\vecbeta^*)\|_2 & \leq C_{c_{\mr{RE}},4}\mathfrak{L}^3(R_{\mr{lasso}}+R'_{\mr{outlier}}), \\
					 \|\hat{\vecbeta} -\vecbeta^*\|_2  &\leq C_{c_{\mr{RE}},5}\mathfrak{L}^3\frac{1}{\kappa_{\mr{l}}}(R_{\mr{lasso}}+R'_{\mr{outlier}}),\\
					\|\hat{\vecbeta} -\vecbeta^*\|_1 &\leq C_{c_{\mr{RE}},6}\mathfrak{L}^3\frac{1}{\mathfrak{r}}\sqrt{s}(R_{\mr{lasso}}+R'_{\mr{outlier}}),
					\end{align}
					and we see that \eqref{ine:result2-rem} recovers \eqref{rec2-sigma}.
				Investigation of whether it is possible to achieve similar error bounds using a estimation method without covariance as in the case of with covariance, and the exploration of the trade-offs in such a scenario, are left as future research tasks.
			\end{remark}

\section{Key techniques}
\label{sec:keyL}
\subsection{Key propositions and lemma for Theorem \ref{t:main}}
First, we introduce Proposition \ref{p:cwpre}, that gives the condition on $\tau_{\rm cut}$ when the covariance matrix $\Sigma$ is known.  The proof is given in Section \ref{sec:pl}.
\begin{proposition}
\label{p:cwpre}
Suppose that the assumptions in Theorem \ref{t:main} hold.
	Define $c_1$ and $c_2$ as  numerical constants.
	Then, with probability at least $1-\delta$, we have
	\begin{align}
	\label{ine:cwpre'}
		\max_{M \in \mathfrak{M}_{r_1,r_2,d}^{\ell_1,\mr{Tr}} }\sum_{i=1}^n \frac{\left\langle\vecx_i \vecx_i^\top-\Sigma,M\right\rangle}{n}\leq c_1(L\kappa_{\mr{u}})^2\left(sr_{d,s}+r_\delta \right)r_2^2.
	\end{align}
	Additionally, 	with  probability  at least $1-2\delta$, we have 
	\begin{align}
	\label{ine:cwpre2}
		\max_{M \in \mathfrak{M}_{r_1,r_2,d}^{\ell_1,\mr{Tr}} }\sum_{i=1}^n \hat{w}_i\left\langle\vecX_i \vecX_i^\top-\Sigma,M\right\rangle\leq c_2(L\kappa_{\mr{u}})^2\left(sr_{d,s}+r_\delta  +r_o'\right)r_2^2.
	\end{align}
\end{proposition}
Therefore,  we see that, when $c_2(L\kappa_{\mr{u}})^2\left(sr_{d,s}r_\delta +r_o'\right)r_2^2\leq \tau_{\rm cut}$, Algorithm \ref{alg:cw0} succeeds at returning $\hat{\vecw}$ under  \eqref{ine:cwpre2}.
The key techniques of the proof of the proposition above are Corollary 2.8 of \cite{Zaj2020Bounds}, that is the one of the variants of the Hanson--Wright inequality \cite{HanWri1971Bound,Wri1973Bound,RudVer2013Hanson,Ada2015Note,HsuKakZha2012Tail}, and generic chaining for a subexponential random variable
(Corollary 5.2 of \cite{Dir2015Tail}).

Next, we introduce a deterministic proposition related to Theorem \ref{t:main}. Let 
\begin{align} 
	r_{\vecv,i} =\hat{w}_i'n\frac{y_i-\vecX_i^\top \vecv}{\lambda_o \sqrt{n}},\quad X_{\vecv,i} = \frac{\vecX_i^\top \vecv}{\lambda_o\sqrt{n}},\quad x_{\vecv,i} = \frac{\vecx_i^\top \vecv}{\lambda_o\sqrt{n}},\quad \xi_{\lambda_o,i} = \frac{\xi_i}{\lambda_o \sqrt{n}},
\end{align}
and for $\eta \in (0,1)$,
\begin{align}
	\vectheta = \hat{\vecbeta}-\vecbeta^*,\quad \vectheta_\eta = (\hat{\vecbeta}-\vecbeta^*)\eta.
\end{align}
The following proposition is proved in a manner similar to the proof of Proposition 9.1 of \cite{AlqCotLec2019Estimation}, and  the proof is given in the Appendix \ref{sec:mainproposition-pre} and \ref{sec:mainproof}. 
\begin{proposition}
\label{p:main}
	Suppose that, for any $\vectheta_\eta \in r_1 \mbb{B}^d_1 \cap r_2 \mbb{B}^d_2 \cap r_\Sigma \mbb{B}^d_\Sigma$, 
	\begin{align}
	\label{ine:det:main:0-1}
		\left| \lambda_o\sqrt{n}\sum_{i=1}^n \hat{w}_i'h(r_{\vecbeta^*,i}) \vecX_i^\top \vectheta_\eta \right| &\leq r_{a,1}r_1+r_{a,2} r_2+r_{a,\Sigma} r_\Sigma,\\
		\label{ine:det:main:0-2}
		b \|\Sigma^\frac{1}{2}\vectheta_\eta\|_2^2- r_{b,2} r_2-r_{b,\Sigma }r_\Sigma-r_{b,1}r_1&\leq \sum_{i=1}^n \lambda_o\sqrt{n} \hat{w}_i'\left(-h(r_{\vecbeta^*+\vectheta_\eta,i}) +h(r_{\vecbeta^*,i})\right) \vecX_i^\top \vectheta_\eta,
	\end{align}
	where $r_{a,1},r_{a,2}, r_{a,\Sigma}, r_{b,1}, r_{b,2}, r_{b,\Sigma}\geq0 ,\, b>0$ are some  numbers.
	Suppose that $\mbb{E}\vecx_i\vecx_i^\top=\Sigma$ satisfies $\mr{RE}(s,c_{\mr RE},\mathfrak{r}),\,\kappa_{\mr{l}}>0$, and 
	\begin{align}
	\label{ine:det:main:0-3}
		&\lambda_s-\left(r_{a,1}+ \frac{c_{r_2}r_{a,2}+ r_{a,\Sigma}}{c_{r_1}\sqrt{s}}\right) >0,\quad\frac{\lambda_s +  \left(r_{a,1}+ \frac{c_{r_2}r_{a,2}+ r_{a,\Sigma}}{c_{r_1}\sqrt{s}}\right)}{\lambda_s -  \left(r_{a,1}+ \frac{c_{r_2}r_{a,2}+ r_{a,\Sigma}}{c_{r_1}\sqrt{s}}\right) }\leq c_{\mr{RE}},\\
	\label{ine:det:main:0-4}
		&r_\Sigma \geq \frac{2}{b} \left(c_{r_1}\sqrt{s}(r_{a,1}+r_{b,1})+
		c_{r_2}(r_{a,2}+r_{b,2}) +r_{a,\Sigma}+r_{b,\Sigma} + c_{r_1}\sqrt{s}\lambda_s\right),\quad r_1 =  c_{r_1}\sqrt{s}r_\Sigma\quad r_2 = c_{r_2} r_\Sigma
	\end{align}
	hold, where $c_{r_1} = c_{r_1}^{\mr{num}}(1+c_{\mr{RE}})/\mathfrak{r}$, $c_{r_2} = c_{r_2}^{\mr{num}}(3+c_{\mr{RE}})/\kappa_{\mr{l}},\,\min\{c_{r_1}^{\mr{num}},\,c_{r_2}^{\mr{num}}\}\geq 2$ and $c_{r_1}/c_{r_2}\leq 1$.
	Then, we have the following:
	\begin{align}
	\label{ine:knownresults}
		\|\vecbeta^*-\hat{\vecbeta}\|_1\leq r_1,\quad \|\vecbeta^*-\hat{\vecbeta}\|_2\leq r_2,\quad \|\Sigma^\frac{1}{2}(\vecbeta^*-\hat{\vecbeta})\|_2\leq r_\Sigma.
	\end{align}
\end{proposition}

In the remainder of Section \ref{sec:keyL}, we introduce Propositions \ref{p:main1}--\ref{p:main:sc} and one lemma.
In Section \ref{sec:pmt}, we prove Theorem \ref{t:main} using the propositions.
In the proof of  Theorem \ref{t:main}, we 
prove that \eqref{ine:det:main:0-1}  - \eqref{ine:det:main:0-4}
are satisfied with high probability for appropriate values of $r_{a,1}, r_{a,2}, r_{a,\Sigma}, r_{b,1}, r_{b,2}, r_{b,\Sigma}$ and $b$ under the assumptions in Theorem \ref{t:main}, and we see that \eqref{ine:knownresults} is also satisfied. Then, we have the result \eqref{ine:result1} in Theorem \ref{t:main}.
We note that, for Proposition \ref{p:main:sc}, similar statements are found, for example, in \cite{SunZhoFan2020Adaptive,CheZho2020Robust}, and the proof of Proposition \ref{p:main:sc} basically follows the same line to the ones in \cite{SunZhoFan2020Adaptive,CheZho2020Robust}. 
Propositions \ref{p:main:out} and \ref{p:main:out2} are proved by relatively simple calculations based on the result of Proposition \ref{p:cwpre}.  
Therefore, the proofs of Propositions \ref{p:main:out}-\ref{p:main:sc}  are given in the Appendix \ref{sec:proofkeyL}.

\begin{proposition}
\label{p:main1}
Suppose that the assumptions in Theorem \ref{t:main} or Theorem \ref{t:main2} hold.
	Then, for any $\vecv\in r_1 \mbb{B}^d_1 \cap r_\Sigma \mbb{B}^d_\Sigma$, with probability at least $1-\delta$, we have
	\begin{align}
	\label{ine:p:main1}
		\left| \sum_{i=1}^n \frac{1}{n} h(\xi_{\lambda_o,i}) \vecx_i^\top\vecv \right|\leq c_3L\left(\rho r_{d,s}r_1+\sqrt{s}r_{d,s}r_\Sigma+r_\delta r_\Sigma\right),
	\end{align}
	where $c_3$ denotes a numerical constant.
\end{proposition}

\begin{proposition}
\label{p:main:out}
Suppose that the assumptions in Theorem \ref{t:main} hold.
	Furthermore, suppose that \eqref{ine:cwpre'} and \eqref{ine:cwpre2} hold and that Algorithm \ref{alg:cw0} returns $\hat{\vecw}$.
	For any $\vecu \in \mbb{R}^n$ such that $\|\vecu\|_\infty \leq 2$ and for any $\vecv \in r_1 \mbb{B}^d_1 \cap r_2 \mbb{B}^d_2$, we have 
	\begin{align}
	\label{ine:p:main:out}
		\left|\sum_{i \in \mc{O}}\hat{w}'_iu_i \vecX_i^\top\vecv \right|\leq c_4L\sqrt{1+c_{\rm cut}}\left(\kappa_{\mr{u}}\sqrt{\frac{o}{n}}\left(\sqrt{sr_{d,s}}+\sqrt{r_\delta}\right)+\kappa_{\mr{u}}r_o\right)r_2,
	\end{align}
	where $c_4$ is a numerical constant that depends on $c_1$ and  $c_2$.
\end{proposition}

Let $I_m$ be an index set such that $|I_m|$ = $m$.
\begin{proposition}
\label{p:main:out2}
Suppose that the assumptions in Theorem \ref{t:main} hold.
	Furthermore, suppose that \eqref{ine:cwpre'} holds.
	Then, for any $m \in \mbb{N}$ such that $m\leq (2c_\varepsilon+1) o$, for any $\vecu \in \mbb{R}^n$ such that $\|\vecu\|_\infty \leq 2$ and for any $\vecv \in r_1 \mbb{B}^d_1 \cap r_2 \mbb{B}^d_2$, we have the following:
	\begin{align}
	\label{ine:p:main:out2}
		\left|\sum_{i \in I_m}\frac{1}{n}u_i \vecx_i^\top\vecv \right|\leq c_5L\left(\kappa_{\mr{u}}\sqrt{\frac{o}{n}}\left(\sqrt{sr_{d,s}}+\sqrt{r_\delta}\right)+\kappa_{\mr{u}}r_o\right)r_2,
	\end{align}
	where $c_5$ is a numerical constant that depends on $c_1$ and $c_\varepsilon$.
\end{proposition}

\begin{proposition}
\label{p:main:sc}
Suppose that the assumptions in Theorem \ref{t:main} or Theorem \ref{t:main2} hold.
	Then, for any $\vecv \in r_1 \mbb{B}^d_1 \cap r_\Sigma \mbb{B}^d_\Sigma$, with probability  at least $1-\delta$, we have 
	\begin{align}
	\label{ine:sc}
		&\sum_{i=1}^n \frac{\lambda_o}{\sqrt{n}} \left(-h(\xi_{\lambda_o,i}-x_{\vecv,i})+h(\xi_{\lambda_o,i})\right)\vecx_i^\top \vecv \geq \frac{\|\Sigma^\frac{1}{2}\vecv\|_2^2}{3}-c_{6}L\lambda_o\sqrt{n}\left(\rho r_{d,s} r_1+\sqrt{s}r_{d,s}r_\Sigma+r_\delta r_\Sigma\right),
	\end{align}
	where $c_6$ is a numerical constant.
\end{proposition}
We define $c_{\max}$.
\begin{definition}
\label{d:max}
	Define 
	\begin{align}
		c_{\max} = \max \left(1,c_3,c_4 \sqrt{1+c_{\rm cut}},c_5,c_6\right).
	\end{align}
\end{definition}
Let $I_{<}$ and $I_{\geq}$ be the sets of indices such that $w_i < 1/(2n)$ and $w_i \geq 1/(2n)$, respectively. 
\begin{lemma}
\label{l:w2}
	Suppose that $0<\varepsilon <1$. Then, for any $\vecw \in \Delta^{n-1}(\varepsilon)$, we have $|I_{<}| \leq 2n\varepsilon$.
\end{lemma}

 \subsection{Key propositions and corollary for Theorem \ref{t:main2}}
 \label{sec:keyL2}
 First, we introduce Corollary \ref{c:cwpre}, that gives the condition on $\tau_{\rm cut}$ when we do not use $\Sigma$ in the estimator.   The proof is given in Appendix.
 \begin{corollary}
 \label{c:cwpre}
	 Suppose that the assumptions in Theorem \ref{t:main2} hold.
	 Define $c_1'$ and $c_7$ as  numerical constants.
	 Then, with probability at least $1-\delta$, we have
	 \begin{align}
	 \label{ine:cwpre2-1}
		 \max_{M \in \mathfrak{M}_{r_1,r_2,d}^{\ell_1,\mr{Tr}} }\sum_{i=1}^n \frac{\left\langle\vecx_i \vecx_i^\top,M\right\rangle}{n}\leq c_1' L^2\left(\kappa_{\mr{u}}^2\left(sr_{d,s}+r_\delta \right)+\Sigma_{\max}^2\right)r_2^2.
	 \end{align}
	 Additionally,  with probability at least $1-\delta$, we have 
	 \begin{align}
	 \label{ine:cwpre2-2}
		 \max_{M \in \mathfrak{M}_{r_1,r_2,d}^{\ell_1,\mr{Tr}} }\sum_{i=1}^n \hat{w}_i\left\langle\vecX_i \vecX_i^\top ,M\right\rangle\leq c_7L^2\left(\kappa_{\mr{u}}^2\left(sr_{d,s}+r_\delta \right)+\Sigma_{\max}^2\right)r_2^2.
	 \end{align}
 \end{corollary}
 Therefore, under \eqref{ine:cwpre2-2}, when $c_7\left((L\kappa_{\mr{u}})^2\left(sr_{d,s}+r_\delta \right)+\Sigma_{\max}^2\right)r_2^2\leq \tau_{\rm cut}$, Algorithm \ref{alg:cw0-un} succeeds at  returning $\hat{\vecw}$. We see that  when the covariance is not used in the estimator, it is necessary to set the value of $\tau_{\mr{cut}}$ to a higher magnitude compared to the case when covariance is used.
 Using the propositions, in Section \ref{asec:pmt}, we can 
prove that \eqref{ine:det:main:0-1}  - \eqref{ine:det:main:0-4}
are satisfied with a high probability for appropriate values of $r_{a,1}, r_{a,2}, r_{a,\Sigma}, r_{b,1}, r_{b,2}, r_{b,\Sigma}$ and $b$ under the assumptions in Theorem \ref{t:main2}, and we see that \eqref{ine:knownresults} is also satisfied. Then, we can have the result \eqref{ine:result:un} in Theorem \ref{t:main2}.
When  $\Sigma$ is not used in the estimator, Propositions \ref{p:main} and \ref{p:main:sc} and Lemma \ref{l:w2} are commonly used, and 
 instead of Propositions \ref{p:main:out} and \ref{p:main:out2}, Propositions \ref{p:main:out-un} and \ref{p:main:out2-un} are used, respectively.
 In Definition \ref{d:max2}, $c_{\max}'$ is defined.
 The proofs of Propositions \ref{p:main:out-un} and  \ref{p:main:out2-un} are given in the  Appendix  \ref{sec:proofkeyL} because  Propositions  \ref{p:main:out-un} and  \ref{p:main:out2-un}  can be proved by  simple calculations based on the result of Proposition \ref{c:cwpre}.

 \begin{proposition}
 \label{p:main:out-un}
	 Suppose that the assumptions in Theorem \ref{t:main2} hold.
	 Furthermore, suppose that \eqref{ine:cwpre2-1} holds and that  Algorithm \ref{alg:cw0-un} returns $\hat{\vecw}$.
	 For any $\vecu \in \mbb{R}^n$ such that $\|\vecu\|_\infty \leq 2$ and for any $\vecv \in r_1 \mbb{B}^d_1 \cap r_2 \mbb{B}^d_2$, we have the following:
	 \begin{align}
		 \label{ine:p:main:out-un}
		 \left|\sum_{i \in \mc{O}}\hat{w}'_iu_i \vecX_i^\top\vecv \right|\leq c_8\sqrt{c_{\rm cut}}L\sqrt{\frac{o}{n}}\sqrt{\kappa_{\mr{u}}^2\left(sr_{d,s}+r_\delta \right)+\Sigma_{\max}^2}r_2,
	 \end{align}
	 where $c_8$ is a numerical constant.
 \end{proposition}
 
 \begin{proposition}
 \label{p:main:out2-un}
	 Suppose that the assumptions in Theorem \ref{t:main2} hold.
	 Furthermore, suppose that \eqref{ine:cwpre2-2} holds.
	 Then, for any $m \in \mbb{N}$ such that $m\leq (2c_\varepsilon+1) o$, for any $\vecu \in \mbb{R}^n$ such that $\|\vecu\|_\infty \leq 2$ and for any $\vecv \in r_1 \mbb{B}^d_1 \cap r_2 \mbb{B}^d_2$, we have the following
	 :
	 \begin{align}
		 \label{ine:p:main:out2-un}
		 \left|\sum_{i \in I_m}\frac{1}{n}u_i \vecx_i^\top\vecv \right|\leq c_9L\sqrt{\frac{o}{n}}\sqrt{\kappa_{\mr{u}}^2\left(sr_{d,s}+r_\delta \right)+\Sigma_{\max}^2}r_2,
	 \end{align}
	 where $c_9$ is a numerical constant that depends on $c_1'$ and $c_\varepsilon$.
 \end{proposition}
 
For Theorem \ref{t:main2}, the following definition of $c_{\max}'$ is used.
 \begin{definition}
 \label{d:max2}
	 Define 
	 \begin{align}
		 c_{\max}' = \max \left(1,c_3,c_6,c_8\sqrt{c_{\rm cut}},c_9\right).
	 \end{align}
 \end{definition}

\section{Proofs of  Propositions  \ref{p:cwpre} and  \ref{p:main1}}
\label{sec:pl}
The value of the numerical constant $C$ shall be allowed to change from line to line.
Define $\gamma_{s_0,\alpha}$-functional:
\begin{definition}[$\gamma_{s_0,\alpha}$-functional, \cite{Men2016Upper}]
	\label{d:tg2}
		Let $(T, d)$ be a semi-metric space with $d(x, z) \leq d(x, y) +
		d(y, z)$ and $d(x,y) = d(y,x)$ for $x, y, z \in T$.
		A sequence $\mc{T} = \{T_m\}_{m\geq 0}$ with subsets of $T$ is said to be admissible if $|T_0|=1$ and $|T_m|\leq 2^{2^m}$ for all $m\geq 1$.
		For any $\alpha \in (0,\infty)$, the $\gamma_{s_0,\alpha}$-functional of $(T,d)$ is defined by
		\begin{align}
			\gamma_{s_0,\alpha} (T,d) = \inf_{\mc{T}} \sup_{t\in T}\sum_{m=s_0}^\infty2^{\frac{m}{\alpha}} \inf_{s\in T_m} d(t,s),
		\end{align}
		where the infimum is taken over all admissible sequences $\mc{T} = \{T_m\}_{m\geq 0}$.
\end{definition}
	In particular, when $s_0=0$, $\gamma_{0,\alpha}$-functional is called the $\gamma_\alpha$-functional. 
	For a random variable $z$ and $p\geq 1$, define $\|\cdot\|_{(p)}$-norm as 
	\begin{align}
		\label{ine:(p)norm}
		\|z\|_{(p)}= \sup_{1\leq q\leq  p} \frac{\|z\|_{L_q}}{\sqrt{q}}.
	\end{align}
	Next, we define $\Lambda_{s_0,u}$ and $\tilde{\Lambda}_{s_0,u}$.
	\begin{definition}[$\Lambda_{s_0,u}$ and $\tilde{\Lambda}_{s_0,u}$, \cite{Men2016Upper}]
		\label{d:lambda}
			Given a class of functions $F$, $u\geq1$, $s_0 \geq 0$, define 
			\begin{align}
				\Lambda_{s_0,u}(F)= \inf_{\mc{F}} \sup_{ f\in F} \sum_{m= s_0}^\infty 2^{\frac{m}{2}}\|f-\pi_mf\|_{(u^22^m)},
			\end{align}
			where the infimum is taken over all admissible sequences $\mc{F} = \{F_m\}_{m\geq 0}$, and $\pi_m f$ is the nearest point in $F_m$ to $f$ in $\|\cdot\|_{(u^22^m)}$-norm. Additionally, define 
			\begin{align}
				\tilde{\Lambda}_{s_0,u}(F) = \Lambda_{s_0,u}(F)+2^\frac{s_0}{2}\sup_{f\in F}\|\pi_{s_0}f\|_{(u^22^{s_0})}.
			\end{align}
		\end{definition}
		We introduce  Lemma \ref{l:gwslope:main}, which is used in the proofs of  Propositions \ref{p:cwpre} and  \ref{p:main1}.
		The proof of this lemma  is given in the  Appendix \ref{asec:prooftg-eval}.
		Let $\vecg = (g_1,\cdots,g_d)^\top$ be the $d$-dimensional standard normal Gaussian random vector.
		\begin{lemma}
		\label{l:gwslope:main}
			Suppose that the assumptions in Theorem \ref{t:main} or Theorem \ref{t:main2} hold.
			We have
			\begin{align}
				\mbb{E} \sup_{\vecv \in r_1 \mbb{B}^d_1 \cap r_\Sigma \mbb{B}^d_\Sigma } \langle \Sigma^\frac{1}{2} \vecg, \vecv \rangle&\leq C\rho r_1\sqrt{ \log (d/s) }.
			\end{align}
		\end{lemma}

	\subsection{Preparation for the proof of Proposition  \ref{p:cwpre}}

	For a matrix $M$, we define the operator norm and Frobenius norm of $M$ as $\|M\|_{\rm op}$ and $\|M\|_{\rm F}$, respectively.
	and we define the number of non-zero elements of $M$  as $\|M\|_0$.
	Additionally, we define $a\mbb{B}^{d\times d}_1 = \{M \in \mbb{R}^{d\times d} \,\mid\, \|M\|_1\leq a \}$, $a\mbb{B}^{d\times d}_{\mr{F}} = \{M \in \mbb{R}^{d\times d} \,\mid\, \|M\|_{\mr{F}}\leq a \}$, and $a\mbb{B}^{d\times d}_0 = \{M \in \mbb{R}^{d\times d} \,\mid\, \|M\|_0\leq a \}$.

	\begin{lemma} \label{prop:twotwo}
		Suppose that (i) and (iii) of Assumptions \ref{a:intro}.
	For any fixed $M,M' \in  s^2 \mbb{B}^{d\times d}_0 \cap r_2^2 \mbb{B}^{d\times d}_{\mr{F}}$, we have
	\begin{align}
		\label{ine:bernstein1_1}
		\mbb{E}\left|\langle \vecx \vecx^\top-\Sigma, M\rangle \right|^p &\leq {p!} {(C(L\kappa_{\mr{u}})^2r_2^2)^p},  \\
		\label{ine:bernstein1_2}
		\| \langle \vecx \vecx^\top, M-M'\rangle \|_{\psi_1} &\le C(L\kappa_{\mr{u}})^2 \|M-M'\|_{\mr{F}}. 
	\end{align}
	\end{lemma}
	
	\begin{proof}
		In this proof, we define $c$ as some positive numerical constant.
		For any $a$-dimensional random variable $\veca$, 
		let the Luxemburg norm, that is an extension of the $\psi_2$-norm from scalar to vector, be denoted by
		\begin{align}
			\|\veca\|_{\psi_2} = \inf\left\{\eta>0\,:\,\sup_{\vecv \in \mbb{S}^{a-1}}\mbb{E}\exp\left(\frac{\langle \vecv,\veca\rangle^2}{\eta^2}\right)\leq 2\right\} 
		\end{align}
		From  \eqref{ine:lsg2-2}, for any index set $J$ such that $|J|\leq 2s^2$ and $\vecv \in \mbb{S}^{2s^2-1}$ we have
		\begin{align}
			\mbb{E}\exp\left(\frac{\langle \vecv,\vecx|_J\rangle^2}{c_{\mathfrak{L}}^2\mathfrak{L}^2\kappa_{\mr{u}}^2}\right) \leq 2 ,
		\end{align}
		and we have
		\begin{align}
		\label{lux}
			\|\vecx|_J\|_{\psi_2}  \leq c_{\mathfrak{L}}\mathfrak{L}\kappa_{\mr{u}}\leq L \kappa_{\mr{u}}.
		\end{align}
		For any matrix $A$ and any index set $J \subset \{1,\cdots,d\}$, define $A|_{J,J}$ as the matrix such that all the elements of the $i$-th rows and columns are zero for $i \in J^c$. 
		Fix $M \in s^2 \mbb{B}^{d\times d}_0 \cap r_2^2 \mbb{B}^{d\times d}_{\mr{F}}$. For $M$, let $K \subset \{1,\cdots,d\}$ be an index set such that $M_{ij} = 0$ for $i \in K^c$ or $j \in K^c$. We note that $|K|\leq s^2$. From Corollary 2.8 of \cite{Zaj2020Bounds} and \eqref{lux},  for any $t>0$, we have
		\begin{align}
			\label{ine:HS}
			\mbb{P} \left(\left|\langle \vecx \vecx^\top-\Sigma , M\rangle \right|>t \right) &=\mbb{P} \left(\left|\langle \vecx|_K \vecx|_K^\top-\Sigma|_{K,K} , M\rangle\right| >t \right) \nonumber\\
			&\leq 2 \exp\left\{-c \min \left(\frac{t^2}{(L\kappa_{\mr{u}})^4\|M\|_{\mr F}^2}, \frac{t}{(L\kappa_{\mr{u}})^2\|M\|_{\mr{F}}}\right)\right\}.
		\end{align}
		From \eqref{ine:HS}, for $t\leq (L\kappa_{\mr{u}})^2\|M\|_{\mr{F}}$, we have
		\begin{align}
			\label{ine:HS2}
			\mbb{P} \left(\left|\langle \vecx \vecx^\top-\Sigma , M\rangle \right|> t \right) \leq  2\exp\left(-c \frac{t^2}{(L\kappa_{\mr{u}})^4\|M\|_{\mr F}^2}\right),
		\end{align}
		and for $t\geq (L\kappa_{\mr{u}})^2\|M\|_{\mr{F}}$, we have
		\begin{align}
			\label{ine:HS3}
			\mbb{P} \left(\left|\langle \vecx \vecx^\top-\Sigma , M\rangle \right|> t \right)\leq  2\exp\left(-c \frac{t}{(L\kappa_{\mr{u}})^2\|M\|_{\mr{F}}}\right).
		\end{align}
	  We follow almost the same argument of the proof of Proposition 2.5.2 of \cite{Ver2018High}. For any $1\leq p<\infty$, we have
		\begin{align}
			\mbb{E}\left|\langle \vecx \vecx^\top-\Sigma , M\rangle \right|^{p} &= \int^\infty_0\mbb{P} \left(\left|\langle \vecx \vecx^\top-\Sigma , M\rangle \right|^{p}\geq u \right)du\nonumber\\
			&= \int^\infty_0 \mbb{P} \left(\left|\langle \vecx \vecx^\top-\Sigma , M\rangle \right|\geq t\right)pt^{p-1}dt\nonumber\\ 
			&\leq \int^\infty_02\exp\left(-c \frac{t^2}{(L\kappa_{\mr{u}})^4\|M\|_{\mr F}^2}\right) pt^{p-1}dt+\int^\infty_02\exp\left(-c \frac{t}{(L\kappa_{\mr{u}})^2\|M\|_{\mr F}} \right)pt^{p-1}dt\nonumber\\ 
			&\leq p\frac{(L\kappa_{\mr{u}})^{2p}\|M\|_{\mr F}^p}{\sqrt{c^p}}\Gamma(p/2)+2p \frac{(L\kappa_{\mr{u}})^{2p}\|M\|_{\mr F}^p}{c^p}\Gamma(p)\nonumber\\ 
			&\leq C \left(L\kappa_{\mr{u}} \left(\frac{1}{c^\frac{1}{4}}+\frac{1}{\sqrt{c}}\right)\right)^{2p}\|M\|_{\mr F}^p p \left( (p/2)^{p/2}+p^p\right),
		\end{align}
		and  we have
		\begin{align}
			\left(\mbb{E}\langle \vecx_i \vecx_i^\top-\Sigma, M\rangle^p\right) ^\frac{1}{p}&\leq C\left(L\kappa_{\mr{u}} \left(\frac{1}{c^\frac{1}{4}}+\frac{1}{\sqrt{c}}\right)\right)^2 \|M\|_{\mr{F}}  \left(p (p/2)^{p/2}+p^{p+1}\right)^\frac{1}{p}\nonumber\\
			&\leq C \left(L\kappa_{\mr{u}} \left(\frac{1}{c^\frac{1}{4}}+\frac{1}{\sqrt{c}}\right)\right)^2 \|M\|_{\mr{F}}  \left(2p^{p+1}\right)^\frac{1}{p}\nonumber\\
			&\leq   C \left(L\kappa_{\mr{u}} \left(\frac{1}{c^\frac{1}{4}}+\frac{1}{\sqrt{c}}\right)\right)^2\|M\|_{\mr{F}} p\leq  C \left(L\kappa_{\mr{u}}\right)^2\|M\|_{\mr{F}}p.
		\end{align}
		From Stirling's formula $p^p \le p! e^p$, we have
		\begin{align}
		\label{ine:bernstein}
			\max_{M\in\mathfrak{M}_{r_1,r_2,d}^{\ell_1,\mr{Tr}} }\frac{1}{n}\sum_{i=1}^n\mbb{E}\left|\langle \vecx_i \vecx_i^\top-\Sigma, M\rangle \right|^p\leq C p! e^p(L\kappa_{\mr{u}})^{2p}r_2^{2p},
		\end{align}
		and the proof of \eqref{ine:bernstein1_1} is complete.

		For any fixed $M,M' \in \mathfrak{M}_{r_1,r_2,d}^{\ell_1,\mr{Tr}} $,  let $K' \subset \{1,\cdots,d\}$ is the index set such that $(M-M')_{ij} = 0$ for $i \in K'^c$ or $j \in K'^c$. We note that $|K'|\leq 2\times s^2$. From Proposition 2.6, Remark 2.7 of \cite{Zaj2020Bounds} and \eqref{lux}, we have 
		\begin{align}
			\label{ine:sparsepsi1}
			\left\|\langle \vecx \vecx^\top, M-M'\rangle\right\|_{\psi_1}=\left\|\langle \vecx|_{K'} \vecx|_{K'} ^\top, (M-M')|_{K',K'}\rangle\right\|_{\psi_1} &\leq C (L\kappa_{\mr{u}})^2  \|(M-M')|_{K',K'}\|_{\mr{F}} \nonumber\\
			&\leq C (L\kappa_{\mr{u}})^2  \|M-M'\|_{\mr{F}},
		\end{align}
		and the proof of \eqref{ine:bernstein1_2} is complete. 
	\end{proof}
	
Using the lemma above, we have the key lemma (Lemma \ref{l:gc}) to prove Propositions \ref{p:cwpre}.
	For a set $K$, we define $\mr{conv}(K)$ as  its convex hull.
	Before the statement and the proof of Lemma \ref{l:gc}, we introduce the following lemma, which slightly generalizes Lemma 3.1 of \cite{PlaVer2013One} and the proof is in Appendix \ref{sec:p1}.
	\begin{lemma}
		\label{l:convl0l1}
		We have
			\begin{align}
				\label{ine:convmat}
				r_1^2 \mbb{B}^{d\times d}_1 \cap r_2^2 \mbb{B}^{d\times d}_{\mr{F}} &\subset 2 \mr{conv}\left(\left(\frac{r_1}{r_2}\right)^4 \mbb{B}^{d\times d}_0 \cap r_2^2 \mbb{B}^{d\times d}_{\mr{F}}\right),\\
				\label{ine:convvec}
				r_1 \mbb{B}^{d}_1 \cap r_2 \mbb{B}^{d}_{2} &\subset  2 \mr{conv}\left(\left(\frac{r_1}{r_2}\right)^2 \mbb{B}^d_0 \cap r_2 \mbb{B}^d_2\right).
			\end{align}
		\end{lemma}
Then, we introduce Lemma \ref{l:gc}.
	\begin{lemma}
	\label{l:gc}
	Suppose that (i) and (iii) of Assumptions \ref{a:intro} hold.  Then, with probability at least $1-\delta$, we have
		\begin{align}
			&\max_{M\in \mathfrak{M}_{r_1,r_2,d}^{\ell_1,\mr{Tr}} }\left| \frac{1}{n}\sum_{i=1}^n\langle \vecx_i \vecx_i^\top-\Sigma, M\rangle\right| \nonumber \\
			& \hspace*{10mm} \leq C(L\kappa_{\mr{u}})^2\left(\frac{\gamma_1(s^2 \mbb{B}^{d\times d}_0 \cap r_2^2 \mbb{B}^{d\times d}_{\mr{F}},\|\cdot\|_{\mr{F}})}{n}+\frac{\gamma_2(s^2 \mbb{B}^{d\times d}_0 \cap r_2^2 \mbb{B}^{d\times d}_{\mr{F}},\|\cdot\|_{\mr{F}})}{\sqrt{n}}+r_\delta^2r^2_2+r_\delta r_2^2\right).
		\end{align}
	\end{lemma}
	
	\begin{proof}
		First,  we have
		\begin{align}
			\mathfrak{M}_{r_1,r_2,d}^{\ell_1,\mr{Tr}}\subset r_1^2 \mbb{B}^{d\times d}_1 \cap r_2^2 \mbb{B}^{d\times d}_{\mr{F}} \stackrel{(a)}{\subset} 2 \mr{conv}(s^2 \mbb{B}^{d\times d}_0 \cap r_2^2 \mbb{B}^{d\times d}_{\mr{F}}),
	\end{align}
	where (a) follows from $r_1^4/r_2^4 \leq c_{r_1}^4s^2/c_{r_2}^4$, $c_{r_1}/c_{r_2}\leq 1$ and Lemma \ref{l:convl0l1}, identifying vectors with matrices. Then, we have
	\begin{align}
		&\max_{M\in \mathfrak{M}_{r_1,r_2,d}^{\ell_1,\mr{Tr}}}\left| \frac{1}{n}\sum_{i=1}^n\langle \vecx_i \vecx_i^\top-\Sigma, M\rangle\right|\nonumber\\
		& \hspace*{10mm} \leq \max_{M\in  r_1^2 \mbb{B}^{d\times d}_1 \cap r_2^2 \mbb{B}^{d\times d}_{\mr{F}}}\left| \frac{1}{n}\sum_{i=1}^n\langle \vecx_i \vecx_i^\top-\Sigma, M\rangle\right|\leq 		\max_{M\in 2 \mr{conv}(s^2 \mbb{B}^{d\times d}_0 \cap r_2^2 \mbb{B}^{d\times d}_{\mr{F}})}\left| \frac{1}{n}\sum_{i=1}^n\langle \vecx_i \vecx_i^\top-\Sigma, M\rangle\right|.
	\end{align}
	Identifying vectors with matrices, from Lemma D.8 of \cite{Oym2018Learning}, we have
	\begin{align}
	\max_{M\in 2 \mr{conv}(s^2 \mbb{B}^{d\times d}_0 \cap r_2^2 \mbb{B}^{d\times d}_{\mr{F}})}\left| \frac{1}{n}\sum_{i=1}^n\langle \vecx_i \vecx_i^\top-\Sigma, M\rangle\right|\leq 	\max_{M\in  s^2 \mbb{B}^{d\times d}_0 \cap r_2^2 \mbb{B}^{d\times d}_{\mr{F}}}\left| \frac{2}{n}\sum_{i=1}^n\langle \vecx_i \vecx_i^\top-\Sigma, M\rangle\right|.
	\end{align}

		Next, note that the upper bound of the first inequality in Lemma~\ref{prop:twotwo} is independent of $i$ and $M$. Hence, from Corollary 5.2 of \cite{Dir2015Tail}, with probability at least $1-\delta$, from the same argument used to have \eqref{ine:sparsepsi1}, we have
		\begin{align}
		\label{Cor52}
			\lefteqn{ \max_{M\in\mathfrak{M}_{r_1,r_2,d}^{\ell_1,\mr{Tr}} }\left|\frac{1}{n}\sum_{i=1}^n\langle \vecx_i \vecx_i^\top-\Sigma, M\rangle \right|} \\
			& \hspace*{10mm} \leq C\left(\frac{\gamma_1(s^2 \mbb{B}^{d\times d}_0 \cap r_2^2 \mbb{B}^{d\times d}_{\mr{F}},d_1)}{n}+\frac{\gamma_2(s^2 \mbb{B}^{d\times d}_0 \cap r_2^2 \mbb{B}^{d\times d}_{\mr{F}},d_2)}{\sqrt{n}}+(L\kappa_{\mr{u}})^2r_\delta r^2_2+(L\kappa_{\mr{u}})^2r_\delta^2 r_2^2\right), \nonumber
		\end{align}
		where the semi-metrics $d_1(M,M')$ and $d_2(M,M')$ for $M,M'\in s^2 \mbb{B}^{d\times d}_0 \cap r_2^2 \mbb{B}^{d\times d}_{\mr{F}}$ are defined in \cite{Dir2015Tail}. Since $\{\vecx_i\}_{i=1}^n$ is an i.i.d. sequence, two semi-metrics are the same and then, for any $i\in\{1,\cdots,n\}$, we have
		\begin{align}
			d_1(M,M') = d_2(M,M') = \| \langle \vecx_i \vecx_i^\top, M-M'\rangle \|_{\psi_1}\leq  C (L\kappa_{\mr{u}})^2  \|M-M'\|_{\mr{F}},
		\end{align}
		where the last inequality follows from \eqref{ine:bernstein1_2}. 
		We know that $ \gamma_1(s^2 \mbb{B}^{d\times d}_0 \cap r_2^2 \mbb{B}^{d\times d}_{\mr{F}},d_1)$ and $ \gamma_2(s^2 \mbb{B}^{d\times d}_0 \cap r_2^2 \mbb{B}^{d\times d}_{\mr{F}},d_2)$ is monotone increasing with respect to $d_1$ and $d_2$, and for some constants $\mathfrak{c}_1,\mathfrak{c}_2$, we have
		\begin{align}
		\gamma_1(s^2 \mbb{B}^{d\times d}_0 \cap r_2^2 \mbb{B}^{d\times d}_{\mr{F}},\mathfrak{c}_1d_1)&=\mathfrak{c}_1\gamma_1(s^2 \mbb{B}^{d\times d}_0 \cap r_2^2 \mbb{B}^{d\times d}_{\mr{F}},d_1),\nonumber\\ 
		\gamma_2(s^2 \mbb{B}^{d\times d}_0 \cap r_2^2 \mbb{B}^{d\times d}_{\mr{F}},\mathfrak{c}_2d_2)&=\mathfrak{c}_2\gamma_2(s^2 \mbb{B}^{d\times d}_0 \cap r_2^2 \mbb{B}^{d\times d}_{\mr{F}},d_2).
		\end{align}
		Combining \eqref{Cor52} with the above properties of the $\gamma_\alpha$-functional, we have
		\begin{align}
			\lefteqn{ \max_{M\in\mathfrak{M}_{r_1,r_2,d}^{\ell_1,\mr{Tr}} }\left|\frac{1}{n}\sum_{i=1}^n\langle \vecx_i \vecx_i^\top-\Sigma, M\rangle \right| } \\
	 		& \hspace*{10mm} \leq C(L\kappa_{\mr{u}})^2\left(\frac{\gamma_1(s^2 \mbb{B}^{d\times d}_0 \cap r_2^2 \mbb{B}^{d\times d}_{\mr{F}},\|\cdot\|_{\mr{F}})}{n}+\frac{\gamma_2(s^2 \mbb{B}^{d\times d}_0 \cap r_2^2 \mbb{B}^{d\times d}_{\mr{F}},\|\cdot\|_{\mr{F}})}{\sqrt{n}}+r_\delta^2r^2_2+r_\delta r_2^2\right), \nonumber
		\end{align}
		and the proof is complete.
	\end{proof}

	\subsection{Proof of Proposition \ref{p:cwpre}}
	First, we prove \eqref{ine:cwpre'} in Proposition \ref{p:cwpre} via Lemma \ref{l:tg-eval}, that is necessary to calculate  $\gamma_1(s^2 \mbb{B}^{d\times d}_0 \cap r_2^2 \mbb{B}^{d\times d}_{\mr{F}},\|\cdot\|_{\mr{F}})$ and $\gamma_2(s^2 \mbb{B}^{d\times d}_0 \cap r_2^2 \mbb{B}^{d\times d}_{\mr{F}},\|\cdot\|_{\mr{F}})$.
	They are provided in the following lemma, that is proved  in the  Appendix. 
	\begin{lemma}
	\label{l:tg-eval}
		Suppose that (i) and (ii) of Assumption \ref{a:intro} holds.  Then, we have 
		\begin{align}
			\gamma_1(s^2 \mbb{B}^{d\times d}_0 \cap r_2^2 \mbb{B}^{d\times d}_{\mr{F}},\|\cdot\|_{\mr{F}}) &\leq Cs^2r_2^2 \log (d/s)\nonumber \\
			\gamma_2(s^2 \mbb{B}^{d\times d}_0 \cap r_2^2 \mbb{B}^{d\times d}_{\mr{F}},\|\cdot\|_{\mr{F}}) &\leq Csr_2^2 \sqrt{\log (d/s)},
		\end{align}
	\end{lemma}
	First, we prove \eqref{ine:cwpre'}. 
	From  Lemmas \ref{l:gc} and \ref{l:tg-eval}, we have, with probability at least $1-\delta$,
	\begin{align}
	\label{cons}
	\max_{M\in \mathfrak{M}_{r_1,r_2,d}^{\ell_1,\mr{Tr}} }\left| \frac{1}{n}\sum_{i=1}^n\langle \vecx_i \vecx_i^\top-\Sigma, M\rangle\right| &\leq C(L\kappa_{\mr{u}})^2\left(\frac{s^2r_2^2 \log (d/s)}{n}+\frac{sr_2^2 \sqrt{\log (d/s)}}{\sqrt{n}}+r_\delta^2r^2_2+r_\delta r_2^2\right)\\
	\label{cons2}
	&\stackrel{(a)}{\leq} C(L\kappa_{\mr{u}})^2\left(sr_{d,s}+r_\delta \right)r_2^2,
	\end{align}
	where (a) follows from $sr_{d,s},r_\delta\leq 1$.

	Next, we prove \eqref{ine:cwpre2}. Define 
	\begin{align}
		\label{ine:weight}
		w_i^\circ = \begin{cases}
		\frac{1}{n(1-o/n)}&i\in\mc{I} \\
		0 & i\in\mc{O}
		\end{cases}.
	\end{align}
	From the optimality of $\{\hat{w}_i\}_{i=1}^n$ and the fact that $\{w_i^\circ\}_{i=1}^n \in\Delta^{n-1}(\varepsilon)$ and the definition of $\vecX_i$,
	we have
	\begin{align}
	\label{ine:optM0-pre}
		\max_{M\in \mathfrak{M}_{r_1,r_2,d}^{\ell_1,\mr{Tr}} }\sum_{i =1 }^n \hat{w}_i\langle \vecX_i\vecX_i^\top-\Sigma, M\rangle 
		&\leq \max_{M\in \mathfrak{M}_{r_1,r_2,d}^{\ell_1,\mr{Tr}} }\sum_{i =1 }^n w_i^\circ \langle \vecX_i\vecX_i^\top-\Sigma, M\rangle \nonumber\\
		&= \max_{M\in \mathfrak{M}_{r_1,r_2,d}^{\ell_1,\mr{Tr}} }\sum_{i =1 }^n w_i^\circ \langle \vecx_i\vecx_i^\top-\Sigma, M\rangle\nonumber \\
		&\leq \max_{M\in \mathfrak{M}_{r_1,r_2,d}^{\ell_1,\mr{Tr}} }\left|\sum_{i =1 }^n w_i^\circ \langle \vecx_i\vecx_i^\top-\Sigma, M\rangle\right|.
	\end{align}
	From triangular inequality and the fact that $o/n\leq 1/2$, we have
	\begin{align}
	\label{ine:optM0-pre2}
		\left|\sum_{i =1 }^n w_i^\circ \langle \vecx_i\vecx_i^\top-\Sigma, M\rangle\right| &=\frac{1}{1-o/n}\left|\sum_{i =1 }^n \frac{\langle \vecx_i\vecx_i^\top-\Sigma, M\rangle}{n}-\sum_{i \in \mc{O}} \frac{\langle \vecx_i\vecx_i^\top-\Sigma, M\rangle}{n} \right|\nonumber\\
		&\leq 2\left|\sum_{i =1 }^n \frac{\langle \vecx_i\vecx_i^\top-\Sigma, M\rangle}{n}-\sum_{i \in \mc{O}} \frac{\langle \vecx_i\vecx_i^\top-\Sigma, M\rangle}{n} \right|\nonumber\\
		&\leq 2\left|\sum_{i =1 }^n \frac{\langle \vecx_i\vecx_i^\top-\Sigma, M\rangle}{n}\right|+2\left|\sum_{i \in \mc{O}} \frac{\langle \vecx_i\vecx_i^\top-\Sigma, M\rangle}{n} \right|.
	\end{align}
	From \eqref{ine:optM0-pre} and \eqref{ine:optM0-pre2}, we have
	\begin{align}
	\label{ine:optM0}
		&\max_{M\in \mathfrak{M}_{r_1,r_2,d}^{\ell_1,\mr{Tr}} }\sum_{i =1 }^n \hat{w}_i\langle \vecX_i\vecX_i^\top-\Sigma, M\rangle \nonumber \\
		& \hspace*{10mm} \leq 2 \max_{M\in \mathfrak{M}_{r_1,r_2,d}^{\ell_1,\mr{Tr}} }\left(\left|\sum_{i =1 }^n \frac{\langle \vecx_i\vecx_i^\top-\Sigma, M\rangle}{n}\right|+\left|\sum_{i \in \mc{O}} \frac{\langle \vecx_i\vecx_i^\top-\Sigma, M\rangle}{n} \right|\right)\nonumber \\
		& \hspace*{10mm} \leq 2 \left(\max_{M\in \mathfrak{M}_{r_1,r_2,d}^{\ell_1,\mr{Tr}} }\left|\sum_{i =1 }^n \frac{\langle \vecx_i\vecx_i^\top-\Sigma, M\rangle}{n}\right|+\max_{M\in \mathfrak{M}_{r_1,r_2,d}^{\ell_1,\mr{Tr}} }\left|\sum_{i \in \mc{O}} \frac{\langle \vecx_i\vecx_i^\top-\Sigma, M\rangle}{n} \right|\right)\nonumber \\
		& \hspace*{10mm} \leq 2\left(\max_{M\in \mathfrak{M}_{r_1,r_2,d}^{\ell_1,\mr{Tr}} }\left|\sum_{i =1 }^n \frac{\langle \vecx_i\vecx_i^\top-\Sigma, M\rangle}{n}\right|+\max_{|\mc{J}|=o}\max_{M\in \mathfrak{M}_{r_1,r_2,d}^{\ell_1,\mr{Tr}} }\left|\sum_{i \in \mc{J}}\frac{ \langle \vecx_i\vecx_i^\top-\Sigma, M\rangle}{n}\right|\right),
	\end{align}
	We evaluate the last term of \eqref{ine:optM0} in a  manner similar to the proof of Lemma 5 of \cite{DalMin2022All}.
	From \eqref{cons}, we have
	\begin{align}
	\label{ine:optM3}
	 	\max_{|\mc{J}|=o}\max_{M\in \mathfrak{M}_{r_1,r_2,d}^{\ell_1,\mr{Tr}} }\left|\sum_{i \in \mc{J}} \frac{\langle \vecx_i\vecx_i^\top-\Sigma, M\rangle}{o} \right|\geq C(L\kappa_{\mr{u}})^2\left(s^2\frac{\log  (d/s)}{o}+s\sqrt{\frac{\log  (d/s)}{o}}+\sqrt{\frac{t}{o}}+\frac{t}{o}\right)r_2^2,
	\end{align}
	with probability at least 
	\begin{align}
		&\leq \binom{n}{o}\times \nonumber \\
		&\mbb{P}\left[\max_{M\in \mathfrak{M}_{r_1,r_2,d}^{\ell_1,\mr{Tr}} }\left|\sum_{i =1}^o \frac{\langle \vecz_i\vecz_i^\top-\Sigma, M\rangle}{o} \right|\geq C(L\kappa_{\mr{u}})^2\left(s^2\frac{\log  (d/s)}{o}+s\sqrt{\frac{\log (d/s)}{o}}+\sqrt{\frac{t}{o}}+\frac{t}{o}\right)r_2^2 \right]\nonumber \\
		&\leq \binom{n}{o} e^{-t},
	\end{align}
	where $\{\vecz_i\}_{i=1}^o$ is a sequence of i.i.d. random vectors sampled from the same distribution as $\{\vecx_i\}_{i=1}^n$. 
	Let $t = o\log (ne/o)+\log(1/\delta)$. We have
	\[ 
		\binom{n}{o} e^{-t} \le \frac{\prod_{k=0}^{o-1} (n-k)}{o!} \left(\frac{o}{ne}\right)^o \delta = \prod_{k=0}^{o-1} \frac{n-k}{n}\frac{o^o}{o! e^o} \delta \le \delta,
	\]
	where the last inequality follows from Stirling's formula $o^o \le o! e^o$. From \eqref{ine:optM3}, with probability at least $1-\delta$, we have
	\begin{align}
	\label{ine:optM4}
		\lefteqn{ \max_{|\mc{J}|=o}\max_{M\in \mathfrak{M}_{r_1,r_2,d}^{\ell_1,\mr{Tr}} }\left|\sum_{i \in \mc{J}} \frac{1}{n} \langle \vecx_i\vecx_i^\top-\Sigma, M\rangle \right| } \nonumber \\
		& \hspace*{3mm}  \leq C \frac{o}{n}(L\kappa_{\mr{u}})^2\left( s^2\frac{\log (d/s)}{o}+s\sqrt{\frac{\log (d/s)}{o}}+\sqrt{\frac{o\log (ne/o)+\log(1/\delta)}{o}}+\frac{o\log (ne/o)+\log(1/\delta)}{o}\right)r_2^2.
	\end{align}
	From $e \le n/o$ and $r_\delta \le 1$, we have
	\begin{align}
		\frac{o\log (ne/o)+\log(1/\delta)}{n} &\le 2 \frac{o}{n}\log(n/o) + \frac{1}{n}\log(1/\delta) =2 r_o' + r_\delta^2 \le 2 r_o' + r_\delta, \nonumber\\
		\sqrt{\frac{o}{n}}\sqrt{\frac{o\log (ne/o)+\log(1/\delta)}{n}} & \le \sqrt{\frac{o}{n}} \sqrt{2 \frac{o}{n}\log(n/o)} + \sqrt{\frac{o}{n}} \sqrt{\frac{1}{n}\log(1/\delta)} \le 2 r_o' + r_\delta. 
	\end{align}
	Combining \eqref{ine:optM4} with the above two inequalities, with probability at least $1-\delta$, we have
	\begin{align}
	\label{ine:optM5}
 		\max_{|\mc{J}|=o}\max_{M\in \mathfrak{M}_{r_1,r_2,d}^{\ell_1,\mr{Tr}} }\left|\sum_{i \in \mc{J}} \frac{1}{n} \langle \vecx_i\vecx_i^\top-\Sigma, M\rangle \right| 
 		&\leq C (L\kappa_{\mr{u}})^2\left(s^2\frac{\log (d/s)}{n}+s\sqrt{\frac{\log (d/s)}{n}}\sqrt{\frac{o}{n}}+4r_o'+2r_\delta\right)r_2^2\nonumber\\
		& \stackrel{(a)}{\leq} C (L\kappa_{\mr{u}})^2\left(2s^2\frac{\log (d/s)}{n}+\frac{o}{n}+4r_o'+2r_\delta\right)r_2^2 \nonumber \\
		& \stackrel{(b)}{\leq} C (L\kappa_{\mr{u}})^2\left(s^2\frac{\log (d/s)}{n}+5r_o'+2r_\delta\right)r_2^2,
	\end{align}
	where (a) follows from Young's inequality, and  (b) follows from $0<o/n\leq 1/(5e)$ and $o/n\leq r_o'$.
	Combining  \eqref{cons}, \eqref{ine:optM0} and \eqref{ine:optM5}, with probability at least $1-2\delta$, we have
	\begin{align}
	\label{ine:unknownmainlemma}
		\max_{M\in \mathfrak{M}_{r_1,r_2,d}^{\ell_1,\mr{Tr}} }\sum_{i =1 }^n \hat{w}_i\langle \vecX_i\vecX_i^\top-\Sigma, M\rangle 
		&\leq C(L\kappa_{\mr{u}})^2\left(sr_{d,s}+r_\delta +s^2\frac{\log (d/s)}{n}+5r_o'+2r_\delta\right)r_2^2\nonumber\\
		&\stackrel{(a)}{\leq} C(L\kappa_{\mr{u}})^2\left(sr_{d,s}+r_o'+r_\delta\right)r_2^2,
	\end{align}
	where (a) follows from  $sr_{d,s}\leq 1$,
	and the proof is complete.

	\subsection{Proof of  Proposition \ref{p:main1}}
		In this proof, we define $c_1$ and $c_2$ as the numerical constants that are the same ones of Theorem 4.4 of \cite{Men2016Upper}.
		Applying Theorem 4.4 of \cite{Men2016Upper},  with probability at least $1-2\exp(-c_1u^2 2^{s_0}) - 2\exp(-c_1 w^2 n)$, we have 
		 \begin{align}
		 \label{gc-pre}
			 \sup_{\vecv \in r_1 \mbb{B}^d_1 \cap r_\Sigma \mbb{B}^d_\Sigma }\left|\frac{1}{n} \sum_{i=1}^n \langle h(\xi_{\lambda_o,i})\vecx_i,\vecv\rangle \right|&\leq c_2uw\frac{\|h\left(\xi_{\lambda_o,i}\right)\|_{\psi_2}}{\sqrt{n}}\tilde{\Lambda}_{s_0,u}\left(\langle \cdot, \vecv\rangle,\,\vecv \in  r_1 \mbb{B}^d_1 \cap r_\Sigma \mbb{B}^d_\Sigma\right)\nonumber\\
			 &\stackrel{(a)}{\leq} c_2uw\frac{1}{\sqrt{n}}\tilde{\Lambda}_{s_0,u}\left(\langle \cdot, \vecv\rangle,\,\vecv \in  r_1 \mbb{B}^d_1 \cap r_\Sigma \mbb{B}^d_\Sigma\right)\nonumber\\
			 &\stackrel{(b)}{\leq } CL\frac{c_2uw}{\sqrt{n}}\left(\sup_{\vecv \in r_1 \mbb{B}^d_1 \cap r_\Sigma \mbb{B}^d_\Sigma } \langle \Sigma^\frac{1}{2}\vecg, \vecv\rangle+2^\frac{s_0}{2}\sup_{\vecv \in r_1 \mbb{B}^d_1 \cap r_\Sigma \mbb{B}^d_\Sigma}\|\vecv^\top \vecx\|_{L_2} \right)\nonumber\\
			 &\leq C \frac{L}{\sqrt{n}}\left(\mbb{E}\sup_{\vecv \in r_1 \mbb{B}^d_1 \cap r_\Sigma \mbb{B}^d_\Sigma } \langle \Sigma^\frac{1}{2}\vecg, \vecv\rangle+2^\frac{s_0}{2} r_\Sigma \right),
		 \end{align}
		 where (a) follows from  $\|h\left(\xi_{\lambda_o,i}\right)\|_{\psi_2}\leq 1$, and (b) follows from the same argument of Section 1 of \cite{Men2016Upper}.
	
		 Define $s'$ as the minimum integer such that $2^{s'}\geq 2\log(1/\delta)$. Set $s_0, w$ and $u$ such that $s_0=s'$ and $c_1u^2, c_1 w^2 \geq 1$. Then, from Lemma \ref{l:gwslope:main}, we have
		 \begin{align}
			\label{gc-pre5}
				\sup_{\vecv \in r_1 \mbb{B}^d_1 \cap r_\Sigma \mbb{B}^d_\Sigma }\left|\frac{1}{n} \sum_{i=1}^n \langle h(\xi_{\lambda_o,i})\vecx_i,\vecv\rangle \right| 
				&\stackrel{(a)}{\leq} C \frac{L}{\sqrt{n}}\left(\rho r_1\sqrt{ \log (d/s) }+\sqrt{\log(1/\delta)} r_\Sigma \right),
			\end{align}
		 with probability at least 
		 \begin{align}
				1-2\exp(-c_1u^2 2^{s_0}) - 2\exp(-c_1 w^2 n) \geq 1-2\exp(-2^{s_0}) - 2\exp(-n) \geq 1-\delta,
			\end{align}
			where we use  $\sqrt{2}r_\delta\leq 1$.

\section{Acknowledgement}
	This work was supported by JSPS KAKENHI Grant Number 17K00065.
	\vskip 0.2in

\bibliography{ORSELRC} 
\newpage
\appendix

\begin{center} \textbf{Appendix} \end{center}
In the appendices, 
the value of the numerical constant $C$  be allowed to change from line to line.

\section{Proof of Theorem \ref{t:main}}
\label{sec:pmt}
Suppose that the assumptions in Theorem \ref{t:main} hold.
To prove Theorem \ref{t:main}, it is sufficient that we confirm  \eqref{ine:det:main:0-1} - \eqref{ine:det:main:0-4}  in Proposition \ref{p:main} hold with high probability, as described already, and  then we obtain the result \eqref{ine:result1} with concrete values of $r_{a,1}, r_{a,2}, r_{a,\Sigma}, r_{b,1}, r_{b,2}, r_{b,\Sigma}, b$.  Obviously, we have $r_1=c_{r_1}\sqrt{s} r_\Sigma$ and $r_2 = c_{r_2}r_\Sigma$ in \eqref{ine:det:main:0-4} from \eqref{ine:r121}. In the following, we show the remaining.
We see that, under the assumptions in Theorem \ref{t:main}, Propositions \ref{p:cwpre}, \ref{p:main1} and \ref{p:main:sc} hold immediately with probability at least $1-4\delta$, and then, Propositions \ref{p:main:out} and \ref{p:main:out2} also hold because \eqref{ine:cwpre'} and \eqref{ine:cwpre2} in  Proposition \ref{p:cwpre}  hold and Algorithm \ref{alg:cw0} returns $\hat{w}$ from the definition of $\tau_{\mr {cut}}$.

	\subsection{Confirmation of \eqref{ine:det:main:0-1}}
	From the following  lemma, we can confirm \eqref{ine:det:main:0-1}.
	The proofs are given in the  Appendix \ref{sec:calc}. 
	\begin{lemma}
	\label{1calcknown}
	Assume that Propositions \ref{p:cwpre}, \ref{p:main1}-\ref{p:main:sc} hold.
		For any  $ \vectheta_\eta \in r_1 \mbb{B}^d_1 \cap r_2 \mbb{B}^d_2 \cap r_\Sigma \mbb{B}^d_\Sigma $, we have
		\begin{align}
		\label{ine:calcknown1}
			\left|\sum_{i=1}^n \hat{w}_i'h(r_{\vecbeta^*,i}) \vecX_i^\top \vectheta_\eta \right| \leq  3c_{\max}^2 L\left(\rho r_{d,s}r_1+r_\delta r_\Sigma+\kappa_{\mr{u}}\sqrt{\frac{o}{n}(sr_{d,s}+r_\delta)}r_2+\kappa_{\mr{u}}r_or_2\right).
		\end{align}
	\end{lemma}
	From Lemma \ref{1calcknown}, we see that \eqref{ine:det:main:0-1}
	holds with
	\begin{align}
	\label{confirmknown}
		r_{a,1} &= 3c_{\max}^2L \lambda_o\sqrt{n}\rho r_{d,s},\quad r_{a,2} =3c_{\max}^2L \lambda_o\sqrt{n}\left(\kappa_{\mr{u}}\sqrt{\frac{o}{n}(sr_{d,s}+r_\delta)}+\kappa_{\mr{u}}r_o \right),\nonumber\\
		r_{a,\Sigma} &=3c_{\max}^2L \lambda_o\sqrt{n}r_\delta.
	\end{align}

	\subsection{Confirmation of  \eqref{ine:det:main:0-3}}
	From \eqref{confirmknown},
	\begin{align}
		(C_s:=)r_{a,1}+ \frac{c_{r_2}r_{a,2}+ r_{a,\Sigma}}{c_{r_1}\sqrt{s}} = 3c_{\max}^2L\lambda_o\sqrt{n}\times\frac{1}{c_{r_1}\sqrt{s}}R_{d,n,o}.
	\end{align}
	From the definition of $\lambda_s$,
	\begin{align}
		\frac{\lambda_s}{C_s } \geq \frac{c_s c_{\max}^2L\lambda_o\sqrt{n}\frac{1}{c_{r_1}\sqrt{s}}R_{d,n,o}}{3c_{\max}^2L\lambda_o\sqrt{n}\frac{1}{c_{r_1}\sqrt{s}}R_{d,n,o}} = \frac{c_s}{3}\ge \frac{c_{\mr{RE}}+1}{c_{\mr{RE}}-1}>0.
	\end{align}
	Hence, we have $\lambda_s-C_s>0$ and 
	\begin{align}
		\frac{\lambda_s+C_s}{\lambda_s-C_s}=\frac{1+\frac{C_s}{\lambda_s}}{1-\frac{C_s}{\lambda_s}} \leq \frac{1+\frac{c_{\mr{RE}}-1}{c_{\mr{RE}}+1}}{1-\frac{c_{\mr{RE}}-1}{c_{\mr{RE}}+1}}= c_{\mr{RE}}.
	\end{align} 
	Therefore, we see that  \eqref{ine:det:main:0-3} holds.

	\subsection{Confirmation of  \eqref{ine:det:main:0-2}}
	From the following  lemma, we can confirm \eqref{ine:det:main:0-2}.
	The proofs are given in the  Appendix \ref{sec:calc}. 
	\begin{lemma}
	\label{3calcknown}
	Assume that Propositions \ref{p:cwpre}, \ref{p:main1}-\ref{p:main:sc} hold.
	For any  $ \vectheta_\eta \in r_1 \mbb{B}^d_1 \cap r_2 \mbb{B}^d_2 \cap r_\Sigma \mbb{B}^d_\Sigma$, we have
	\begin{align}
		\label{ine:calcknown3}
			&\sum_{i=1}^n \lambda_o\sqrt{n} \hat{w}_i'\left(-h(r_{\vecbeta^*+\vectheta_\eta,i}) +h(r_{\vecbeta^*,i})\right) \vecX_i^\top \vectheta_\eta \nonumber\\
			&\geq \frac{\|\Sigma^\frac{1}{2}\vecv\|_2^2}{3}-c_{\max}^2L\lambda_o\sqrt{n}\left(\rho r_{d,s} r_1+r_\delta r_\Sigma+\kappa_{\mr{u}}\sqrt{\frac{o}{n}(sr_{d,s}+r_\delta)}r_2+\kappa_{\mr{u}}r_or_2\right).
	\end{align}
	\end{lemma}

	From Lemma \ref{3calcknown}, we see that \eqref{ine:det:main:0-2} holds with 
	\begin{align}
		b &= \frac{1}{3},\quad r_{b,1} = \frac{r_{a,1}}{3},\quad r_{b,2} = \frac{r_{a,2}}{3}, \quad r_{b,\Sigma} = \frac{r_{a,\Sigma}}{3}.
	\end{align}

	\subsection{Confirmation of  \eqref{ine:det:main:0-4}}
	As we mentioned at the beginning of Section \ref{sec:pmt}, $r_1 = c_{r_1}\sqrt{s}r_\Sigma$ and $r_2 = c_{r_2}r_{\Sigma}$ is clear from their definitions.
	We confirm $r_\Sigma$. We see that 
	\begin{align}
		&\frac{2}{b} \left(c_{r_1}\sqrt{s}(r_{a,1}+r_{b,1})+
		c_{r_2}(r_{a,2}+r_{b,2}) +r_{a,\Sigma}+r_{b,\Sigma} + c_{r_1}\sqrt{s}\lambda_s\right)\nonumber \\
		&\stackrel{(a)}\leq 6c_{\max}^2 L\lambda_o\sqrt{n}\left(4 + c_s  \right) R_{d,n,o}\nonumber\\
		&\leq r_\Sigma,
	\end{align}
 and the proof is complete.

\section{Proof of Theorem \ref{t:main2}}
\label{asec:pmt}
The strategy of this proof is the same as the one of Theorem \ref{t:main}.
Suppose that the assumptions in Theorem \ref{t:main2} hold.
To prove Theorem \ref{t:main2},  it is sufficient that we confirm  \eqref{ine:det:main:0-1} - \eqref{ine:det:main:0-4}  in Proposition \ref{p:main} hold with high probability, as described already, and  then we obtain the result \eqref{ine:result:un} with concrete values of $r_{a,1}, r_{a,2}, r_{a,\Sigma}, r_{b,1}, r_{b,2}, r_{b,\Sigma}, b$.  Obviously, we have $r_1=c_{r_1}\sqrt{s} r_\Sigma$ and $r_2 = c_{r_2}r_\Sigma$ in \eqref{ine:det:main:0-4} from \eqref{ine:r121}. In the following, we show the remaining.
We see that, under the assumptions in Theorem \ref{t:main2}, Corollary \ref{c:cwpre},  Propositions \ref{p:main1} and \ref{p:main:sc} hold immediately with probability at least $1-3\delta$, and then, Propositions \ref{p:main:out-un} and \ref{p:main:out2-un} also hold because \eqref{ine:cwpre2-1} and \eqref{ine:cwpre2-2} in  Corollary \ref{c:cwpre}  hold, and Algorithm \ref{alg:cw0-un} returns $\hat{w}$ from the definition of $\tau_{\mr {cut}}$.

\subsection{Confirmation of  \eqref{ine:det:main:0-1}}
From the following  lemma, we can confirm \eqref{ine:det:main:0-1}.
The proofs are given by Appendix \ref{sec:calc}. 
\begin{lemma}
\label{1calcunknown}
	Assume that Corollary \ref{c:cwpre}, Propositions \ref{p:main1}, \ref{p:main:sc}, \ref{p:main:out-un} and \ref{p:main:out2-un} hold. For any  $ \vectheta_\eta \in r_1 \mbb{B}^d_1 \cap r_2 \mbb{B}^d_2 \cap r_\Sigma \mbb{B}^d_\Sigma$, we have
	\begin{align}
	\label{ine:calckunnown1}
		\left|\sum_{i=1}^n \hat{w}_i'h(r_{\vecbeta^*,i}) \vecX_i^\top \vectheta_\eta \right| \leq  3c'^2_{\max} L\left(\rho r_{d,s}r_1+r_\delta r_\Sigma+\sqrt{\frac{o}{n}}\sqrt{\kappa_{\mr{u}}^2\left(sr_{d,s}+r_\delta \right)+\Sigma_{\max}^2}r_2\right).
	\end{align}
\end{lemma}
From Lemma \ref{1calcknown}, we see that \eqref{ine:det:main:0-1}
holds with
\begin{align}
\label{confirmunknown}
	r_{a,1} &= 3c'^2_{\max} L \lambda_o\sqrt{n}\rho r_{d,s},\quad r_{a,2} =3c'^2_{\max} L \lambda_o\sqrt{n}\sqrt{\frac{o}{n}}\sqrt{\kappa_{\mr{u}}^2\left(sr_{d,s}+r_\delta \right)+\Sigma_{\max}^2},\nonumber\\
	r_{a,\Sigma} &=3c'^2_{\max} L \lambda_o\sqrt{n}r_\delta.
\end{align}

\subsection{Confirmation of  \eqref{ine:det:main:0-3}}
From \eqref{confirmunknown},
\begin{align}
	(C_s:=)r_{a,1}+ \frac{c_{r_2}r_{a,2}+ r_{a,\Sigma}}{c_{r_1}\sqrt{s}} = 3c'^2_{\max} L\lambda_o\sqrt{n}\times \frac{1}{c_{r_1}\sqrt{s}}R_{d,n,o}'.
\end{align}
From the definition of $\lambda_s$,
\begin{align}
	\frac{\lambda_s}{C_s} \geq \frac{c_s c'^2_{\max} L\lambda_o\sqrt{n}\frac{1}{c_{r_1}\sqrt{s}}R_{d,n,o}'}{3c'^2_{\max} L\lambda_o\sqrt{n}\frac{1}{c_{r_1}\sqrt{s}}R_{d,n,o}'} = \frac{c_s}{3}\ge \frac{c_{\mr{RE}}+1}{c_{\mr{RE}}-1}>0.
\end{align}
Hence, we have $\lambda_s-C_s>0$ and 
\begin{align}
	\frac{\lambda_s+C_s}{\lambda_s-C_s}=\frac{1+\frac{C_s}{\lambda_s}}{1-\frac{C_s}{\lambda_s}} \leq \frac{1+\frac{c_{\mr{RE}}-1}{c_{\mr{RE}}+1}}{1-\frac{c_{\mr{RE}}-1}{c_{\mr{RE}}+1}}= c_{\mr{RE}}.
\end{align} 
Therefore, we see that  \eqref{ine:det:main:0-3} holds.

\subsection{Confirmation of  \eqref{ine:det:main:0-2}}
From the following  lemma, we can confirm \eqref{ine:det:main:0-2}.
The proofs are given by Appendix \ref{sec:calc}. 
\begin{lemma}
\label{3calcunknown}
Assume that Corollary \ref{c:cwpre}, Propositions \ref{p:main1}, \ref{p:main:sc}, \ref{p:main:out-un} and \ref{p:main:out2-un} hold. 
For any  $ \vectheta_\eta \in r_1 \mbb{B}^d_1 \cap r_2 \mbb{B}^d_2 \cap r_\Sigma \mbb{B}^d_\Sigma$, we have
\begin{align}
	\label{ine:calcunknown3}
		&\sum_{i=1}^n \lambda_o\sqrt{n} \hat{w}_i'\left(-h(r_{\vecbeta^*+\vectheta_\eta,i}) +h(r_{\vecbeta^*,i})\right) \vecX_i^\top \vectheta_\eta \nonumber\\
		& \hspace*{10mm} \geq \frac{\|\Sigma^\frac{1}{2}\vecv\|_2^2}{3}-c'^2_{\max} L\lambda_o\sqrt{n}\left(\rho r_{d,s} r_1+r_\delta r_\Sigma+\sqrt{\frac{o}{n}}\sqrt{\kappa_{\mr{u}}^2\left(sr_{d,s}+r_\delta \right)+\Sigma_{\max}^2}r_2\right).
\end{align}
\end{lemma}

From Lemma \ref{3calcunknown}, we see that \eqref{ine:det:main:0-2} holds with 
\begin{align}
	b &= \frac{1}{3},\quad r_{b,1} = \frac{r_{a,1}}{3},\quad r_{b,2} = \frac{r_{a,2}}{3}, \quad r_{b,\Sigma} = \frac{r_{a,\Sigma}}{3}.
\end{align}

\subsection{Confirmation of  \eqref{ine:det:main:0-4}}
As we mentioned at the beginning of Appendix \ref{asec:pmt}, $r_1 = c_{r_1}\sqrt{s}r_\Sigma$ and $r_2 = c_{r_2}r_{\Sigma}$ is clear from their definitions.
We confirm $r_\Sigma$. We see that 
\begin{align}
	&\frac{2}{b} \left(c_{r_1}\sqrt{s}(r_{a,1}+r_{b,1})+
	c_{r_2}(r_{a,2}+r_{b,2}) +r_{a,\Sigma}+r_{b,\Sigma} + c_{r_1}\sqrt{s}\lambda_s\right)\nonumber \\
	& \hspace*{10mm}  \leq 6c'^2_{\max} L\lambda_o\sqrt{n}\left(4 + c_s  \right)\left(\rho c_{r_1}\sqrt{s} r_{d,s} +r_\delta +c_{r_2}\sqrt{\frac{o}{n}}\sqrt{\kappa_{\mr{u}}^2\left(sr_{d,s}+r_\delta \right)+\Sigma_{\max}^2}\right)\nonumber\\
	& \hspace*{10mm}  \leq r_\Sigma,
\end{align}
and the proof is complete. 

\section{Preparation of proof of Proposition \ref{p:main}}
\label{sec:mainproposition-pre}
We introduce the following four lemmas, that are used in the proof of Proposition \ref{p:main}.
\begin{lemma}
\label{l:mainpre1-1}
 Suppose that \eqref{ine:det:main:0-1}, \eqref{ine:det:main:0-3}, $\|\vectheta_\eta\|_1\leq  r_1$,  $\|\vectheta_\eta\|_2= r_2$ and  $\|\Sigma^\frac{1}{2}\vectheta_\eta\|_2\leq r_\Sigma$, where $r_1 = c_{r_1}\sqrt{s}r_\Sigma $ and  $r_2 =c_{r_2} r_\Sigma $  hold.
Then, for any fixed $\eta\in(0,1)$, we have
\begin{align}
 \|\vectheta_\eta\|_2  &\leq \frac{3+c_{\mr{RE}}}{\kappa_{\mr{l}}}\|\Sigma^\frac{1}{2}\vectheta_\eta\|_2.
\end{align}
\end{lemma}
\begin{lemma}
	\label{l:mainpre1-2}
	 Suppose that \eqref{ine:det:main:0-1}, \eqref{ine:det:main:0-3}, $\|\vectheta_\eta\|_1 = r_1$, $\|\vectheta_\eta\|_2\leq  r_2$ and $\|\Sigma^\frac{1}{2}\vectheta_\eta\|_2\leq  r_\Sigma$, where $r_1 = c_{r_1}\sqrt{s}r_\Sigma $ and  $r_2 =c_{r_2} r_\Sigma $ hold.
	Then, for any fixed $\eta\in(0,1)$, we have
	\begin{align}
\|\vectheta_\eta\|_1 &\leq \frac{1+c_{\mr{RE}}}{	\mathfrak{r} } \sqrt{s}\|\Sigma^\frac{1}{2}\vectheta_\eta\|_2.
	\end{align}
	\end{lemma}

	\begin{lemma}
		\label{l:mainpre2-1}
		Suppose that \eqref{ine:det:main:0-1} - \eqref{ine:det:main:0-2},  $\|\vectheta_\eta\|_1\leq  r_1$,  $\|\vectheta_\eta\|_2\leq r_2$ and  $\|\Sigma^\frac{1}{2}\vectheta_\eta\|_2= r_2$, where  $r_1 = c_{r_1}\sqrt{s}r_\Sigma $ and  $r_2 =c_{r_2} r_\Sigma $ hold.
		Then, for any  $\eta\in(0,1)$, we have
		\begin{align}
			\|\Sigma^\frac{1}{2}\vectheta_\eta\|_2 &\leq \frac{1}{b} \left(c_{r_1}\sqrt{s}(r_{a,1}+r_{b,1})+
			c_{r_2}(r_{a,2}+r_{b,2}) +r_{a,\Sigma}+r_{b,\Sigma} + c_{r_1}\sqrt{s}\lambda_s\right).
		\end{align}
	\end{lemma}

\subsection{Proof of Lemma \ref{l:mainpre1-1}}
For a vector $\vecv=(v_1,\cdots,v_d) $, define $\{v^{\sharp}_1,\cdots,v^{\sharp}_d\}$ as a non-increasing rearrangement of $\{|v_1|,\cdots,|v_d|\}$, and $\vecv^{\sharp}\in \mbb{R}^d$ as a vector such that $\vecv^{\sharp}|_i = v^{\sharp}_i$. 
For the sets $S_1=\{1,...,s\}$ and $S_2=\{s+1,...,d\}$, let $v^{\sharp1}=v^\sharp_{S_1}$ and $v^{\sharp2}=v^\sharp_{S_2}$.

In Section \ref{sec:case1:re}, we have
\begin{align}
	\|\vectheta_\eta\|_2 \leq \frac{\sqrt{1+c_{\mr{RE}}}}{\kappa_{\mr{l}}} \|\Sigma^\frac{1}{2}\vectheta_\eta\|_2
\end{align}
assuming $\|\vectheta_\eta \|_2 \leq \|\vectheta_\eta\|_1/\sqrt{s}$, and in  Section \ref{sec:case2:re}, we have
\begin{align}
	\|\vectheta_\eta\|_2 \leq  2\|\vectheta_\eta^{\sharp1}\|_2\leq \frac{2}{\kappa_{\mr{l}}}\|\Sigma^\frac{1}{2}\vectheta_\eta\|_2
\end{align}
assuming $\|\vectheta_\eta \|_2 \geq \|\vectheta_\eta\|_1/\sqrt{s}$. From the above two inequalities, we have 
\begin{align}
	\|\vectheta_\eta\|_2  &\leq \frac{3+c_{\mr{RE}}}{\kappa_{\mr{l}}}\|\Sigma^\frac{1}{2}\vectheta_\eta\|_2.
 \end{align}

	\subsubsection{Case I}
	\label{sec:case1:re}
	In Section \ref{sec:case1:re}, suppose that $\|\vectheta_\eta \|_2 \leq \|\vectheta_\eta\|_1/\sqrt{s}$.
	Let 
	\begin{align}
	 Q'(\eta) = \lambda_o\sqrt{n} \hat{w}_i'\sum_{i=1}^n \left(-h (r_{\vecbeta^*+\vectheta_\eta,i})+h(r_{\vecbeta^*,i}) \right)\vecX_i^\top \vectheta.
	 \end{align}
	From the proof of Lemma F.2. of \cite{FanLiuSunZha2018Lamm}, we have $\eta Q'(\eta) \leq \eta Q'(1)$ and this means
	\begin{align}
	\label{ine:det:lemma:1}
	\sum_{i=1}^n \lambda_o\sqrt{n} \hat{w}_i'\left(-h (r_{\vecbeta^*+\vectheta_\eta,i})+h(r_{\vecbeta^*,i})\right) \vecX_i^\top \vectheta_\eta &\leq \sum_{i=1}^n \lambda_o\sqrt{n} \hat{w}_i'\eta \left(-h (r_{\hat{\vecbeta},i}) +h(r_{\vecbeta^*,i})\right)\vecX_i^\top\vectheta.
	\end{align}
	Let $\partial \vecv$ be the sub-differential of $\|\vecv\|_1$.
	Adding $\eta \lambda_s(\|\hat{\vecbeta}\|_1-\|\vecbeta^*\|_1) $ to both sides of \eqref{ine:det:lemma:1}, we have
	
	\begin{align}
	\label{ine:det:lemma:2}
	&\sum_{i=1}^n \lambda_o\sqrt{n} \hat{w}_i'\left(-h (r_{\vecbeta^*+\vectheta_\eta,i}) +h(r_{\vecbeta^*,i})\right)\vecX_i^\top\vectheta_\eta +\eta \lambda_*(\|\hat{\vecbeta}\|_1-\|\vecbeta^*\|_1)\nonumber \\
	 & \stackrel{(a)}{\leq} \sum_{i=1}^n \lambda_o\sqrt{n} \hat{w}_i'\eta \left(-h (r_{\hat{\vecbeta},i}) +h(r_{\vecbeta^*,i})\right)\vecX_i^\top\hat{\vectheta}+\eta \lambda_s\langle \partial \hat{\vecbeta},\vectheta\rangle\nonumber\\
	 &\stackrel{(b)}{=} \sum_{i=1}^n \lambda_o\sqrt{n} \hat{w}_i' h(r_{\vecbeta^*,i})\vecX_i^\top\vectheta_\eta,
	\end{align}
	where (a) follows from $\|\hat{\vecbeta}\|_1-\|\vecbeta^*\|_1 \leq \langle \partial \hat{\vecbeta},\vectheta\rangle$, which is the definition of the sub-differential, and (b) follows from the optimality of $\hat{\vecbeta}$.

	From the convexity of Huber loss, the left-hand side  of \eqref{ine:det:lemma:2} is positive and we have
	\begin{align}
		\label{ine:det:lemma:3}
	0\leq \sum_{i=1}^n \lambda_o\sqrt{n} \hat{w}_i' h(r_{\vecbeta^*,i})\vecX_i^\top\vectheta_\eta+	\eta \lambda_s(\|\vecbeta^*\|_1-\|\hat{\vecbeta}\|_1).
	\end{align}
From \eqref{ine:det:main:0-1}, the first term of the right-hand side of \eqref{ine:det:lemma:3} is evaluated as
	\begin{align}
		\label{ine:det:lemma:4}
		\sum_{i=1}^n \lambda_o\sqrt{n} \hat{w}_i' h(r_{\vecbeta^*,i})\vecX_i^\top\vectheta_\eta&\leq r_{a,1}r_1+r_{a,2} r_2 + r_{a,\Sigma}r_\Sigma\nonumber \\
		&\stackrel{(a)}{\leq} \left(\frac{c_{r_1}}{c_{r_2}}\sqrt{s}r_{a,1}+r_{a,2} +\frac{1}{c_{r_2}}  r_{a,\Sigma} \right)\|\vectheta_\eta\|_2\nonumber \\
		&\stackrel{(b)}\leq \frac{1}{\sqrt{s}}\left(\frac{c_{r_1}}{c_{r_2}}\sqrt{s}r_{a,1}+r_{a,2} + \frac{1}{c_{r_2}} r_{a,\Sigma}\right)\|\vectheta_\eta\|_1,
	\end{align}
	where (a) follows from  $r_1 = c_{r_1}\sqrt{s}r_\Sigma $ and  $r_2 =c_{r_2} r_\Sigma $  and (b) follows from the assumption of this section, $\|\vectheta_\eta\|_2 \leq \|\vectheta_\eta\|_1/\sqrt{s}$.
	From~\eqref{ine:det:lemma:3},~\eqref{ine:det:lemma:4} and the assumption $\|\vectheta_\eta \|_2 \leq \|\vectheta_\eta\|_1/\sqrt{s}$, we have
	\begin{align}
	0 \leq \left(\frac{r_{a,2}}{\sqrt{s}} + \frac{c_{r_1}\sqrt{s} r_{a,1}+r_{a,\Sigma}}{c_{r_2}\sqrt{s}}  \right)\|\vectheta_\eta\|_1+\eta \lambda_s (\|\vecbeta^*\|_1-\|\hat{\vecbeta}\|_1).
	\end{align}
	Define $\mc{J}_{\veca}$ as the index set of  non-zero entries of $\veca$, and $\vectheta_{\eta, \mc{J}_{\veca}}$ as a vector such that $\vectheta_{\eta, \mc{J}_{\veca}}|_i=\vectheta_\eta|_i$ for $i \in \mc{J}_{\veca}$ and $\vectheta_{\eta, \mc{J}_{\veca}}|_i = 0$ for $i \notin \mc{J}_{\veca}$.
	Furthermore, we see
	\begin{align}
	0
	&\leq \left(\frac{r_{a,2}}{\sqrt{s}} + \frac{c_{r_1}\sqrt{s} r_{a,1}+r_{a,\Sigma}}{c_{r_2}\sqrt{s}}  \right)\|\vectheta_\eta\|_1+\eta \lambda_s(\|\vecbeta^*\|_1-\|\hat{\vecbeta}\|_1) \nonumber \\
	& \leq \left(\frac{r_{a,2}}{\sqrt{s}} + \frac{c_{r_1}\sqrt{s} r_{a,1}+r_{a,\Sigma}}{c_{r_2}\sqrt{s}}  \right) (\|\vectheta_{\eta,\mc{J}_{\vecbeta^*}}\|_1+\|\vectheta_{\eta,\mc{J}^c_{\vecbeta^*}}\|_1) + \eta\lambda_s(\|\vecbeta^*_{\mc{J}_{\vecbeta^*}}- \hat{\vecbeta}_{\mc{J}_{\vecbeta^*}}\|_1-\|\hat{\vecbeta}_{\mc{J}^c_{\vecbeta^*}}\|_1)\nonumber\\
	& =\left(\lambda_s + \left(\frac{r_{a,2}}{\sqrt{s}} + \frac{c_{r_1}\sqrt{s} r_{a,1}+r_{a,\Sigma}}{c_{r_2}\sqrt{s}}  \right)\right)\|\vectheta_{\eta,\mc{J}_{\vecbeta^*}}\|_1
	+\left(-\lambda_s+ \left(\frac{r_{a,2}}{\sqrt{s}} + \frac{c_{r_1}\sqrt{s} r_{a,1}+r_{a,\Sigma}}{c_{r_2}\sqrt{s}}  \right)\right)\|\vectheta_{\eta,\mc{J}^c_{\vecbeta^*}}\|_1.
	\end{align}
	Then, we have
		\begin{align}
			\| \vectheta_{\eta,\mc{J}_{\vecbeta^*}^c}\|_1 \leq \frac{\lambda_s + \left(\frac{r_{a,2}}{\sqrt{s}} + \frac{c_{r_1}\sqrt{s} r_{a,1}+r_{a,\Sigma}}{c_{r_2}\sqrt{s}}  \right)}{\lambda_s - \left(\frac{r_{a,2}}{\sqrt{s}} + \frac{c_{r_1}\sqrt{s} r_{a,1}+r_{a,\Sigma}}{c_{r_2}\sqrt{s}}  \right)}\| \vectheta_{\eta,\mc{J}_{\vecbeta^*}}\|_1\stackrel{(a)}{\leq}\frac{\lambda_s +  \left(r_{a,1}+ \frac{c_{r_2}r_{a,2}+ r_{a,\Sigma}}{c_{r_1}\sqrt{s}}\right)}{\lambda_s -  \left(r_{a,1}+ \frac{c_{r_2}r_{a,2}+ r_{a,\Sigma}}{c_{r_1}\sqrt{s}}\right) }\| \vectheta_{\eta,\mc{J}_{\vecbeta^*}}\|_1 \stackrel{(b)}{\leq} c_{\mr{RE}} \| \vectheta_{\eta,\mc{J}_{\vecbeta^*}}\|_1,
		\end{align}
		where (a) follows from the fact that $c_{r_2}\geq c_{r_1}$ and (b) follows from \eqref{ine:det:main:0-3}, and
		from the definition of $\| \vectheta_\eta^{\sharp2}\|_1 $ and $\| \vectheta_\eta^{\sharp1}\|_1$, we have
		\begin{align}
			\label{ine:det:lemma:5}
			\| \vectheta_\eta^{\sharp2}\|_1 \leq c_{\mr{RE}} \| \vectheta_\eta^{\sharp1}\|_1.
		\end{align}
	Then, from the standard shelling argument, we have
	\begin{align}
		\| \vectheta_\eta^{\sharp2}\|_2^2=\sum_{i=s+1}^d(\vectheta_{\eta}^{\sharp}|_i)^2 \leq \sum_{i=s+1}^d\left| \vectheta_{\eta}^{\sharp}|_i\right|\left(\frac{1}{s}\sum_{j=1}^s\left|\vectheta_{\eta}^{\sharp}|_j\right|\right) \leq \frac{1}{s}\| \vectheta_\eta^{\sharp1}\|_1\| \vectheta_\eta^{\sharp2}\|_1\leq \frac{c_{\mr{RE}}\| \vectheta_\eta^{\sharp1}\|_1^2}{s}\leq c_{\mr{RE}}\| \vectheta_\eta^{\sharp1}\|_2^2.
	\end{align}
	and from the definition of $\kappa_{\mr{l}}$, we have 
	\begin{align}
		\kappa_{\mr{l}}^2 \|\vectheta_\eta\|_2^2 \leq \kappa_{\mr{l}}^2 \left(\|\vectheta_\eta^{\sharp1}\|_2^2 +\|\vectheta_\eta^{\sharp2}\|_2^2\right)\leq \kappa_{\mr{l}}^2 (1+c_{\mr{RE}})\|\vectheta_\eta^{\sharp1}\|_2^2 \leq(1+c_{\mr{RE}}) \|\Sigma^\frac{1}{2}\vectheta_\eta\|_2^2
	\end{align}
		\subsubsection{Case II}
		\label{sec:case2:re}
		In Section \ref{sec:case2:re}, suppose that $\|\vectheta_\eta \|_2 \geq \|\vectheta_\eta\|_1/\sqrt{s}$.
		\begin{align}
			\| \vectheta_\eta^{\sharp2}\|_2^2=\sum_{i=s+1}^d(\vectheta_{\eta}^{\sharp}|_i)^2 \leq \sum_{i=s+1}^d\left|\vectheta_{\eta}^{\sharp}|_i\right|\left(\frac{1}{s}\sum_{j=1}^s\left|\vectheta_{\eta}^{\sharp}|_j\right|\right) \leq \frac{1}{s}\| \vectheta_\eta^{\sharp1}\|_1\| \vectheta_\eta^{\sharp2}\|_1\leq \| \vectheta_\eta^{\sharp1}\|_2\| \vectheta_\eta\|_2.
		\end{align}
		Then, we have 
		\begin{align}
			\|\vectheta_\eta\|_2^2 \leq \|\vectheta_\eta^{\sharp1}\|_2^2 +\| \vectheta_\eta^{\sharp2}\|_2^2\leq  \|\vectheta_\eta^{\sharp1}\|_2\|\vectheta_\eta\|_2 +\| \vectheta_\eta^{\sharp1}\|_2\| \vectheta_\eta\|_2\Rightarrow 			\|\vectheta_\eta\|_2 \leq  2\|\vectheta_\eta^{\sharp1}\|_2,
		\end{align}
		and we have
		\begin{align}
				\|\vectheta_\eta\|_2 \leq  2\|\vectheta_\eta^{\sharp1}\|_2\leq \frac{2}{\kappa_{\mr{l}}}\|\Sigma^\frac{1}{2}\vectheta_\eta^{\sharp1}\|_2\leq\frac{2}{\kappa_{\mr{l}}}\|\Sigma^\frac{1}{2}\vectheta\|_2.
		\end{align}

\subsection{Proof of Lemma \ref{l:mainpre1-2}}
From the same argument of the proof of Lemma \ref{l:mainpre1-1}, we have \eqref{ine:det:lemma:3}.
From \eqref{ine:det:main:0-1}, the first term of the right-hand side of \eqref{ine:det:lemma:3} is evaluated as
\begin{align}
	\label{ine:det:lemma2:1}
	\sum_{i=1}^n \lambda_o\sqrt{n} \hat{w}_i' h(r_{\vecbeta^*,i})\vecX_i^\top\vectheta_\eta&\leq r_{a,1}r_1+r_{a,2} r_2 + r_{a,\Sigma} r_\Sigma\nonumber \\
	& \stackrel{(a)}{\leq} \left(r_{a,1}+\frac{c_{r_2}}{c_{r_1}\sqrt{s}} r_{a,2} + \frac{1}{c_{r_1}\sqrt{s}}r_{a,\Sigma}\right)\|\vectheta_\eta\|_1.
\end{align}
where (a) follows from $r_1 = c_{r_1}\sqrt{s}r_\Sigma $ and  $r_2 =c_{r_2} r_\Sigma $.
From~\eqref{ine:det:lemma:3} and \eqref{ine:det:lemma2:1}, we have
\begin{align}
0 \leq \left(r_{a,1}+ \frac{c_{r_2}r_{a,2}+ r_{a,\Sigma}}{c_{r_1}\sqrt{s}}\right)\|\vectheta_\eta\|_1+\eta \lambda_s (\|\vecbeta^*\|_1-\|\hat{\vecbeta}\|_1).
\end{align}
Furthermore, we see
\begin{align}
0
&\leq  \left(r_{a,1}+ \frac{c_{r_2}r_{a,2}+ r_{a,\Sigma}}{c_{r_1}\sqrt{s}}\right)\|\vectheta_\eta\|_1+\eta \lambda_s(\|\vecbeta^*\|_1-\|\hat{\vecbeta}\|_1) \nonumber \\
& \leq  \left(r_{a,1}+ \frac{c_{r_2}r_{a,2}+ r_{a,\Sigma}}{c_{r_1}\sqrt{s}}\right) (\|\vectheta_{\eta,\mc{J}_{\vecbeta^*}}\|_1+\|\vectheta_{\eta,\mc{J}^c_{\vecbeta^*}}\|_1) + \eta\lambda_s(\|\vecbeta^*_{\mc{J}_{\vecbeta^*}}- \hat{\vecbeta}_{\mc{J}_{\vecbeta^*}}\|_1-\|\hat{\vecbeta}_{\mc{J}^c_{\vecbeta^*}}\|_1)\nonumber\\
& =\left(\lambda_s +  \left(r_{a,1}+ \frac{c_{r_2}r_{a,2}+ r_{a,\Sigma}}{c_{r_1}\sqrt{s}}\right)\right)\|\vectheta_{\eta,\mc{J}_{\vecbeta^*}}\|_1
+\left(-\lambda_s+  \left(r_{a,1}+ \frac{c_{r_2}r_{a,2}+ r_{a,\Sigma}}{c_{r_1}\sqrt{s}}\right)\right)\|\vectheta_{\eta,\mc{J}^c_{\vecbeta^*}}\|_1.
\end{align}
Then, we have
	\begin{align}
		\label{ine:preshelling}
		\| \vectheta_{\eta,\mc{J}_{\vecbeta^*}^c}\|_1 \leq \frac{\lambda_s +  \left(r_{a,1}+ \frac{c_{r_2}r_{a,2}+ r_{a,\Sigma}}{c_{r_1}\sqrt{s}}\right)}{\lambda_s -  \left(r_{a,1}+ \frac{c_{r_2}r_{a,2}+ r_{a,\Sigma}}{c_{r_1}\sqrt{s}}\right) }\| \vectheta_{\eta,\mc{J}_{\vecbeta^*}}\|_1\stackrel{(a)}{\leq} c_{\mr{RE}} \| \vectheta_{\eta,\mc{J}_{\vecbeta^*}}\|_1,
	\end{align}
	where (a) follows from \eqref{ine:det:main:0-3}, and  we have
	\begin{align}
		\|\vectheta_\eta\|_1 = \| \vectheta_{\eta,\mc{J}_{\vecbeta^*}}\|_1 +\| \vectheta_{\eta,\mc{J}_{\vecbeta^*}^c}\|_1 \leq  (1+c_{\mr{RE}})\| \vectheta_{\eta,\mc{J}_{\vecbeta^*}}\|_1\leq (1+c_{\mr{RE}})\sqrt{s}\| \vectheta_{\eta,\mc{J}_{\vecbeta^*}}\|_2.
	\end{align}
	From \eqref{ine:preshelling} and the restricted eigenvalue condition, we have
	\begin{align}
		\|\vectheta_\eta\|_1 \leq  (1+c_{\mr{RE}})\sqrt{s}\| \vectheta_{\eta,\mc{J}_{\vecbeta^*}}\|_2\leq  \frac{1+c_{\mr{RE}}}{\mathfrak{r}}\sqrt{s}\|\Sigma^\frac{1}{2} \vectheta_\eta\|_2.
	\end{align}

	\subsection{Proof of Lemma \ref{l:mainpre2-1}}
	From the same argument of the proof of Lemma \ref{l:mainpre1-1}, we have \eqref{ine:det:lemma:3}.
	From \eqref{ine:det:lemma:3}, we have
	\begin{align}
		\label{ine:det:lemma3:1}
		&\sum_{i=1}^n \lambda_o\sqrt{n} \hat{w}_i'\left(-h(r_{\vecbeta^*+\vectheta_\eta,i}) +h(r_{\vecbeta^*,i})\right) \vecX_i^\top \vectheta_\eta \leq \sum_{i=1}^n \lambda_o\sqrt{n} \hat{w}_i' h(r_{\vecbeta^*,i}) \vecX_i^\top \vectheta_\eta+\eta \lambda_s(\|\vecbeta^*\|_1-\|\hat{\vecbeta}\|_1).
	\end{align}
	We evaluate each term of \eqref{ine:det:lemma3:1}. From~\eqref{ine:det:main:0-2}, the left-hand side  of~\eqref{ine:det:lemma3:1} is evaluated as
	\begin{align}
		\sum_{i=1}^n \lambda_o\sqrt{n} \hat{w}_i'\left(-h(r_{\vecbeta^*+\vectheta_\eta,i}) +h(r_{\vecbeta^*,i})\right) \vecX_i^\top \vectheta_\eta &\geq b \|\Sigma^\frac{1}{2}\vectheta_\eta\|_2^2 -r_{b,1}r_1-r_{b,2}r_2-r_{b,\Sigma}r_\Sigma\nonumber\\
		&\stackrel{(a)}{\geq } b \|\Sigma^\frac{1}{2}\vectheta_\eta\|_2^2 -\left(c_{r_1}\sqrt{s}r_{b,1}-c_{r_2}r_{b,2} -r_{b,\Sigma}\right)\|\Sigma^\frac{1}{2}\vectheta_\eta\|_2,
	\end{align}
	where (a) follows from  $r_1 = c_{r_1}\sqrt{s}r_\Sigma $ and  $r_2 =c_{r_2} r_\Sigma $ .
	From~\eqref{ine:det:main:0-1}, the first term of the right-hand side of \eqref{ine:det:lemma3:1} is evaluated as
	\begin{align}
		\sum_{i=1}^n \lambda_o\sqrt{n} \hat{w}_i' h(r_{\vecbeta^*,i}) \vecX_i^\top \vectheta_\eta&\leq r_{a,1}r_1+ r_{a,2}r_2+r_{a,\Sigma}r_{\Sigma}\nonumber \\
		&\stackrel{(a)}\leq \left( c_{r_1}\sqrt{s}r_{a,1}+
		c_{r_2}r_{a,2} + r_{a,\Sigma}\right)\|\Sigma^\frac{1}{2}\vectheta_\eta\|_2,
	\end{align}
	where (a) follows from  $r_1 = c_{r_1}\sqrt{s}r_\Sigma $ and  $r_2 =c_{r_2} r_\Sigma $.
	The second term of the right-hand side of \eqref{ine:det:lemma3:1} is evaluated as
	\begin{align}
		\eta \lambda_s(\|\vecbeta^*\|_1-\|\hat{\vecbeta}\|_1) \leq \lambda_s \|\vectheta_\eta\|_1 \stackrel{(a)}{\leq} c_{r_1}\sqrt{s}\lambda_s \|\Sigma^\frac{1}{2} \vectheta_\eta\|_2,
	\end{align}
	where (a) follows from  $r_1 = c_{r_1}\sqrt{s}r_\Sigma $.

	Combining the inequalities above, we have
	\begin{align}
		b \|\Sigma^\frac{1}{2}\vectheta_\eta\|_2^2&\leq  \left(c_{r_1}\sqrt{s}(r_{a,1}+r_{b,1})+
		c_{r_2}(r_{a,2}+r_{b,2}) +r_{a,\Sigma}+r_{b,\Sigma} + c_{r_1}\sqrt{s}\lambda_s\right)\|\Sigma^\frac{1}{2}\vectheta_\eta\|_2,
	 \end{align}
	 and from $\|\Sigma^\frac{1}{2}\vectheta_\eta\|_2\geq 0$, we have
	\begin{align}
		\|\Sigma^\frac{1}{2}\vectheta_\eta\|_2 &\leq \frac{1}{b} \left(c_{r_1}\sqrt{s}(r_{a,1}+r_{b,1})+
		c_{r_2}(r_{a,2}+r_{b,2}) +r_{a,\Sigma}+r_{b,\Sigma} + c_{r_1}\sqrt{s}\lambda_s\right),
	\end{align}
	and the proof is complete.

\section{Proof of Proposition~\ref{p:main}}
\label{sec:mainproof}
\subsection{Step1}
\label{subsec:mainpropstep1}
We derive a contradiction if $\|\vectheta\|_1> r_1$,  $\|\vectheta\|_2> r_2$ and $\|\Sigma^\frac{1}{2}\vectheta\|_2> r_{\Sigma}$ hold. Assume that  $\|\vectheta\|_1> r_1$,  $\|\vectheta\|_2> r_2$ and $\|\Sigma^\frac{1}{2}\vectheta\|_2> r_{\Sigma}$. Then we can find  $\eta_1,\,\eta_2,\eta_2'\in(0,1)$ such that $\|\vectheta_{\eta_1}\|_1 = r_1$, $\|\vectheta_{\eta_2}\|_2 = r_2$ and $\|\Sigma^\frac{1}{2}\vectheta_{\eta_2'}\|_2 = r_{\Sigma}$ hold.
Define $\eta_3 = \min\{\eta_1,\eta_2, \eta_2'\}$.
We consider the case $\eta_3 = \eta_2'$ in Section \ref{subsubsec:mainpropstep1a},  the case $\eta_3 = \eta_2$ in Section \ref{subsubsec:mainpropstep1b}, and the case $\eta_3 = \eta_1$ in Section \ref{subsubsec:mainpropstep1c}.

\subsubsection{Step 1(a)}
\label{subsubsec:mainpropstep1a}
Assume that $\eta_3 = \eta_2'$. We see that $\|\Sigma^\frac{1}{2}\vectheta_{\eta_3}\|_2 = r_{\Sigma}$, $\|\vectheta_{\eta_3}\|_1 \leq r_1$ and $\|\vectheta_{\eta_3}\|_2 \leq r_2$ hold. Then, from  Lemma~\ref{l:mainpre2-1},  we have 
\begin{align}
\|\Sigma^\frac{1}{2}\vectheta_{\eta_3}\|_2 &\leq \frac{1}{b} \left(c_{r_1}\sqrt{s}(r_{a,1}+r_{b,1})+
c_{r_2}(r_{a,2}+r_{b,2}) +r_{a,\Sigma}+r_{b,\Sigma} + c_{r_1}\sqrt{s}\lambda_s\right).
\end{align}
The case $\eta_3=\eta_2'$ is a contradiction from 
$\|\Sigma^\frac{1}{2}\vectheta_{\eta_3}\|_2 = r_\Sigma$
and \eqref{ine:det:main:0-4}.

\subsubsection{Step 1(b)}
\label{subsubsec:mainpropstep1b}
Assume that $\eta_3 = \eta_2$. We see that  $\|\Sigma^\frac{1}{2}\vectheta_{\eta_3}\|_2 \leq r_{\Sigma}$, $\|\vectheta_{\eta_3}\|_1 \leq r_1$ and $\|\vectheta_{\eta_3}\|_2 = r_2$ hold. Then, from  Lemma~\ref{l:mainpre1-1}, we have 
\begin{align}
\|\vectheta_{\eta_3}\|_2 &\leq \frac{3+c_{\mr{RE}}}{\kappa_{\mr{l}}}\|\Sigma^\frac{1}{2}\vectheta_{\eta_3}\|_2 \leq \frac{3+c_{\mr{RE}}}{\kappa_{\mr{l}}}r_{\Sigma}.
\end{align}
The case $\eta_3=\eta_2$ is a contradiction from
$\|\vectheta_{\eta_3}\|_2 = r_2 $ and \eqref{ine:det:main:0-4}.

\subsubsection{Step 1(c)}
\label{subsubsec:mainpropstep1c}
Assume that $\eta_3 = \eta_1$. We see that  $\|\Sigma^\frac{1}{2}\vectheta_{\eta_3}\|_2 \leq r_{\Sigma}$, $\|\vectheta_{\eta_3}\|_1 = r_1$ and $\|\vectheta_{\eta_3}\|_2 \leq r_2$ hold. 
Then, from Lemma \ref{l:mainpre1-2}, for $\eta = \eta_{3}$, we have
\begin{align}
\|\vectheta_{\eta_3}\|_1\leq \frac{1+c_{\mr{RE}}}{\mathfrak{r}}\sqrt{s}\|\Sigma^\frac{1}{2}\vectheta_{\eta_3}\|_2 \leq  \frac{1+c_{\mr{RE}}}{\mathfrak{r}}\sqrt{s}r_\Sigma.
\end{align}
The case $\eta_3=\eta_1$ is a contradiction from
$\|\vectheta_{\eta_3}\|_1 =r_1$ and  \eqref{ine:det:main:0-4}.
\subsection{Step 2}
\label{asubsec:2}
From the arguments in Section \ref{subsec:mainpropstep1}, we have $\|\Sigma^\frac{1}{2}\vectheta\|_2 \leq r_{\Sigma,}$ or $\|\vectheta\|_1 \leq r_1$ or $\|\vectheta\|_2 \leq r_2$  holds.
\begin{itemize}
	\item[(a)]In Section \ref{asubsubsec:2-1}, assume that  $\|\Sigma^\frac{1}{2}\vectheta\|_2 \leq r_{\Sigma}$ and $\|\vectheta\|_1 > r_1$ and $\|\vectheta\|_2 > r_2$
	hold and then derive a contradiction.
 \item[(b)]In Section \ref{asubsubsec:2-2}, assume that  $\|\Sigma^\frac{1}{2}\vectheta\|_2 > r_{\Sigma}$ and $\|\vectheta\|_1 \leq r_1$ and $\|\vectheta\|_2 > r_2$
hold and then derive a contradiction.
\item[(c)]In Section \ref{asubsubsec:2-3}, assume that  $\|\Sigma^\frac{1}{2}\vectheta\|_2 > r_{\Sigma}$ and $\|\vectheta\|_1 >r_1$ and $\|\vectheta\|_2 \leq r_2$
hold and then derive a contradiction.
\item[(d)]In Section \ref{asubsubsec:2-4}, assume that  $\|\Sigma^\frac{1}{2}\vectheta\|_2 > r_{\Sigma}$ and $\|\vectheta\|_1 \leq r_1$ and $\|\vectheta\|_2 \leq r_2$
hold and then derive a contradiction.
\item[(e)]In Section \ref{asubsubsec:2-5}, assume that  $\|\Sigma^\frac{1}{2}\vectheta\|_2 \leq r_{\Sigma}$ and $\|\vectheta\|_1 > r_1$ and $\|\vectheta\|_2 \leq r_2$
hold and then derive a contradiction.
\item[(f)]In Section \ref{asubsubsec:2-6}, assume that  $\|\Sigma^\frac{1}{2}\vectheta\|_2 \leq r_{\Sigma}$ and $\|\vectheta\|_1 \leq r_1$ and $\|\vectheta\|_2 > r_2$
hold and then derive a contradiction.
\end{itemize}
Finally, we have
\begin{align}
	\|\Sigma^\frac{1}{2}(\hat{\vecbeta}-\vecbeta^*)\|_2 \leq r_\Sigma \text {, }\|\hat{\vecbeta}-\vecbeta^*\|_2 \leq r_2 \text {, and }	\|\hat{\vecbeta}-\vecbeta^*\|_1 \leq r_1,
\end{align}
and the proof is complete.
\subsubsection{Step 2(a)}
\label{asubsubsec:2-1}
Assume that $\|\Sigma^\frac{1}{2}\vectheta\|_2 \leq r_{\Sigma,}$ and $\|\vectheta\|_1 > r_1$ and $\|\vectheta\|_2 > r_2$ hold, and
then we can find $\eta_4, \eta_4'\in (0,1)$ such that $\|\vectheta_{\eta_4}\|_1 = r_1$ and $\|\vectheta_{\eta_4'}\|_2 = r_2$ hold.
We note that $\|\Sigma^\frac{1}{2}\vectheta_{\eta_4}\|_2\leq r_\Sigma$ and $\|\Sigma^\frac{1}{2}\vectheta_{\eta_4'}\|_2\leq r_\Sigma$ also hold.
Then, from the same arguments of Sections \ref{subsubsec:mainpropstep1b} and \ref{subsubsec:mainpropstep1c}, we have a contradiction.

\subsubsection{Step 2(b)}
\label{asubsubsec:2-2}
Assume that $\|\Sigma^\frac{1}{2}\vectheta\|_2 > r_{\Sigma}$ and $\|\vectheta\|_1 \leq r_1$ and $\|\vectheta\|_2 > r_2$ hold, and
then we can find $\eta_5, \eta_5'\in (0,1)$ such that $\|\Sigma^\frac{1}{2}\vectheta_{\eta_5}\|_2 = r_{\Sigma}$ and $\|\vectheta_{\eta_5'}\|_2 = r_2$ hold.
We note that $\|\vectheta_{\eta_5}\|_1\leq r_1$ and $\|\vectheta_{\eta_5'}\|_1\leq r_1$  also hold.
Then, from the same arguments of Sections \ref{subsubsec:mainpropstep1a} and  \ref{subsubsec:mainpropstep1b}, we have a contradiction.

\subsubsection{Step 2(c)}
\label{asubsubsec:2-3}
Assume that $\|\Sigma^\frac{1}{2}\vectheta\|_2 > r_{\Sigma}$ and $\|\vectheta\|_1 >r_1$ and $\|\vectheta\|_2 \leq r_2$ hold and,
then we can find $\eta_6,\eta_6'\in (0,1)$ such that $\|\vectheta_{\eta_6}\|_1=r_1$ and  $\|\Sigma^\frac{1}{2}\vectheta_{\eta_6'}\|_2 =r_\Sigma$ hold.
We note that $\|\vectheta_{\eta_6}\|_2\leq r_2$ and $\|\vectheta_{\eta_6'}\|_2\leq r_2$  also hold.
Then, from the same arguments of Sections \ref{subsubsec:mainpropstep1a} and  \ref{subsubsec:mainpropstep1c}, we have a contradiction.

\subsubsection{Step 2(d)}
\label{asubsubsec:2-4}
Assume that $\|\Sigma^\frac{1}{2}\vectheta\|_2 > r_{\Sigma}$ and $\|\vectheta\|_1 \leq r_1$ and $\|\vectheta\|_2 \leq r_2$ hold and,
then we can find $\eta_7\in (0,1)$ such that  $\|\Sigma^\frac{1}{2}\vectheta_{\eta_7}\|_2 = r_{\Sigma}$ holds.
We note that $\|\vectheta_{\eta_7}\|_1\leq r_1$ and  $\|\vectheta_{\eta_7}\|_2\leq r_2$ also hold.
Then, from the same arguments of Section \ref{subsubsec:mainpropstep1a}, we have  a contradiction.

\subsubsection{Step 2(e)}
\label{asubsubsec:2-5}
Assume that $\|\Sigma^\frac{1}{2}\vectheta\|_2 \leq r_{\Sigma}$ and $\|\vectheta\|_1 > r_1$ and $\|\vectheta\|_2 \leq r_2$ hold and,
then we can find $\eta_8\in (0,1)$ such that $\|\vectheta_{\eta_8}\|_1=r_1$ holds.
We note that $\|\Sigma^\frac{1}{2}\vectheta_{\eta_8}\|_2\leq r_\Sigma$ and $\|\vectheta_{\eta_8}\|_2\leq r_2$ also hold.
Then, from the same arguments of Section \ref{subsubsec:mainpropstep1c}, we have  a contradiction.

\subsubsection{Step 2(f)}
\label{asubsubsec:2-6}
Assume that $\|\Sigma^\frac{1}{2}\vectheta\|_2 \leq r_{\Sigma}$ and $\|\vectheta\|_1 \leq r_1$ and $\|\vectheta\|_2 > r_2$hold and,
then we can find $\eta_9\in (0,1)$ such that $\|\vectheta_{\eta_9}\|_2=r_2$ holds.
We note that $\|\Sigma^\frac{1}{2}\vectheta_{\eta_9}\|_2\leq r_\Sigma$  and $\|\vectheta_{\eta_9}\|_1\leq r_1$ also hold.
Then, from the same arguments of Section \ref{subsubsec:mainpropstep1b}, we have  a contradiction.

\section{Proofs of  Lemmas \ref{l:w2}, \ref{l:gwslope:main}, \ref{l:convl0l1} and \ref{l:tg-eval}}
\label{sec:p1}

\subsection{Proof of Lemma \ref{l:w2}}
We assume $|I_{<}| > 2\varepsilon n$, and then we derive a contradiction.
 From the constraint about $w_i$, we have $0\leq w_i \leq \frac{1}{\left(1-\varepsilon\right)n}$ for any $i \in\{1,\cdots,n\}$ and we have
 \begin{align}
	\sum_{i=1}^n w_i &= \sum_{i \in I_{<}} w_i+ \sum_{i \in I_{\geq}} w_i \nonumber\\&\leq |I_{_<}| \times \frac{1+}{2n} + (n-|I_{<}|) \times \frac{1}{\left(1-\varepsilon\right)n}\nonumber\\
	& = 2\varepsilon n \times \frac{1}{2n} + (|I_{<}|-2\varepsilon n) \times \frac{1+\varepsilon}{2n} + (n-2\varepsilon n) \times \frac{1}{\left(1-\varepsilon\right)n} +(2\varepsilon n -|I_{<}|) \times \frac{1}{\left(1-\varepsilon\right)n}\nonumber\\
	&= \varepsilon + (n-2\varepsilon n) \times \frac{1}{\left(1-\varepsilon\right)n} +(|I_{<}|-2\varepsilon n) \times\left(\frac{1}{2n}-\frac{1}{\left(1-\varepsilon\right)n}\right)\nonumber\\
	&< \varepsilon + \frac{n-2\varepsilon n}{\left(1-\varepsilon\right)n}\nonumber\\
	&= \varepsilon + \frac{1-2\varepsilon }{1-\varepsilon}\nonumber\\
	&\leq \frac{1-\varepsilon -\varepsilon^2}{1-\varepsilon}\nonumber \\
	&< 1.
 \end{align}
This is a contradiction because $	\sum_{i=1}^n w_i=1$. Then, we have $|I_{<}| \leq 2\varepsilon n$.

\subsection{Proofs of Lemmas \ref{l:gwslope:main},  \ref{l:convl0l1} and  \ref{l:tg-eval}}
\label{asec:prooftg-eval}

\subsubsection{Proof of Lemma \ref{l:gwslope:main}}
The following argument is essentially identical to a part of the proof of Proposition 9 of \cite{Bel2019Localized}.
We note that
\begin{align}
	r_1 \mbb{B}^d_1\cap r_\Sigma \mbb{B}^d_\Sigma\subset 	2r_1 \mbb{B}^d_1\cap r_\Sigma \mbb{B}^d_\Sigma &=2r_1 \left( 1\times \mbb{B}^d_1\cap \frac{r_\Sigma}{2r_1}\mbb{B}^d_\Sigma \right),
\end{align}
and we have
\begin{align}
	\mbb{E}\sup_{\vecv \in r_1 \mbb{B}^d_1\cap r_\Sigma \mbb{B}^d_\Sigma} \langle \Sigma^\frac{1}{2}  \vecg,\vecv\rangle 
	\leq\mbb{E}\sup_{\vecv \in 	2r_1 \left( 1\times \mbb{B}^d_1\cap \frac{r_\Sigma}{2r_1}\mbb{B}^d_\Sigma \right)} \langle \Sigma^\frac{1}{2} \vecg,\vecv\rangle.
\end{align}
Let $\{\vecu_1,\cdots,\vecu_d\}$ be the columns of $\Sigma^{\frac{1}{2} \top}$. Then, we have
\begin{align}
\mbb{E}\sup_{\vecv \in 	2r_1 \left( 1\times \mbb{B}^d_1\cap \frac{r_\Sigma}{2r_1}\mbb{B}^d_\Sigma \right)} \langle \Sigma^\frac{1}{2}  \vecg, \vecv\rangle&\leq \mbb{E}\sup_{\vecv \in 	2r_1 \left(\mr{conv}\left(\{\pm\vecu_1,\cdots,\pm\vecu_d\}\right)\cap \frac{r_\Sigma}{2r_1} \mbb{B}_2\right)} \langle \vecg, \vecv\rangle\nonumber\\
&\leq \mbb{E}\sup_{\vecv \in 	2r_1 \rho\left(\mr{conv}\left( \frac{1}{\rho}\{\pm\vecu_1,\cdots,\pm\vecu_d\}\right)\cap \frac{r_\Sigma}{2r_1} \mbb{B}_2\right)} \langle \vecg, \vecv\rangle.
\end{align}
From Proposition 1 of \cite{Bel2019Localized}, we have
\begin{align}
	\label{ine:order}
\mbb{E}\sup_{\vecv \in 2r_1 \rho\left(\mr{conv}\left( \frac{1}{\rho}\{\pm\vecu_1,\cdots,\pm\vecu_d\}\right)\cap \frac{r_\Sigma}{2r_1} \mbb{B}_2\right)} \langle \vecg,\vecv\rangle&\leq 	4\rho r_1 \sqrt{\log \left(4ed \left(\frac{r_\Sigma}{2r_1}\right)^2\right)}\nonumber\\
	&= 	4\rho r_1 \sqrt{\log \left(\frac{ed}{sc^2_{r_1}}\right)}\nonumber\\
	&\stackrel{(a)}{\leq}	C\rho r_1 \sqrt{\log (d/s)},
\end{align}
where (a) follows from $\mathfrak{r}\leq 1$.

\subsubsection{Proof of Lemma \ref{l:convl0l1}}
First, we prove \eqref{ine:convvec}.
For any $\vecv \in r_1 \mbb{B}^{d}_1 \cap r_2 \mbb{B}^{d}_{2}$, we see that 
\begin{align}
	\frac{\vecv}{r_2} \in \left\{\vecv \in \mbb{R}^d \mid \|\vecv\|_2\leq 1,\,\|\vecv\|_1 \leq \frac{r_1}{r_2}\right\} = \frac{r_1}{r_2} \mbb{B}^{d}_1 \cap  \mbb{B}^{d}_{2}.
\end{align}
From Lemma 3.1 of \cite{PlaVer2013One}, we have
\begin{align}
	\frac{r_1}{r_2} \mbb{B}^{d}_1 \cap  \mbb{B}^{d}_{2} \subset 2 \mr{conv}\left(\left(\frac{r_1}{r_2}\right)^2 \mbb{B}^d_0 \cap \mbb{B}^d_2\right),
\end{align}
and 
\begin{align}
	\label{ine:PVL3}
	\frac{\vecv}{r_2} \in  2 \mr{conv}\left(\left(\frac{r_1}{r_2}\right)^2 \mbb{B}^d_0 \cap \mbb{B}^d_2\right).
\end{align}
From the definition of the convex hull, this means 
\begin{align}
	\frac{\vecv}{r_2} =  2 \sum_{i=1}^a \lambda_i \veca_i.
\end{align}
for some $\{\lambda_i\}_{i=1}^a$ and $\{\veca_i\}_{i=1}^a$ such that $\sum_{i=1}^a \lambda_i = 1,\,\lambda_i\geq 0$ and  $\veca_i \in \left(\frac{r_1}{r_2}\right)^2 \mbb{B}^d_0 \cap \mbb{B}^d_2$ for any $i \in \{1,\cdots,a\}$.
We note that 
\begin{align}
	\label{ine:normscale}
	\|r_2 \veca_i\|_0 \leq \left(\frac{r_1}{r_2}\right)^2,\quad \|r_2 \veca_i\|_2 \leq r_2.
\end{align}
From \eqref{ine:PVL3} and \eqref{ine:normscale}, we see 
\begin{align}
	\vecv \in 2 \mr{conv}\left(\left(\frac{r_1}{r_2}\right)^2 \mbb{B}^d_0 \cap r_2\mbb{B}^d_2\right).
\end{align}
Then the proof of \eqref{ine:convvec} is complete.
Identifying the vectors and the matrices,  the proof of  \eqref{ine:convmat}  is also complete.
\subsubsection{Proof of Lemma \ref{l:tg-eval}}
We note that 
\begin{align}
	\label{ine:eleint}
	\int^\infty_0 \frac{x}{e^x}dx &=1,\,\,\int^\infty_0 \frac{\sqrt{x}}{e^x}dx \leq \int^1_0 \frac{1}{e^x}dx +\int^\infty_0 \frac{x}{e^x}dx \leq \left[\frac{-1}{e^x}\right]^\infty_0+1 =2.
\end{align}
First, we evaluate $\gamma_1(s^2 \mbb{B}^{d\times d}_0 \cap r_2^2 \mbb{B}^{d\times d}_{\mr{F}} ,\|\cdot\|_{\mr{F}})$.
From a standard entropy bound from chaining theory Lemma D.17 of \cite{Oym2018Learning}, we have
\begin{align}
	\label{ine:dud}
	\gamma_1(s^2 \mbb{B}^{d\times d}_0 \cap r_2^2 \mbb{B}^{d\times d}_{\mr{F}},\|\cdot\|_{\mr{F}}) \leq C\int_0^{r_2^2} \log N(s^2 \mbb{B}^{d\times d}_0 \cap r_2^2 \mbb{B}^{d\times d}_{\mr{F}} ,\epsilon)d\epsilon,
\end{align}
where $ N(s^2 \mbb{B}^{d\times d}_0 \cap r_2^2 \mbb{B}^{d\times d}_{\mr{F}},\epsilon)$ is the covering number of $s^2 \mbb{B}^{d\times d}_0 \cap r_2^2 \mbb{B}^{d\times d}_{\mr{F}} $, that is the minimal cardinality of an $\epsilon$-net of $s^2 \mbb{B}^{d\times d}_0 \cap r_2^2 \mbb{B}^{d\times d}_{\mr{F}}$.
From Lemma Lemma \ref{l:convl0l1} and $d/s\geq 3$, we have 
\begin{align}
	\int_0^{r_2^2} \log N(s^2 \mbb{B}^{d\times d}_0 \cap r_2^2 \mbb{B}^{d\times d}_{\mr{F}},\epsilon)d\epsilon &\leq C\int_0^{r_2^2} s^2\log r_2^2 \frac{9d^2}{\epsilon s^2}d \epsilon\nonumber\\
	&= Cs^2\int_0^{r_2^2} \left(\log (9d^2/s^2)+  \log(r_2^2/\epsilon)\right)d\epsilon\nonumber\\
	&=Cs^2\left(r_2^2\log (9d^2/s^2)+ r_2^2 \int_1^{\infty} \frac{x}{e^x}dx\right) \nonumber\\
	&\leq Cs^2\left(r_2^2\log (9d^2/s^2)+ r_2^2 \int_0^{\infty} \frac{x}{e^x}dx\right) \nonumber\\
	&\stackrel{(a)}{=} Cs^2\left(r_2^2 \log (9d^2/s^2)+ r_2^2\right) \nonumber\\
	& \stackrel{(b)}{\leq}Cs^2r_2^2 \log (d/s),
\end{align}
where (a) follows from \eqref{ine:eleint}, and (b) follows from $3\leq d/s$.
Consequently, we have
\begin{align}
	\gamma_1(s^2 \mbb{B}^{d\times d}_0 \cap r_2^2 \mbb{B}^{d\times d}_{\mr{F}},\|\cdot\|_{\mr{F}}) \leq Cs^2r_2^2 \log (d/s).
\end{align}

Second, we evaluate $\gamma_2(s^2 \mbb{B}^{d\times d}_0 \cap r_2^2 \mbb{B}^{d\times d}_{\mr{F}},\|\cdot\|_{\mr{F}})$. From  similar argument of the case $\gamma_1(s^2 \mbb{B}^{d\times d}_0 \cap r_2^2 \mbb{B}^{d\times d}_{\mr{F}},\|\cdot\|_{\mr{F}})$, we have
\begin{align}
	\int_0^{r_2^2}\sqrt{ \log N(s^2 \mbb{B}^{d\times d}_0 \cap r_2^2 \mbb{B}^{d\times d}_{\mr{F}},\epsilon)}d\epsilon &\leq C\int_0^{r_2^2} s\sqrt{\log r_2^2\frac{9d^2}{\epsilon  s^2}}d \epsilon\nonumber\\
	&\stackrel{(a)}{\leq} C\left(s \int^{r_2^2}_0  \sqrt{\log (d/s)}+\int^{r_2^2}_1 \sqrt{\log \frac{r_2^2}{\epsilon}}d \epsilon\right)\nonumber\\
	&\leq C\left(s r_2^2  \sqrt{\log (d/s)}+\int^{r_2^2}_1 \sqrt{\log \frac{r_2^2}{\epsilon}}d \epsilon\right)\nonumber\\
	&\leq C\left(s r_2^2  \sqrt{\log (d/s)}+\int^{r_2^2}_0 \sqrt{\log \frac{r_2^2}{\epsilon}}d \epsilon\right)\nonumber\\
	&= C\left(s r_2^2  \sqrt{\log (d/s)}+\int^\infty_0 \frac{\sqrt{x}}{e^x}dx\right)\nonumber\\
	&\stackrel{(b)}{\leq} Csr_2^2 \sqrt{\log (d/s)},
\end{align}
where (a) follows from triangular inequality, and (b) follows from \eqref{ine:eleint} and $d/s\geq 3$.

\section{Proofs of Corollary \ref{c:cwpre} and   Propositions \ref{p:main:out}, \ref{p:main:out2},\ref{p:main:sc}, \ref{p:main:out-un} and \ref{p:main:out2-un}}
\label{sec:proofkeyL}
Define 
\begin{align}
\mathfrak{M}_{r_1,r_2,d,\vecv}^{\ell_1,\ell_2} = \{M\in \mc{S}(d)\,:\, M = \vecv \vecv^\top, \vecv \in \mbb{R}^d,\, \,\|\vecv\|_1 \leq r_1, \|\vecv\|_2 \leq r_2\}.
\end{align}
\subsection{Proof  of Corollary \ref{c:cwpre}}
First, we prove \eqref{ine:cwpre2-1}.
We see that, for any $M \in \mathfrak{M}_{r_1,r_2,d}^{\ell_1,\mr{Tr}} $,
\begin{align}
\label{ine:unknownbase}
	\left|\frac{1}{n}\sum_{i=1}^n\langle \vecx_i \vecx_i^\top, M\rangle \right|
	&\leq \left|\frac{1}{n}\sum_{i=1}^n\langle \vecx_i \vecx_i^\top-\Sigma, M\rangle \right|+\left|\langle \Sigma, M\rangle \right|\nonumber \\
	&\leq  \left|\frac{1}{n}\sum_{i=1}^n\langle \vecx_i \vecx_i^\top-\Sigma, M\rangle \right|+\Sigma_{\max}^2 r_2^2,
\end{align}
where we use H{\"o}lder's inequality.
From \eqref{cons2} and \eqref{ine:unknownbase}, with probability at least $1-\delta$, we have
\begin{align}
	\max_{M\in\mathfrak{M}_{r_1,r_2,d}^{\ell_1,\mr{Tr}} }\left|\frac{1}{n}\sum_{i=1}^n\langle \vecx_i \vecx_i^\top, M\rangle \right|
	&\leq C(L\kappa_{\mr{u}})^2\left(sr_{d,s}+r_\delta \right)r_2^2+\Sigma_{\max}^2r_2^2.
\end{align}

Next, we prove \eqref{ine:cwpre2-2}. Remember \eqref{ine:weight}, we have
\begin{align}
\label{ine:optM0un}
	\max_{M\in\mathfrak{M}_{r_1,r_2,d}^{\ell_1,\mr{Tr}}  }\sum_{i =1 }^n \hat{w}_i\langle \vecX_i\vecX_i^\top, M\rangle 
	&\stackrel{(a)}{\leq} \max_{M\in\mathfrak{M}_{r_1,r_2,d}^{\ell_1,\mr{Tr}} }\sum_{i =1 }^n w_i^\circ \langle \vecX_i\vecX_i^\top, M\rangle \nonumber \\
	&= \frac{1}{1-o/n} \max_{M\in\mathfrak{M}_{r_1,r_2,d}^{\ell_1,\mr{Tr}} }\sum_{i =1 }^n \frac{1}{n}\langle \vecx_i\vecx_i^\top, M\rangle,
\end{align}
where (a) follows from the optimality of $\{\hat{w}_i\}_{i=1}^n$ and $\{w_i^\circ\}_{i=1}^n \in\Delta^{n-1}(\varepsilon)$.
From  \eqref{ine:cwpre2-1}, \eqref{ine:optM0un} and  $1/(1-o/n)\leq 2$,  with probability at least $1-\delta$, we have
\begin{align}
	\max_{M\in\mathfrak{M}_{r_1,r_2,d}^{\ell_1,\mr{Tr}} }\sum_{i =1 }^n \hat{w}_i\langle \vecX_i\vecX_i^\top, M\rangle 
	&\leq  C\left((L\kappa_{\mr{u}})^2\left(sr_{d,s}+r_\delta \right)+\Sigma_{\max}^2\right)r_2^2,
\end{align}
and from $L\geq 1$, the proof is complete.

\subsection{Proof of Proposition \ref{p:main:out}}
We have
\begin{align}
	\label{ine:Efirst-0}
	\max_{\vecv \in r_1 \mbb{B}^d_1 \cap r_2 \mbb{B}^d_2} \left|\sum_{i \in \mc{O}}\hat{w}'_iu_i \vecX_i^\top\vecv\right| &\stackrel{(a)}{\leq}  \max_{\vecv \in 2 \mr{conv}\left(\left(\frac{r_1}{r_2}\right)^2 \mbb{B}^d_0 \cap r_2 \mbb{B}^d_2\right)}\left|\sum_{i \in \mc{O}}\hat{w}'_iu_i \vecX_i^\top\vecv\right|\nonumber\\
	&  \leq \max_{\vecv \in \mr{conv}\left(\left(\frac{r_1}{r_2}\right)^2 \mbb{B}^d_0 \cap r_2 \mbb{B}^d_2\right)}2\left|\sum_{i \in \mc{O}}\hat{w}'_iu_i \vecX_i^\top\vecv \right|\nonumber\\
	&\stackrel{(b)}{\leq} 2 \max_{\vecv \in \left(\frac{r_1}{r_2}\right)^2  \mbb{B}^d_0 \cap r_2 \mbb{B}^d_2}\left|\sum_{i \in \mc{O}}\hat{w}'_iu_i \vecX_i^\top\vecv \right|,
\end{align}
where (a) follows from Lemma \ref{l:convl0l1}, and (b) follows from Lemma D.8 of \cite{Oym2018Learning}.
We note that, for any $\vecv \in \mbb{R}^d$,
\begin{align}
	\label{ine:Efirst}
	\left|\sum_{i \in \mc{O}}\hat{w}'_iu_i \vecX_i^\top\vecv \right|^2 &\stackrel{(a)}{\leq} 4 \frac{o}{n}\sum_{i \in \mc{O}}\hat{w}'_i |\vecX_i^\top \vecv|^2\stackrel{(b)}{\leq} 8 \frac{o}{n}\sum_{i\in \mc{O}}\hat{w}_i |\vecX_i^\top \vecv|^2,
\end{align}
where (a) follows from H{\"o}lder's inequality and $\sum_{i \in \mc{O}}u_i^2 \leq 4o$, and (b) follows from the fact that $\hat{w}_i'\leq 2\hat{w}_i$ for any $i\in (1,\cdots,n)$.
We focus on $\sum_{i \in \mc{O}}\hat{w}_i |\vecX_i^\top \vecv|^2$.

First,  we have
\begin{align}
	\label{ine:optM-pre}
	&\max_{\vecv \in \left(\frac{r_1}{r_2}\right)^2 \mbb{B}^d_0 \cap r_2 \mbb{B}^d_2}\sum_{i\in \mc{O}}\hat{w}_i |\vecX_i^\top \vecv|^2 \nonumber\\
	&= \max_{\vecv \in \left(\frac{r_1}{r_2}\right)^2 \mbb{B}^d_0 \cap r_2 \mbb{B}^d_2}\sum_{i\in \mc{O}}\hat{w}_i \langle \vecX_i \vecX_i^\top  ,\vecv \vecv^\top \rangle \nonumber\\
	&\leq \max_{\vecv \in \left(\frac{r_1}{r_2}\right)^2 \mbb{B}^d_0 \cap r_2 \mbb{B}^d_2}\sum_{i=1}^n \hat{w}_i \langle \vecX_i \vecX_i^\top  ,\vecv \vecv^\top \rangle +	 \max_{\vecv \in \left(\frac{r_1}{r_2}\right)^2 \mbb{B}^d_0 \cap r_2 \mbb{B}^d_2}\sum_{i\in \mc{I}}\hat{w}_i \langle \vecX_i \vecX_i^\top  ,\vecv \vecv^\top \rangle\nonumber\\
	&= \max_{\vecv \in \left(\frac{r_1}{r_2}\right)^2 \mbb{B}^d_0 \cap r_2 \mbb{B}^d_2}\sum_{i=1}^n \hat{w}_i \langle \vecX_i \vecX_i^\top  ,\vecv \vecv^\top \rangle +	 \max_{\vecv \in \left(\frac{r_1}{r_2}\right)^2 \mbb{B}^d_0 \cap r_2 \mbb{B}^d_2}\sum_{i\in \mc{I}}\hat{w}_i \langle \vecx_i \vecx_i^\top  ,\vecv \vecv^\top \rangle\nonumber\\
	&\leq \max_{\vecv \in \left(\frac{r_1}{r_2}\right)^2 \mbb{B}^d_0 \cap r_2 \mbb{B}^d_2}\sum_{i=1}^n \hat{w}_i \langle \vecX_i \vecX_i^\top -\Sigma ,\vecv \vecv^\top \rangle +	 \max_{\vecv \in \left(\frac{r_1}{r_2}\right)^2 \mbb{B}^d_0 \cap r_2 \mbb{B}^d_2}\sum_{i\in \mc{I}}\hat{w}_i \langle \vecx_i \vecx_i^\top  -\Sigma  ,\vecv \vecv^\top \rangle\nonumber\\
 &\quad \quad +\max_{\vecv \in \left(\frac{r_1}{r_2}\right)^2 \mbb{B}^d_0 \cap r_2 \mbb{B}^d_2}\sum_{i\in \mc{O}}\hat{w}_i \langle \Sigma ,\vecv \vecv^\top \rangle  \nonumber\\
	&\leq \max_{\vecv \in \mr{conv}\left(\left(\frac{r_1}{r_2}\right)^2 \mbb{B}^d_0 \cap r_2 \mbb{B}^d_2\right)}\sum_{i=1}^n \hat{w}_i \langle \vecX_i \vecX_i^\top -\Sigma ,\vecv \vecv^\top \rangle +	 \max_{\vecv \in\mr{conv}\left(\left(\frac{r_1}{r_2}\right)^2 \mbb{B}^d_0 \cap r_2 \mbb{B}^d_2\right)}\sum_{i\in \mc{I}}\hat{w}_i \langle \vecx_i \vecx_i^\top  -\Sigma  ,\vecv \vecv^\top \rangle\nonumber\\
	&\quad \quad +\max_{\vecv \in \left(\frac{r_1}{r_2}\right)^2 \mbb{B}^d_0 \cap r_2 \mbb{B}^d_2}\sum_{i\in \mc{O}}\hat{w}_i \langle \Sigma ,\vecv \vecv^\top \rangle  \nonumber\\
	&\stackrel{(a)}{\leq} \max_{\vecv \in  r_1\mbb{B}^d \cap r_2 \mbb{B}^d_2}\sum_{i=1}^n \hat{w}_i \langle \vecX_i \vecX_i^\top -\Sigma ,\vecv \vecv^\top \rangle +	 \max_{\vecv \in r_1\mbb{B}^d \cap r_2 \mbb{B}^d_2}\sum_{i\in \mc{I}}\hat{w}_i \langle \vecx_i \vecx_i^\top  -\Sigma  ,\vecv \vecv^\top \rangle+\max_{\vecv \in s\mbb{B}^d_0 \cap r_2 \mbb{B}^d_2}\sum_{i\in \mc{O}}\hat{w}_i \langle \Sigma ,\vecv \vecv^\top \rangle  \nonumber\\
	&\stackrel{(b)}{\leq} \max_{M \in \mathfrak{M}_{r_1,r_2,d}^{\ell_1,\mr{Tr}} }\sum_{i=1}^n \hat{w}_i \langle \vecX_i \vecX_i^\top -\Sigma ,M\rangle +	 \max_{M \in \mathfrak{M}_{r_1,r_2,d}^{\ell_1,\mr{Tr}} }\sum_{i\in \mc{I}}\hat{w}_i \langle \vecx_i \vecx_i^\top  -\Sigma  ,M\rangle+\max_{\vecv \in s\mbb{B}^d_0 \cap r_2 \mbb{B}^d_2}\sum_{i\in \mc{O}}\hat{w}_i \langle \Sigma ,\vecv \vecv^\top \rangle,
\end{align}
where (a) follows from Lemma \ref{l:convl0l1} and $r_1 / r_2 \leq \sqrt{s}$, and (b) follows from $\mathfrak{M}_{r_1,r_2,d,\vecv}^{\ell_1,\ell_2}  \subset \mathfrak{M}_{r_1,r_2,d}^{\ell_1,\mr{Tr}} $ because $\mr{Tr}(M)=\mr{Tr}(\vecv \vecv^\top)=\|\vecv\|_2^2$ for $M \in \mathfrak{M}_{r_1,r_2,d,\vecv}^{\ell_1,\ell_2} $.
We note that, from $1/(1-\varepsilon)\leq 2$ and matrix H{\"o}lder's inequality,
\begin{align}
	\label{ine:optM-pre2}
	\max_{\vecv \in s\mbb{B}^d_0 \cap r_2 \mbb{B}^d_2}\sum_{i \in \mc{O}} \hat{w}_i \langle \Sigma,\vecv \vecv ^\top \rangle
	\stackrel{(a)}{\leq} 2\frac{o}{n}	\max_{\vecv \in s\mbb{B}^d_0 \cap r_2 \mbb{B}^d_2}\sum_{i \in \mc{O}}  \langle \Sigma,\vecv \vecv ^\top \rangle \stackrel{(b)}{\leq}  2 \kappa^2_{\mr{u}}r'_o  r_2^2.
\end{align}
where (a) follows from $\hat{w}_i'\leq 1/(1-\varepsilon)\leq 2$, and (b) follows from $o/n\leq 1/(5e)$ and the definition of $\kappa_{\mr{u}}$. 
From \eqref{ine:optM-pre} and \eqref{ine:optM-pre2}, we have
\begin{align}
	\label{ine:optM-pre4}
	\sum_{i\in \mc{O}}\hat{w}_i |\vecX_i^\top \vecv|^2 &\leq \max_{M\in\mathfrak{M}_{r_1,r_2,d}^{\ell_1,\mr{Tr}}  } \sum_{i=1}^n\hat{w}_i\langle \vecX_i \vecX_i^\top-\Sigma,M\rangle+ \max_{M\in\mathfrak{M}_{r_1,r_2,d}^{\ell_1,\mr{Tr}}  } \left|\sum_{i\in \mc{I}}\hat{w}_i\langle \vecx_i \vecx_i^\top-\Sigma,M\rangle\right|+ 2\kappa^2_{\mr{u}}r_o ' r_2^2\nonumber \\
	&\leq \tau_{\rm cut}+\max_{M\in\mathfrak{M}_{r_1,r_2,d}^{\ell_1,\mr{Tr}}  } \left|\sum_{i\in \mc{I}}\hat{w}_i\langle \vecx_i \vecx_i^\top-\Sigma,M\rangle\right|+2 \kappa^2_{\mr{u}}r_o'  r_2^2,
\end{align}
where the last line follows from the assumption of success of Algorithm \ref{alg:cw0}.\\

Next, we consider the second term of right-hand side of \eqref{ine:optM-pre4}. 
We have $\frac{1-o/n}{1-\varepsilon}\leq 2$ because $\varepsilon = c_\varepsilon \times \frac{o}{n}$ with $1\leq c_{\varepsilon}<2$ and $o/n \leq 1/(5e) (\leq 1/3 )$, and we have
\begin{align}
	\max_{M\in\mathfrak{M}_{r_1,r_2,d}^{\ell_1,\mr{Tr}}  } \left|\sum_{i\in \mc{I}}\hat{w}_i\langle \vecx_i \vecx_i^\top-\Sigma,M\rangle\right|
	&=\sum_{j \in \mc{I}}\hat{w}_j \max_{M\in\mathfrak{M}_{r_1,r_2,d}^{\ell_1,\mr{Tr}}  } \left|\sum_{i\in \mc{I}}\frac{\hat{w}_i}{\sum_{j \in \mc{I}}\hat{w}_j }\langle \vecx_i \vecx_i^\top-\Sigma,M\rangle\right| \nonumber \\
	&\leq \frac{1-o/n}{1-\varepsilon}\max_{M\in\mathfrak{M}_{r_1,r_2,d}^{\ell_1,\mr{Tr}}  } \left|\sum_{i\in \mc{I}}\frac{\hat{w}_i}{\sum_{j \in \mc{I}}\hat{w}_j }\langle \vecx_i \vecx_i^\top-\Sigma,M\rangle\right|\nonumber \\
	&\leq 2\max_{M\in\mathfrak{M}_{r_1,r_2,d}^{\ell_1,\mr{Tr}}  } \left|\sum_{i\in \mc{I}}\frac{\hat{w}_i}{\sum_{j \in \mc{I}}\hat{w}_j }\langle \vecx_i \vecx_i^\top-\Sigma,M\rangle\right|.
\end{align}
Define the following three sets:
\begin{align*}
	\Delta^{\mc{I}} (\varepsilon+o/n)&= \left\{(w_1,\dots,w_n) \mid 0\leq w_i\leq \frac{1}{n\{1-(\varepsilon+o/n)\}},\, \sum_{i \in \mc{I}} w_i = \sum_{i=1}^n w_i = 1,\, \right\},\nonumber\\
	 \Delta^{\mc{I}} (3o/n) &= \left\{(w_1,\dots,w_n) \mid 0\leq w_i\leq \frac{1}{n(1-3o/n)},\, \sum_{i \in \mc{I}} w_i = \sum_{i=1}^n w_i =  1\right\},\nonumber\\
	 \Delta^{n-1} (3o/n) &= \left\{(w_1,\dots,w_n) \mid 0\leq w_i\leq \frac{1}{n(1-3o/n)},\, \sum_{i=1}^n w_i =  1\right\}.
\end{align*}
From $\sum_{j \in \mc{I}} \hat{w}_j \geq 1-\frac{o}{n(1-\varepsilon)} = \frac{1-\varepsilon-o/n}{1-\varepsilon}$, for any $i \in \mc{I}$, we have $0\leq \frac{\hat{w}_i}{\sum_{j \in \mc{I}}\hat{w}_j }\leq \frac{1}{n(1-(\varepsilon+o/n))}$, and from $\varepsilon = c_\varepsilon \times o/n$  with $1\leq c_\varepsilon <2$, we have $\Delta^{\mc{I}} (\varepsilon+o/n) \subset \Delta^{\mc{I}} (3o/n) \subset \Delta^{n-1}(3o/n)$.
Therefore, we have
\begin{align}
	\max_{M\in\mathfrak{M}_{r_1,r_2,d}^{\ell_1,\mr{Tr}}  } \left|\sum_{i\in \mc{I}}\frac{\hat{w}_i}{\sum_{j \in \mc{I}}\hat{w}_j }\langle \vecx_i \vecx_i^\top-\Sigma,M\rangle\right|& \leq \max_{\vecw \in 		\Delta^{\mc{I}} (\varepsilon+o/n)}	\max_{M\in\mathfrak{M}_{r_1,r_2,d}^{\ell_1,\mr{Tr}}  } \left|\sum_{i\in \mc{I}} w_i \langle \vecx_i \vecx_i^\top-\Sigma,M\rangle\right|\nonumber\\
	&\leq  \max_{\vecw \in 		\Delta^{n-1} (3o/n)}	\max_{M\in\mathfrak{M}_{r_1,r_2,d}^{\ell_1,\mr{Tr}}  } \left|\sum_{i=1}^n w_i \langle \vecx_i \vecx_i^\top-\Sigma,M\rangle\right|.
\end{align}
From Lemma 1 of \cite{DalMin2022All} and $1/(1-3 o/n) \leq 1/(1-3/(5e))\leq 2$, we have
\begin{align}
	&\max_{\vecw \in \Delta^{n-1}(3o/n)}	\max_{M\in\mathfrak{M}_{r_1,r_2,d}^{\ell_1,\mr{Tr}}  } \left|\sum_{i=1}^n w_i \langle \vecx_i \vecx_i^\top-\Sigma,M\rangle\right| \nonumber\\
	&\quad \quad \leq  		\max_{|\mc{J}| = n-3o}	\max_{M\in\mathfrak{M}_{r_1,r_2,d}^{\ell_1,\mr{Tr}}  } \left|\sum_{i\in \mc{J}} \frac{1}{n(1-3o/n)} \langle \vecx_i \vecx_i^\top-\Sigma,M\rangle\right|\nonumber\\
	&\quad \quad \leq 2\max_{|\mc{J}| = n-3o}	\max_{M\in\mathfrak{M}_{r_1,r_2,d}^{\ell_1,\mr{Tr}}  } \left|\sum_{i \in \mc{J}} \frac{1}{n} \langle \vecx_i \vecx_i^\top-\Sigma,M\rangle\right|\nonumber\\
	&\quad \quad \leq  2\left(\max_{M\in\mathfrak{M}_{r_1,r_2,d}^{\ell_1,\mr{Tr}}  } \left|\sum_{i=1}^n\frac{\langle \vecx_i \vecx_i^\top-\Sigma,M\rangle}{n}\right|+\max_{M\in\mathfrak{M}_{r_1,r_2,d}^{\ell_1,\mr{Tr}}  } \left|\sum_{i\in \mc{J}}\frac{\langle \vecx_i \vecx_i^\top-\Sigma,M\rangle}{n}\right|\right).
\end{align}
Consequently, from almost the same calculation for \eqref{ine:optM0}, we have
\begin{align}
	\label{ine:optM-pre5}
	\max_{M\in\mathfrak{M}_{r_1,r_2,d}^{\ell_1,\mr{Tr}}  } \left|\sum_{i\in \mc{I}}\hat{w}_i\langle \vecx_i \vecx_i^\top-\Sigma,M\rangle\right|&\leq CL\kappa^2_{\mr{u}}\left(sr_{d,s}+r_\delta +r'_o \right)r^2_2.
\end{align}
Lastly, from \eqref{ine:Efirst}, \eqref{ine:optM-pre4} and \eqref{ine:optM-pre5}, we have
\begin{align}
	\label{ine:computation-out}
	\left|\sum_{i\in \mc{O}}\hat{w}_i u_i \vecX_i^\top \vecv \right|&\leq C\sqrt{\frac{o}{n}}\sqrt{\tau_{\rm cut}+(L\kappa)^2_{\mr{u}}\left(sr_{d,s}+r_\delta +r'_o \right)r^2_2+\kappa^2_{\mr{u}} r'_o r_2^2}\nonumber \\
	&\stackrel{(a)}{\leq} CL\sqrt{1+c_{\rm cut}}\sqrt{\frac{o}{n}}\sqrt{\kappa^2_{\mr{u}}sr_{d,s}r_2^2+\kappa^2_{\mr{u}}r_\delta r_2^2+\kappa^2_{\mr{u}} r_o' r_2^2}\nonumber \\
	&\stackrel{(b)}{\leq} CL\sqrt{1+c_{\rm cut}}\kappa_{\mr{u}}\left(\sqrt{\frac{o}{n}}\left(\sqrt{sr_{d,s}}+\sqrt{r_\delta}\right)+ r_o\right)r_2,
\end{align}
where (a) follows from the definition of $\tau_{\rm cut}$, and (b) follows from  triangular inequality, and the proof is complete.

\subsection{Proof of Proposition \ref{p:main:out2}}
From $\log \frac{n}{m}\geq 1$ and the arguments very similar to the proof of Proposition \ref{p:main:out}, we have
\begin{align}
	\sum_{i \in I_m}u_i \vecx_i^\top\vecv &\leq CL\sqrt{\frac{m}{n}}\sqrt{\kappa^2_{\mr{u}}\left(sr_{d,s}r_2^2+r_\delta r_2^2+\frac{m}{n}\log \frac{n}{m}r_2^2\right)+\kappa^2_{\mr{u}}\frac{m}{n}r_2^2}\nonumber\\
	&\stackrel{(a)}{\leq } CL\sqrt{\frac{(1+2c_{\varepsilon})o}{n}}\kappa_{\mr{u}}\sqrt{sr_{d,s}r_2^2+r_\delta r_2^2+\frac{(1+2c_{\varepsilon})o}{n}\log \frac{n}{(1+2c_{\varepsilon})o}r_2^2}\nonumber \\
	&\stackrel{(b)}{\leq}	 CL\sqrt{1+2c_\varepsilon }\kappa_{\mr{u}}\sqrt{\frac{ o}{n}}\sqrt{sr_{d,s}r_2^2+r_\delta r_2^2+\frac{o}{n}\log \frac{n}{o}r_2^2}.
\end{align}
where (a) follows from the fact that $0<x<1/e$, $x\log(1/x)$ is increasing, and (b) follows from $\log \frac{1}{1+2c_\varepsilon}\leq 1\leq \log \frac{n}{o}$.
From triangular inequality, the proof is complete.

\subsection{Proof of Proposition \ref{p:main:sc}}
First,  we introduce Lemma \ref{p:1e}, which is used in the proof of  Proposition \ref{p:main:sc}.
\begin{lemma}
\label{p:1e}
Suppose that (i) and (iii) of  Assumption \ref{a:intro} holds.
Define $\{a_i\}_{i=1}^n$ as a sequence of i.i.d. Rademacher random variables which are independent of $\{\vecx_i\}_{i=1}^n$.
Then, we have
\begin{align}
	\label{ine:gc-normale}
	&\mbb{E}\sup_{\vecv \in  r_1 \mbb{B}^d_1 \cap r_\Sigma \mbb{B}^d_\Sigma}\left|\frac{1}{n}\sum_{i=1}^n a_i \vecx_i^\top \vecv\right| \leq  L\rho r_{d,s}r_1.
\end{align}
\end{lemma}
\begin{proof}
	Define $\mbb{E}^*$ as the conditional expectation of $\{a_i\}_{i=1}^n$ given $\{\vecx_i\}_{i=1}^n$.
	From Exercise 2.2.2 of \cite{Tal2014Upper}, for any $\vecv_0 \in r_1 \mbb{B}^d_1 \cap r_\Sigma \mbb{B}^d_\Sigma$, we have
	\begin{align}
		\mbb{E}^*\sup_{\vecv \in  r_1 \mbb{B}^d_1 \cap r_\Sigma \mbb{B}^d_\Sigma}\left|\frac{1}{n}\sum_{i=1}^n a_i \vecx_i^\top \vecv\right| &\leq  	2\mbb{E}^*\sup_{\vecv \in  r_1 \mbb{B}^d_1 \cap r_\Sigma \mbb{B}^d_\Sigma}\frac{1}{n}\sum_{i=1}^n a_i \vecx_i^\top \vecv +		\mbb{E}\left|\frac{1}{n}\sum_{i=1}^n a_i \vecx_i^\top \vecv_0\right|,
	\end{align}
	and taking $\vecv_0 =0$, we have
	\begin{align}
		\label{ine:radsub0}
		\mbb{E}^*\sup_{\vecv \in  r_1 \mbb{B}^d_1 \cap r_\Sigma \mbb{B}^d_\Sigma}\left|\frac{1}{n}\sum_{i=1}^n a_i \vecx_i^\top \vecv\right| 
		&\leq 2\mbb{E}^*\sup_{\vecv \in  r_1 \mbb{B}^d_1 \cap r_\Sigma \mbb{B}^d_\Sigma}\frac{1}{n}\sum_{i=1}^n a_i \vecx_i^\top \vecv.
	\end{align}
	Taking the expectation with respect to $\{\vecx_i\}_{i=1}^n$ on both sides of \eqref{ine:radsub0}, we have
	\begin{align}
		\label{ine:radsub0-1}
		\mbb{E}\sup_{\vecv \in  r_1 \mbb{B}^d_1 \cap r_\Sigma \mbb{B}^d_\Sigma}\left|\frac{1}{n}\sum_{i=1}^n a_i \vecx_i^\top \vecv\right| &\leq  2\mbb{E}\sup_{\vecv \in  r_1 \mbb{B}^d_1 \cap r_\Sigma \mbb{B}^d_\Sigma}\frac{1}{n}\sum_{i=1}^n a_i \vecx_i^\top \vecv.
	\end{align}
	For any $i$ and  any $\vecv_1, \vecv_2 \in r_1 \mbb{B}^d_1 \cap r_\Sigma \mbb{B}^d_\Sigma$ we have
	\begin{align}
		\label{ine:radsub}
		\|\langle a_i\vecx_i,\vecv_1-\vecv_2\rangle\|_{\psi_2}\stackrel{(a)}{\leq} 	\|\langle \vecx_i,\vecv_1-\vecv_2\rangle\|_{\psi_2}\stackrel{(b)}{\leq} \mathfrak{L}	\|\langle \vecx_i,\vecv_1-\vecv_2\rangle\|_{L_2},
	\end{align}
	where (a) follows from $|a_i|=1$ and the definition of $\|\cdot\|_{\psi_2}$ and (b) follows from \eqref{ine:lsgtmp}, and we see that $\langle a_i\vecx_i,\vecv\rangle$ is a subGaussian random variable
	Then, from Proposition 2.6.1 of \cite{Ver2018High}, we have
	\begin{align}
		\label{ine:radsub2}
		\left\|\left\langle \frac{1}{n}\sum_{i=1}^n a_i \vecx_i,\vecv_1-\vecv_2\right\rangle\right\|_{\psi_2}^2\leq C \sum_{i=1}^n		\left\|\left\langle \frac{a_i\vecx_i}{n} ,\vecv_1-\vecv_2\right\rangle\right\|_{\psi_2}^2\stackrel{(a)}{\leq} C \frac{\mathfrak{L}^2	}{n}	\|\left\langle \vecx,\vecv_1-\vecv_2\right\rangle\|_{L_2}^2,
	\end{align}
	where (a) follows from \eqref{ine:radsub}.
	From the assumption on $\vecx$, we have
	\begin{align}
				\label{ine:radsub3}
		\|\left\langle \vecx,\vecv_1-\vecv_2\right\rangle\|_{L_2}^2 = \| \langle \Sigma^\frac{1}{2},\vecv_1-\vecv_2 \rangle\|_2^2=\|\langle \Sigma^\frac{1}{2}\vecg,\vecv_1-\vecv_2\rangle\|_{L_2}^2.
	\end{align}
From \eqref{ine:radsub2} and \eqref{ine:radsub3}, we have
	\begin{align}
		\label{ine:radsub4}
		\left\|\left\langle \frac{1}{n}\sum_{i=1}^na_i \vecx_i,\vecv_1-\vecv_2\right\rangle\right\|_{\psi_2}\leq C \frac{\mathfrak{L}}{\sqrt{n}}		\|\langle\Sigma^\frac{1}{2}\vecg,\vecv_1-\vecv_2\rangle\|_{L_2}.
	\end{align}
Then, from Corollary 8.6.2 of \cite{Ver2018High}, we have
\begin{align}
	\label{ine:radsub5}
	\mbb{E}\sup_{\vecv \in  r_1 \mbb{B}^d_1 \cap r_\Sigma \mbb{B}^d_\Sigma}\left\langle \frac{1}{n}\sum_{i=1}^na_i \vecx_i,\vecv\right\rangle\leq C \frac{\mathfrak{L}}{\sqrt{n}}		\mbb{E}\sup_{\vecv \in  r_1 \mbb{B}^d_1 \cap r_\Sigma \mbb{B}^d_\Sigma}\langle\Sigma^\frac{1}{2}\vecg,\vecv\rangle.
\end{align}
	Then, from Lemma \ref{l:gwslope:main} and \eqref{ine:radsub0-1}, the proof is complete.
\end{proof}

Then, we proceed the proof of Proposition \ref{p:main:sc}.
 The left-hand side  of \eqref{ine:sc} divided by $\lambda_o^2$ can be expressed as
\begin{align}
	 \sum_{i=1}^n \left(-h(\xi_{\lambda_o,i}-x_{\vecv,i})+h \left(\xi_{\lambda_o,i}\right) \right)x_{\vecv,i}.
\end{align}
From the convexity of Huber loss, for any $a,b \in \mbb{R}$, we have
\begin{align}
	H(a)-H(b)\geq h(b)(a-b)\,\text{  and  }\,H(b)-H(a) \geq h(a)(b-a),
\end{align}
and we have
\begin{align}
	\label{ine:positiveHuber}
	0\leq \left(h(a)-h(b)\right)(a-b).
\end{align}
Let $\mr{I}_{E_i}$ be the indicator function of the event
\begin{align}
	E_i := ( |\xi_{\lambda_o,i}| \leq 1/2) \cap ( |x_{\vecv,i}| \leq 1/2).
\end{align}
From \eqref{ine:positiveHuber},
we have
\begin{align}
	&\sum_{i=1}^n \left(-h(\xi_{\lambda_o,i}-x_{\vecv,i})+h \left(\xi_{\lambda_o,i}\right) \right)x_{\vecv,i} \geq \sum_{i=1}^n \left(-h(\xi_{\lambda_o,i}-x_{\vecv,i})+h \left(\xi_{\lambda_o,i}\right) \right)x_{\vecv,i}\mr{I}_{E_i}= \sum_{i=1}^n x_{\vecv,i}^2\mr{I}_{E_i}.
\end{align}
Define the functions
\begin{align}
\label{def:phipsi}
	\varphi(x) =\begin{cases}
	x^2 & \mbox{ if } |x| \leq 1/4\\
	(x-1/4)^2 & \mbox{ if } 1/4\leq x \leq 1/2 \\
	(x+1/4)^2 & \mbox{ if } -1/2\leq x \leq -1/4 \\
	0 & \mbox{ if } |x| >1/2
\end{cases} ~\mbox{ and }~
	\psi(x) = \mr{I}_{(|x| \leq 1/2 ) },
\end{align}
where $ \mr{I}_{(|x| \leq 1/2 ) }$ is the indicator function of the event $|x| \leq 1/2 $.
Let $f_i(\vecv) = \varphi(x_{\vecv,i}) \psi(\xi_{\lambda_o,i})$
and we have
\begin{align}
\label{ine:huv-conv-f}
	 \sum_{i=1}^n x_{\vecv,i}^2\mr{I}_{E_i}\stackrel{(a)}{\geq} \sum_{i=1}^n \varphi(x_{\vecv,i}) \psi(\xi_{\lambda_o,i})=\sum_{i=1}^n f_i(\vecv),
\end{align}
where (a) follows from $\varphi(v) \leq v^2$ for $|v| \leq 1/2$. We note that
\begin{align}
\label{ine:f-1/4}
	f_i(\vecv) \leq\varphi(x_{\vecv,i}) \leq \min\left(x_{\vecv,i}^2,1\right).
\end{align}
To bound $\sum_{i=1}^n f_i(\vecv)$ from below, we have
\begin{align}
\label{ine:fbelow}
\inf_{\vecv \in r_1 \mbb{B}^d_1 \cap r_\Sigma \mbb{B}^d_\Sigma}\sum_{i=1}^n f_i(\vecv)&\geq \inf_{\vecv \in  r_1 \mbb{B}^d_1 \cap r_\Sigma \mbb{B}^d_\Sigma}\mbb{E}\sum_{i=1}^nf(\vecv) -\sup_{\vecv \in  r_1 \mbb{B}^d_1 \cap r_\Sigma \mbb{B}^d_\Sigma} \Big|\sum_{i=1}^n f_i(\vecv)-\mbb{E}\sum_{i=1}^n f_i(\vecv)\Big|
\end{align}
Define the supremum of a random process indexed by $ r_1 \mbb{B}^d_1 \cap r_\Sigma \mbb{B}^d_\Sigma$:
\begin{align}
\label{ap:delta}
	\Delta := \sup_{ \vecv \in  r_1 \mbb{B}^d_1 \cap r_\Sigma \mbb{B}^d_\Sigma} \left| \sum_{i=1}^n f_i(\vecv) - \mbb{E}\sum_{i=1}^n f_i	(\vecv) \right| . 
\end{align}
Define
\begin{align}
	\mr{I}_{|x_{\vecv,i}| \geq 1/2 }\text{ and }\mr{I}_{|\xi_{\lambda_o,i}|\geq 1/2 }
\end{align}
as the indicator functions of the events $|x_{\vecv,i}| \geq 1/2 $  and $|\xi_{\lambda_o,i}|\geq 1/2 $, respectively.
From  $\mr{I}_{E_i}=1-\mr{I}_{|x_{\vecv,i}|\geq 1/2}-\mr{I}_{|\xi_{\lambda_o,i}|> 1/2}$ and \eqref{ine:huv-conv-f}, we have
\begin{align}
\label{ine:aplower:tmp}
\sum_{i=1}^n x_{\vecv,i}^2\mr{I}_{E_i}\geq \mbb{E}\sum_{i=1}^n f_i(\vecv)&\geq \sum_{i=1}^n\mbb{E} x_{\vecv,i}^2- \sum_{i=1}^n\mbb{E}x_{\vecv,i}^2 \mr{I}_{|x_{\vecv,i}| \geq 1/2 }- \sum_{i=1}^n\mbb{E}x_{\vecv,i}^2 \mr{I}_{|\xi_{\lambda_o,i}|\geq 1/2 }\nonumber \\
& \geq \frac{\|\Sigma^\frac{1}{2}\vecv\|_2^2}{\lambda_o^2} - \sum_{i=1}^n\mbb{E}x_{\vecv,i}^2 \mr{I}_{|x_{\vecv,i}| \geq 1/2 }- \sum_{i=1}^n\mbb{E}x_{\vecv,i}^2 \mr{I}_{|\xi_{\lambda_o,i}|\geq 1/2 }.
\end{align}
We note that, from \eqref{cl}
\begin{align}
	\label{ine:v3}
		\mbb{E}(\vecx_i^\top \vecv)^4\leq 16L^4  \|\Sigma^{\frac{1}{2}}\vecv\|_2^4.
\end{align}
We evaluate the right-hand side of \eqref{ine:aplower:tmp} at each term.
First, for any $\vecv \in  r_1 \mbb{B}^d_1 \cap r_\Sigma \mbb{B}^d_\Sigma$, we have
\begin{align}
\label{ap:ine:cov1}
	\sum_{i=1}^n\mbb{E}x_{\vecv,i}^2\mr{I}_{|x_{\vecv,i}| \geq 1/2 }
	&\stackrel{(a)}{\leq} 	\sum_{i=1}^n\sqrt{\mbb{E} x_{\vecv,i}^4 } \sqrt{\mbb{E}\mr{I}_{|x_{\vecv,i}| \geq 1/2 }}\nonumber \\
	&\stackrel{(b)}{=} 	\sum_{i=1}^n\sqrt{\mbb{E} x_{\vecv,i}^4 } \sqrt{\mbb{P}\left(|x_{\vecv,i}| \geq 1/2 \right)}\nonumber \\
	&\stackrel{(c)}{\leq } 	\sum_{i=1}^n\sqrt{\mbb{E} x_{\vecv,i}^4 } \sqrt{2 \exp \left(-\frac{\lambda_o^2n}{c_{\mathfrak{L}}^2\mathfrak{L}^2\|\Sigma^\frac{1}{2}\vecv\|_2^2}\right)}\nonumber \\
	&\stackrel{(d)}{\leq}	\frac{4}{\lambda_o^2}L^2\|\Sigma^\frac{1}{2}\vecv\|_2^2 \sqrt{2 \exp \left(-\frac{\lambda_o^2n}{L^2\|\Sigma^\frac{1}{2}\vecv\|_2^2}\right)}\nonumber \\
	&\stackrel{(e)}{\leq}\frac{1}{3\lambda_o^2} \|\Sigma^\frac{1}{2}\vecv\|_2^2 ,
\end{align}
	where (a) follows from H{\"o}lder's inequality, (b) follows from the relation between indicator function and expectation, (c) follows from \eqref{ine:lsg3}, (d) follows from \eqref{ine:v3} and the definition of $L$, and (e) follows from \eqref{ine:nsufficient}.
Second, for any $\vecv \in r_1 \mbb{B}^d_1 \cap r_\Sigma \mbb{B}^d_\Sigma$, we have
\begin{align}
\label{ap:ine:cov2}
	\sum_{i=1}^n\mbb{E} x_{\vecv,i}^2 \mr{I}_{|\xi_{\lambda_o,i}|\geq 1/2 } 
	&\stackrel{(a)}{\leq} \sum_{i=1}^n\sqrt{\mbb{E} x_{\vecv,i}^4} \sqrt{\mbb{E}\mr{I}_{|\xi_{\lambda_o,i}|\geq 1/2 }}\nonumber \\
	&\stackrel{(b)}{=}\sum_{i=1}^n\sqrt{\mbb{E} x_{\vecv,i}^4} \sqrt{\mbb{P}( \left|\xi_{\lambda_o,i} \right|\geq 1/2)}\nonumber \\
	&= \frac{1}{\lambda_o^2 }\sqrt{ \mbb{E}(\vecv^\top \vecx_i)^4} \sqrt{\mbb{P}( \left|\xi_{\lambda_o,i} \right|\geq 1/2)}\nonumber \\
	&\stackrel{(c)}{\leq} \frac{4L^2}{\lambda_o^2 }\sqrt{\mbb{P}( \left|\xi_{\lambda_o,i} \right|\geq 1/2)}\|\Sigma^\frac{1}{2}\vecv\|_2^2\nonumber \\
	&\stackrel{(d)}{\leq} \frac{1}{3\lambda_o^2} 	\|\Sigma^\frac{1}{2}\vecv\|_2^2,
\end{align}
where (a) follows from H{\"o}lder's inequality, (b) follows from relation between indicator function and expectation, and (c) follows from \eqref{ine:v3}, and (d) follows from \eqref{ine:lambda}.
Consequently, from \eqref{ap:ine:cov1} and \eqref{ap:ine:cov2} we have	
\begin{align}
\label{ap:h_bellow}
\frac{\|\Sigma^\frac{1}{2}\vecv\|_2^2}{3\lambda_o^2}-\Delta\leq \sum_{i=1}^n \left(-h(\xi_{\lambda_o,i}-x_{\vecv,i}) +h \left(\xi_{\lambda_o,i}\right) \right)x_{\vecv,i}.
\end{align}

Next, we evaluate the stochastic term $\Delta$ defined in \eqref{ap:delta}. 
From \eqref{ine:f-1/4} and Theorem 3 of \cite{Mas2000Constants}, with probability at least $1-\delta$, we have
\begin{align}
\label{ine:delta}
	\Delta & \leq 2 \mbb{E} \Delta + \sigma_f \sqrt{8\log(1/\delta)} + 18\log(1/\delta),
\end{align}
where 
\begin{align}
	\sigma^2_f= \sup_{\vecv \in r_1 \mbb{B}^d_1 \cap r_\Sigma \mbb{B}^d_\Sigma} \sum_{i=1}^n\mbb{E} (f_i(\vecv)-\mbb{E}f_i(\vecv))^2.
\end{align}
From \eqref{ine:f-1/4} and \eqref{ine:v3},  we have
\begin{align}
\label{ap:ine:cov3}
\mbb{E}(f_i(\vecv)-\mbb{E}f_i(\vecv))^2 \leq \mbb{E}f_i^2(\vecv) \leq \mbb{E}f_i(\vecv) \leq \mbb{E}x_{\vecv,i}^2\leq \frac{L^2\|\Sigma^\frac{1}{2}\vecv\|_2^2}{\lambda_o^2n}.
\end{align}
Combining this and \eqref{ine:delta}, 	with probability at least $1-\delta$, we have
\begin{align}
\label{ap:delta_upper}
	\Delta &\leq 2 \mbb{E} \Delta+\frac{L}{\lambda_o} \sqrt{8\log(1/\delta)} \|\Sigma^\frac{1}{2}\vecv\|_2+ 18\log(1/\delta).
\end{align}

From symmetrization inequality (Theorem 11.4 of \cite{BouLugMas2013concentration}), we have $\mbb{E}\Delta \leq 2 \,\mbb{E} \sup_{ \vecv \in  r_1 \mbb{B}^d_1 \cap r_\Sigma \mbb{B}^d_\Sigma} | \mathbb{G}_{\vecv} |  \leq 2 \,\mbb{E} \sup_{ \vecv \in  r_1 \mbb{B}^d_1 \cap r_\Sigma \mbb{B}^d_\Sigma} | \mathbb{G}_{\vecv} | $,
where 
\begin{align}
	\mbb{G}_{\vecv} := \sum_{i=1}^n 
	a_i \varphi (x_{\vecv,i}) \psi (\xi_{\lambda_o,i}),
\end{align} 
and $\{a_i\}_{i=1}^n$ is a sequence of i.i.d. Rademacher random variables which are independent of $\{\vecx_i,\xi_i\}_{i=1}^n$.
We denote $\mbb{E}^*$ as a conditional expectation of $\left\{a_i\right\}_{i=1}^n$ given $\left\{\vecx_i,\xi_i\right\}_{i=1}^n$. Since $\varphi$ is $1$-Lipschitz and $\varphi(0)=0$, from contraction principle (Theorem 11.5 of \cite{BouLugMas2013concentration}), we have
\begin{align}
	\mbb{E} ^*\sup_{\vecv\in  r_1 \mbb{B}^d_1 \cap r_\Sigma \mbb{B}^d_\Sigma} \left| \sum_{i=1}^n a_i \varphi (x_{\vecv,i}) \psi (\xi_{\lambda_o,i}) \right| \leq	\frac{1}{2\lambda_o\sqrt{n}}	\mbb{E}^*\sup_{\vecv\in  r_1 \mbb{B}^d_1 \cap r_\Sigma \mbb{B}^d_\Sigma} \left| \sum_{i=1}^n a_i \vecx_i^\top \vecv \right|.
\end{align}	
and from the basic property of the expectation, we have
\begin{align}
	\mbb{E} \sup_{\vecv\in  r_1 \mbb{B}^d_1 \cap r_\Sigma \mbb{B}^d_\Sigma} \left| \sum_{i=1}^n a_i \varphi (x_{\vecv,i}) \psi (\xi_{\lambda_o,i}) \right| \leq	\frac{1}{2\lambda_o\sqrt{n}}	\mbb{E}\sup_{\vecv\in  r_1 \mbb{B}^d_1 \cap r_\Sigma \mbb{B}^d_\Sigma} \left| \sum_{i=1}^n a_i \vecx_i^\top \vecv \right|.
\end{align}	
From Lemma \ref{p:1e}, we have 
\begin{align}
\label{ine:hub-stoc-upper}
\lambda_o^2 \mbb{E}\Delta \leq \lambda_o^2\mbb{E}\sup_{\vecv\in  r_1 \mbb{B}^d_1 \cap r_\Sigma \mbb{B}^d_\Sigma} \left| \sum_{i=1}^n a_i \varphi (x_{\vecv,i}) \psi (\xi_{\lambda_o,i} ) \right| \leq CL\lambda_o\sqrt{n}\rho r_1r_{d,s} .
\end{align}
From  \eqref{ap:delta_upper} and \eqref{ine:hub-stoc-upper}, we have
\begin{align}
	\label{ap:delta_upper2}
		\lambda_o^2\Delta &\leq  CL\lambda_o\sqrt{n}\rho r_{d,s}r_1 +CL\lambda_o  \sqrt{\log(1/\delta)} r_{\Sigma}+C \lambda_o^2n r_\delta^2\nonumber\\
		&\stackrel{(a)}{\leq} C L\lambda_o\sqrt{n}(\rho r_{d,s} r_1+  r_\delta r_\Sigma),
	\end{align}
	where (a) follows from $\lambda_o \sqrt{n} r_\delta\leq r_\Sigma$  and $L\geq 1$.
	From \eqref{ap:h_bellow} and \eqref{ap:delta_upper2}, we have
	\begin{align}
		\label{ap:h_bellow2}
		\inf_{\vecv \in  r_1 \mbb{B}^d_1 \cap r_\Sigma \mbb{B}^d_\Sigma}\lambda_o^2 \sum_{i=1}^n \left(-h(\xi_{\lambda_o,i}-x_{\vecv,i}) +h \left(\xi_{\lambda_o,i}\right) \right)x_{\vecv,i}\geq \frac{\|\Sigma^\frac{1}{2}\vecv\|_2^2}{3}-C L\lambda_o\sqrt{n}(\rho r_{d,s} r_1+ r_\delta r_\Sigma),
		\end{align}
	and the proof is complete.

	\subsection{Proof of Proposition \ref{p:main:out-un}}
	We note that
	\begin{align}
		\left|\sum_{i \in \mc{O}}\hat{w}'_iu_i \vecX_i^\top\vecv \right|^2 &\stackrel{(a)}{\leq} 4\frac{o}{n}\sum_{i \in \mc{O}}^n\hat{w}'_i |\vecX_i^\top \vecv|^2\stackrel{(b)}{\leq} 8 \frac{o}{n}\sum_{i=1}^n\hat{w}_i |\vecX_i^\top \vecv|^2,
	\end{align}
	where (a) follows from H{\"o}lder's inequality and $\sum_{i \in \mc{O}}u_i^2 \leq 4o$, and (b) follows from the fact that $\hat{w}_i'\leq 2\hat{w}_i$ for any $i\in (1,\cdots,n)$.
	We focus on $\sum_{i =1}^n\hat{w}_i |\vecX_i^\top \vecv|^2$.
	First, for any $\vecv \in  r_1 \mbb{B}^d_1 \cap r_2 \mbb{B}^d_2$,  we have
	\begin{align}
		\sum_{i=1}^n\hat{w}_i |\vecX_i^\top \vecv|^2 &\stackrel{(a)}{\leq} \max_{M \in\mathfrak{M}_{r_1,r_2,d}^{\ell_1,\mr{Tr}}  } \sum_{i=1}^n\hat{w}_i\langle \vecX_i \vecX_i^\top ,M\rangle\stackrel{(b)}{\leq} \tau_{\rm cut},
	\end{align}
	where (a) follows from the fact that $\mathfrak{M}_{r_1,r_2,d,\vecv} ^{\ell_1,\ell_2}\subset \mathfrak{M}_{r_1,r_2,d}^{\ell_1,\mr{Tr}}$ and (b) follows from the succeeding condition of Algorithm \ref{alg:cw0-un}.
	
	Combining the arguments above, we have
	\begin{align}
		\sum_{i\in \mc{O}}\hat{w}_i \vecX_i^\top \vecv &\leq CL\sqrt{c_{\rm cut}}\sqrt{\frac{o}{n}}\sqrt{\kappa_{\mr{u}}^2\left(sr_{d,s}+r_\delta \right)+\Sigma_{\max}^2}r_2,
	\end{align}
	where we use the definition of $\tau_{\rm cut}$
	and the proof is complete.

	\subsection{Proof of Proposition \ref{p:main:out2-un}}
	We note that, from H{\"o}lder's inequality, we have 
	\begin{align}
		\left|\sum_{i \in I_m}\frac{u_i\vecx_i^\top\vecv }{n} \right|^2 &\stackrel{(a)}{\leq} \sum_{i \in I_m}\frac{1}{n}u_i^2 \sum_{i=1}^n\frac{1}{n}|\vecx_i^\top\vecv|^2\stackrel{(b)}{\leq} 4\frac{m}{n}\sum_{i=1}^n\frac{1}{n} |\vecx_i^\top\vecv|^2,
	\end{align}
	where (a) follows from H{\"o}lder's inequality, and (b) follows from the fact that $\|\vecu\|_2^2\leq 4m$.
	From the fact that $\mathfrak{M}_{r_1,r_2,d,\vecv} ^{\ell_1,\ell_2}\subset \mathfrak{M}_{r_1,r_2,d}^{\ell_1,\mr{Tr}} $, we have
	\begin{align}
		\sum_{i =1}^n\frac{(\vecx_i^\top\vecv)^2 }{n} &\leq \sup_{M\in\mathfrak{M}_{r_1,r_2,d}^{\ell_1,\mr{Tr}}  }\sum_{i =1}^n\frac{\langle \vecx_i \vecx_i^\top ,M\rangle}{n}.
	\end{align}
	From Corollary \ref{c:cwpre} and triangular inequality, we have 
	\begin{align}
		\sum_{i \in I_m}u_i \vecx_i^\top\vecv &\leq \sqrt{c_1' \frac{m}{n}}\sqrt{(L\kappa_{\mr{u}})^2\left(sr_{d,s}+r_\delta \right)+\Sigma_{\max}^2}r_2\nonumber\\
		&\stackrel{(a)}{\leq} CL\sqrt{\frac{o}{n}}\sqrt{\kappa_{\mr{u}}^2\left(sr_{d,s}+r_\delta \right)+\Sigma_{\max}^2}r_2,
	\end{align}
	where (a) follows from $m\leq (2c_\varepsilon+1)o$ and $L\geq 1$, and the proof is complete.

\section{Proof of Lemmas \ref{1calcknown}, \ref{3calcknown}, \ref{1calcunknown} and \ref{3calcunknown} }
\label{sec:calc}
\begin{proof}[Proof of Lemma \ref{1calcknown}]
From the triangular inequality, we have
\begin{align}
	\left|\sum_{i=1}^n \hat{w}_i'h(r_{\vecbeta^*,i}) \vecX_i^\top \vectheta_\eta \right| \nonumber &\leq \left|\sum_{i\in \mc{I}} \hat{w}_i'h(r_{\vecbeta^*,i}) \vecX_i^\top \vectheta_\eta \right|+\left|\sum_{i\in \mc{O}} \hat{w}_i'h(r_{\vecbeta^*,i}) \vecX_i^\top \vectheta_\eta \right|\nonumber \\
	&\leq \left|\sum_{i\in \mc{I}} \hat{w}_i'h(r_{\vecbeta^*,i}) \vecx_i^\top \vectheta_\eta \right|+\left|\sum_{i\in \mc{O}} \hat{w}_i'h(r_{\vecbeta^*,i}) \vecX_i^\top \vectheta_\eta \right|\nonumber\\
	&= \left|\sum_{i=1}^n \frac{1}{n}h(r_{\vecbeta^*,i}) \vecx_i^\top \vectheta_\eta -\sum_{i \in \mc{O}\cup (\mc{I}\cap I_< )} \frac{1}{n}h(r_{\vecbeta^*,i}) \vecx_i^\top \vectheta_\eta\right|+\left|\sum_{i\in \mc{O}} \hat{w}_i'h(r_{\vecbeta^*,i}) \vecX_i^\top \vectheta_\eta \right|\nonumber\\
	&\leq \left|\sum_{i=1}^n \frac{1}{n}h(r_{\vecbeta^*,i}) \vecx_i^\top \vectheta_\eta\right| +\left|\sum_{i\in \mc{O}} \hat{w}_i'h(r_{\vecbeta^*,i}) \vecX_i^\top \vectheta_\eta \right|+\left|\sum_{i \in \mc{O}\cup (\mc{I}\cap I_<)} \frac{1}{n}h(r_{\vecbeta^*,i}) \vecx_i^\top \vectheta_\eta\right|.
\end{align}
We note that $|h(\cdot)|\leq 1$ and from Lemma \ref{l:w2}, $|\mc{O} \cup \left(\mc{I} \cap I_{<}\right)|\leq (1+2c_\varepsilon)o$.
Therefore, from Propositions \ref{p:main1} - \ref{p:main:out2}, we have 
\begin{align}
	\left|\sum_{i=1}^n \frac{1}{n}h(r_{\vecbeta^*,i}) \vecx_i^\top \vectheta_\eta\right|
	&\leq c_3L\left(\rho r_{d,s}r_1+r_\delta r_\Sigma\right)\nonumber\\
	\left|\sum_{i \in \mc{O}} \frac{1}{n}h(r_{\vecbeta^*,i}) \vecx_i^\top \vectheta_\eta\right|
	&\leq c_4L\sqrt{1+c_{\rm cut}}\left(\kappa_{\mr{u}}\sqrt{\frac{o}{n}}\left(\sqrt{sr_{d,s}}+\sqrt{r_\delta}\right)r_2+\kappa_{\mr{u}}r_or_2\right)\nonumber\\
	\left|\sum_{i \in \mc{O}\cup (\mc{I}\cap I_<)} \frac{1}{n}h(r_{\vecbeta^*,i}) \vecx_i^\top \vectheta_\eta\right|
	&\leq c_5L\left(\kappa_{\mr{u}}\sqrt{\frac{o}{n}}\left(\sqrt{sr_{d,s}}+\sqrt{r_\delta}\right)r_2+\kappa_{\mr{u}}r_or_2\right),
\end{align}
and we see that 
\begin{align}
	\label{a:ine:step1-1-1}
	\left|\sum_{i=1}^n \hat{w}_i'h(r_{\vecbeta^*,i}) \vecX_i^\top \vectheta_\eta \right| 
	&\leq 3c_{\max}^2 L\left(\rho r_{d,s}r_1+r_\delta r_\Sigma+\kappa_{\mr{u}}\sqrt{\frac{o}{n}}\left(\sqrt{sr_{d,s}}+\sqrt{r_\delta}\right)r_2+\kappa_{\mr{u}}r_or_2\right).
\end{align}
\end{proof}

\begin{proof}[Proof of Lemma \ref{3calcknown}]
We have
\begin{align}
	&\sum_{i=1}^n \lambda_o\sqrt{n} \hat{w}_i'\left(-h(r_{\vecbeta^*+\vectheta_\eta,i}) +h(r_{\vecbeta^*,i})\right) \vecX_i^\top \vectheta_\eta \nonumber\\
	&= \sum_{i\in \mc{I}} \lambda_o\sqrt{n} \hat{w}_i'\left(-h(r_{\vecbeta^*+\vectheta_\eta,i}) +h(r_{\vecbeta^*,i})\right) \vecX_i^\top \vectheta_\eta +\sum_{i \in \mc{O}} \lambda_o\sqrt{n} \hat{w}_i'\left(-h(r_{\vecbeta^*+\vectheta_\eta,i}) +h(r_{\vecbeta^*,i})\right) \vecX_i^\top \vectheta_\eta\nonumber\\
	&= \sum_{i\in \mc{I}} \lambda_o\sqrt{n} \hat{w}_i'\left(	-h (\xi_{\lambda_o,i}-x_{\vectheta_\eta,i})+h (\xi_{\lambda_o,i})\right) \vecx_i^\top \vectheta_\eta +\sum_{i \in \mc{O}} \lambda_o\sqrt{n} \hat{w}_i'\left(-h(r_{\vecbeta^*+\vectheta_\eta,i}) +h(r_{\vecbeta^*,i})\right) \vecX_i^\top \vectheta_\eta\nonumber\\
	&= \left(\sum_{i=1}^n -\sum_{i\in \mc{O} \cup \left(\mc{I} \cap I_{<}\right)}\right) \frac{\lambda_o}{\sqrt{n}} \left(-h (\xi_{\lambda_o,i}-x_{\vectheta_\eta,i})+h (\xi_{\lambda_o,i})\right) \vecx_i^\top\vectheta_\eta \nonumber\\
	&\quad +\sum_{i \in \mc{O}} \lambda_o\sqrt{n} \hat{w}_i'\left(-h(r_{\vecbeta^*+\vectheta_\eta,i}) +h(r_{\vecbeta^*,i})\right) \vecX_i^\top \vectheta_\eta\nonumber\\
	&\geq \sum_{i=1}^n \frac{\lambda_o}{\sqrt{n}} \left(	-h (\xi_{\lambda_o,i}-x_{\vectheta_\eta,i})+h (\xi_{\lambda_o,i})\right) \vecx_i^\top\vectheta_\eta-\left|\sum_{i\in \mc{O} \cup \left(\mc{I} \cap I_{<}\right)}\frac{\lambda_o}{\sqrt{n}} \left(	-h (\xi_{\lambda_o,i}-x_{\vectheta_\eta,i})+h (\xi_{\lambda_o,i})\right)\vecx_i^\top\vectheta_\eta\right|\nonumber \\
	&\quad -\left|\sum_{i\in \mc{O}} \lambda_o\sqrt{n} \hat{w}_i'\left(-h(r_{\vecbeta^*+\vectheta_\eta,i}) +h(r_{\vecbeta^*,i})\right) \vecX_i^\top \vectheta_\eta\right|.
\end{align}
We note that $|h(\cdot)|\leq 1$ and from Lemma \ref{l:w2}, $| \mc{O} \cup \left(\mc{I} \cap I_{<}\right)|\leq (1+2c_\varepsilon)o$. From Propositions \ref{p:main:sc}, Propositions \ref{p:main:out} and \ref{p:main:out2},  we have
\begin{align}
	\sum_{i=1}^n \frac{\lambda_o}{\sqrt{n}} \left(	-h (\xi_{\lambda_o,i}-x_{\vectheta_\eta,i})+h (\xi_{\lambda_o,i})\right)\vecx_i^\top\vectheta_\eta &\geq \frac{\|\Sigma^\frac{1}{2}\vecv\|_2^2}{3}r_2^2-c_{\max}L\lambda_o\sqrt{n}\left(\rho r_{d,s} r_1+r_\delta r_\Sigma\right)\nonumber\\
	\left|\sum_{i\in \mc{O}} \lambda_o\sqrt{n} \hat{w}_i'\left(-h(r_{\vecbeta^*+\vectheta_\eta,i}) +h(r_{\vecbeta^*,i})\right) \vecX_i^\top \vectheta_\eta\right|
	& \leq c_{\max}^2L\lambda_o\sqrt{n}\left(\kappa_{\mr{u}}\sqrt{\frac{o}{n}}\left(\sqrt{sr_{d,s}}+\sqrt{r_\delta}\right)r_2+\kappa_{\mr{u}}r_or_2\right)\nonumber \\
	\left|\sum_{i\in \mc{O} \cup \left(\mc{I} \cap I_{<}\right)}\frac{\lambda_o}{\sqrt{n}} \left(	-h (\xi_{\lambda_o,i}-x_{\vectheta_\eta,i})+h (\xi_{\lambda_o,i})\right) \vecx_i^\top\vectheta_\eta\right|
	& \leq c_{\max}^2L\lambda_o\sqrt{n}\left(\kappa_{\mr{u}}\sqrt{\frac{o}{n}}\left(\sqrt{sr_{d,s}}+\sqrt{r_\delta}\right)r_2+\kappa_{\mr{u}}r_or_2\right).
\end{align}
Combining the arguments above, we see that 
\begin{align}
&\sum_{i=1}^n \lambda_o\sqrt{n} \hat{w}_i'\left(-h(r_{\vecbeta^*+\vectheta_\eta,i}) +h(r_{\vecbeta^*,i})\right) \vecX_i^\top \vectheta_\eta\nonumber \\
&\geq  \frac{\|\Sigma^\frac{1}{2}\vecv\|_2^2}{3}-c_{\max}^2L\lambda_o\sqrt{n}\left(\rho r_{d,s} r_1+r_\delta r_\Sigma+\kappa_{\mr{u}}\sqrt{\frac{o}{n}}\left(\sqrt{sr_{d,s}}+\sqrt{r_\delta}\right)r_2+\kappa_{\mr{u}}r_or_2\right) ,
\end{align}
and  the proof is complete.
\end{proof}

\begin{proof}[Proof of Lemma \ref{1calcunknown}]
From the triangular inequality, we have
\begin{align}
	\left|\sum_{i=1}^n \hat{w}_i'h(r_{\vecbeta^*,i}) \vecX_i^\top \vectheta_\eta \right| \nonumber &\leq \left|\sum_{i\in \mc{I}} \hat{w}_i'h(r_{\vecbeta^*,i}) \vecX_i^\top \vectheta_\eta \right|+\left|\sum_{i\in \mc{O}} \hat{w}_i'h(r_{\vecbeta^*,i}) \vecX_i^\top \vectheta_\eta \right|\nonumber \\
	&\leq \left|\sum_{i\in \mc{I}} \hat{w}_i'h(r_{\vecbeta^*,i}) \vecx_i^\top \vectheta_\eta \right|+\left|\sum_{i\in \mc{O}} \hat{w}_i'h(r_{\vecbeta^*,i}) \vecX_i^\top \vectheta_\eta \right|\nonumber\\
	&= \left|\sum_{i=1}^n \frac{1}{n}h(r_{\vecbeta^*,i}) \vecx_i^\top \vectheta_\eta -\sum_{i \in \mc{O}\cup (\mc{I}\cap I_< )} \frac{1}{n}h(r_{\vecbeta^*,i}) \vecx_i^\top \vectheta_\eta\right|+\left|\sum_{i\in \mc{O}} \hat{w}_i'h(r_{\vecbeta^*,i}) \vecX_i^\top \vectheta_\eta \right|\nonumber\\
	&\leq \left|\sum_{i=1}^n \frac{1}{n}h(r_{\vecbeta^*,i}) \vecx_i^\top \vectheta_\eta\right| +\left|\sum_{i\in \mc{O}} \hat{w}_i'h(r_{\vecbeta^*,i}) \vecX_i^\top \vectheta_\eta \right|+\left|\sum_{i \in \mc{O}\cup (\mc{I}\cap I_<)} \frac{1}{n}h(r_{\vecbeta^*,i}) \vecx_i^\top \vectheta_\eta\right|.
\end{align}
We note that $|h(\cdot)|\leq 1$ and from Lemma \ref{l:w2}, $|\mc{O} \cup \left(\mc{I} \cap I_{<}\right)|\leq (1+2c'_\varepsilon)o$.
Therefore, from Propositions \ref{p:main1}, \ref{p:main:out-un} and \ref{p:main:out2-un}, we have 
\begin{align}
	\left|\sum_{i=1}^n \frac{1}{n}h(r_{\vecbeta^*,i}) \vecx_i^\top \vectheta_\eta\right|
	&\leq c_3L\left(\rho r_{d,s}r_1+r_\delta r_\Sigma\right),\nonumber\\
	\left|\sum_{i \in \mc{O}} \frac{1}{n}h(r_{\vecbeta^*,i}) \vecx_i^\top \vectheta_\eta\right|
	&\leq c_8\sqrt{c_{\rm cut}}L\sqrt{\frac{o}{n}}\sqrt{\kappa_{\mr{u}}^2\left(sr_{d,s}+r_\delta \right)+\kappa_{\mr{u}}^2}r_2,\nonumber\\
	\left|\sum_{i \in \mc{O}\cup (\mc{I}\cap I_<)} \frac{1}{n}h(r_{\vecbeta^*,i}) \vecx_i^\top \vectheta_\eta\right|
	&\leq c_9L\sqrt{\frac{o}{n}}\sqrt{\kappa_{\mr{u}}^2\left(sr_{d,s}+r_\delta \right)+\kappa_{\mr{u}}^2}r_2,
\end{align}
and we see that 
\begin{align}
	\label{a:ine:un:step1-1-1}
	\left|\sum_{i=1}^n \hat{w}_i'h(r_{\vecbeta^*,i}) \vecX_i^\top \vectheta_\eta \right| 
	&\leq 3c'^2_{\max}  L\left(\rho r_{d,s}r_1+r_\delta r_\Sigma+\sqrt{\frac{o}{n}}\sqrt{\kappa_{\mr{u}}^2\left(sr_{d,s}+r_\delta \right)+\kappa_{\mr{u}}^2}r_2\right).
\end{align}
\end{proof}

\begin{proof}[Proof of Lemma \ref{3calcunknown}]
We have
\begin{align}
	&\sum_{i=1}^n \lambda_o\sqrt{n} \hat{w}_i'\left(-h(r_{\vecbeta^*+\vectheta_\eta,i}) +h(r_{\vecbeta^*,i})\right) \vecX_i^\top \vectheta_\eta \nonumber\\
	&= \sum_{i\in \mc{I}} \lambda_o\sqrt{n} \hat{w}_i'\left(-h(r_{\vecbeta^*+\vectheta_\eta,i}) +h(r_{\vecbeta^*,i})\right) \vecX_i^\top \vectheta_\eta +\sum_{i \in \mc{O}} \lambda_o\sqrt{n} \hat{w}_i'\left(-h(r_{\vecbeta^*+\vectheta_\eta,i}) +h(r_{\vecbeta^*,i})\right) \vecX_i^\top \vectheta_\eta\nonumber\\
	&= \sum_{i\in \mc{I}} \lambda_o\sqrt{n} \hat{w}_i'\left(	-h (\xi_{\lambda_o,i}-x_{\vectheta_\eta,i})+h (\xi_{\lambda_o,i})\right) \vecx_i^\top \vectheta_\eta +\sum_{i \in \mc{O}} \lambda_o\sqrt{n} \hat{w}_i'\left(-h(r_{\vecbeta^*+\vectheta_\eta,i}) +h(r_{\vecbeta^*,i})\right) \vecX_i^\top \vectheta_\eta\nonumber\\
	&= \left(\sum_{i=1}^n -\sum_{i\in \mc{O} \cup \left(\mc{I} \cap I_{<}\right)}\right) \frac{\lambda_o}{\sqrt{n}} \left(	-h (\xi_{\lambda_o,i}-x_{\vectheta_\eta,i})+h (\xi_{\lambda_o,i})\right) \vecx_i^\top\vectheta_\eta \nonumber \\
	&\quad +\sum_{i \in \mc{O}} \lambda_o\sqrt{n} \hat{w}_i'\left(-h(r_{\vecbeta^*+\vectheta_\eta,i}) +h(r_{\vecbeta^*,i})\right) \vecX_i^\top \vectheta_\eta\nonumber\\
	&\geq \sum_{i=1}^n \frac{\lambda_o}{\sqrt{n}} \left(	-h (\xi_{\lambda_o,i}-x_{\vectheta_\eta,i})+h (\xi_{\lambda_o,i})\right) \vecx_i^\top\vectheta_\eta-\left|\sum_{i\in \mc{O} \cup \left(\mc{I} \cap I_{<}\right)}\frac{\lambda_o}{\sqrt{n}} \left(	-h (\xi_{\lambda_o,i}-x_{\vectheta_\eta,i})+h (\xi_{\lambda_o,i})\right)\vecx_i^\top\vectheta_\eta\right|\nonumber \\
	&\quad -\left|\sum_{i\in \mc{O}} \lambda_o\sqrt{n} \hat{w}_i'\left(-h(r_{\vecbeta^*+\vectheta_\eta,i}) +h(r_{\vecbeta^*,i})\right) \vecX_i^\top \vectheta_\eta\right|.
\end{align}
We note that $|h(\cdot)|\leq 1$ and from Lemma \ref{l:w2}, $| \mc{O} \cup \left(\mc{I} \cap I_{<}\right)|\leq (1+2c_\varepsilon)o$. Therefore, from Proposition \ref{p:main:sc}, we have
\begin{align}
	\sum_{i=1}^n \frac{\lambda_o}{\sqrt{n}} \left(	-h (\xi_{\lambda_o,i}-x_{\vectheta_\eta,i})+h (\xi_{\lambda_o,i})\right) \vecx_i^\top\vectheta_\eta &\geq \frac{\|\Sigma^\frac{1}{2}\vecv\|_2^2}{3}r_2^2-c_{\max}L\lambda_o\sqrt{n}\left(\rho r_{d,s} r_1+r_\delta r_\Sigma\right)\nonumber\\
	\left|\sum_{i\in \mc{O}} \lambda_o\sqrt{n} \hat{w}_i'\left(-h(r_{\vecbeta^*+\vectheta_\eta,i}) +h(r_{\vecbeta^*,i})\right)\vecX_i^\top \vectheta_\eta\right|
	& \leq c_8\sqrt{c_{\rm cut}}L\sqrt{\frac{o}{n}}\sqrt{\kappa_{\mr{u}}^2\left(sr_{d,s}+r_\delta \right)+\kappa_{\mr{u}}^2}r_2,\nonumber \\
	\left|\sum_{i\in \mc{O} \cup \left(\mc{I} \cap I_{<}\right)}\frac{\lambda_o}{\sqrt{n}} \left(	-h (\xi_{\lambda_o,i}-x_{\vectheta_\eta,i})+h (\xi_{\lambda_o,i})\right) \vecx_i^\top\vectheta_\eta\right|
	& \leq c_8\sqrt{c_{\rm cut}}L\sqrt{\frac{o}{n}}\sqrt{\kappa_{\mr{u}}^2\left(sr_{d,s}+r_\delta \right)+\kappa_{\mr{u}}^2}r_2,.
\end{align}
Combining the arguments above, we see that 
\begin{align}
&\sum_{i=1}^n \lambda_o\sqrt{n} \hat{w}_i'\left(-h(r_{\vecbeta^*+\vectheta_\eta,i}) +h(r_{\vecbeta^*,i})\right) \vecX_i^\top \vectheta_\eta\nonumber \\
&\geq  \frac{\|\Sigma^\frac{1}{2}\vecv\|_2^2}{3}-c'^2_{\max} L\lambda_o\sqrt{n}\left(\rho r_{d,s} r_1+r_\delta r_\Sigma+\sqrt{\frac{o}{n}}\sqrt{\kappa_{\mr{u}}^2\left(sr_{d,s}+r_\delta \right)+\kappa_{\mr{u}}^2}r_2\right) ,
\end{align}
and  the proof is complete.
\end{proof}
\end{document}